\documentclass[12pt]{article} 
\usepackage[colorlinks]{hyperref}
\usepackage{graphicx}%
\usepackage{multirow}%
\usepackage{amsmath,amssymb,amsfonts}%
\usepackage{mathabx}
\usepackage{stmaryrd}

\usepackage{amsthm}%
\usepackage{mathrsfs}%
\usepackage[title]{appendix}%
\usepackage{xcolor}%
\usepackage{textcomp}%
\usepackage{manyfoot}%
\usepackage{booktabs}%
\usepackage{algorithm}%
\usepackage{algorithmicx}%
\usepackage{algpseudocode}%
\usepackage{listings}%
\usepackage[mathscr]{eucal}
\usepackage{bussproofs}
\usepackage{lineno}
\usepackage{quiver}
\usepackage{float}
\usepackage{tikz}
\usetikzlibrary{calc}
\usepackage{array}
\usepackage{multicol}
\usetikzlibrary{
    positioning, 
    shapes, 
    arrows.meta, 
    calc,        % Needed for coordinate calculations
    backgrounds, % Needed for drawing things explicitly behind everything else
    fit          % Needed for the bounding box
} 

\definecolor{myLightRed}{RGB}{255, 204, 204} % A very light red
\definecolor{myLightBlue}{RGB}{204, 229, 255} % A very light blue

\usepackage{authblk}

\theoremstyle{plain} 
\newtheorem{theorem}{Theorem}[section]

\newtheorem{lemma}[theorem]{Lemma}

\newtheorem{corollary}[theorem]{Corollary}
\theoremstyle{definition}
\newtheorem{definition}[theorem]{Definition}
\newtheorem{example}[theorem]{Example}
\newtheorem{remark}[theorem]{Remark}

\begin{document}

\title{Proof Complexity and Feasible Interpolation}

\author[]{Amirhossein Akbar Tabatabai}
\affil[]{Bernoulli Institute, University of Groningen}

\date{ }

\maketitle

\section{Introduction}

At first glance, propositional language might seem too simple to capture the interest of a working mathematician, as its expressive power, compared to first-order languages, is quite limited, mainly due to the lack of quantifiers. However, this initial impression is misleading. As we will see later in this chapter, propositional language is fully capable of describing any \emph{finite} structure, much like first-order languages are used to describe \emph{arbitrary} structures.

With this hidden expressive power in mind, propositional proofs should be regarded as more significant than they are often perceived, since they can be used to reason indirectly about any finite structure. But what do we mean by a propositional proof? In its most abstract form, a propositional proof is defined by the proof system in which it resides. For a given propositional (or sometimes modal) logic $L$, a \emph{proof system for $L$} is a polynomial-time predicate $P(\pi, \phi)$ such that $\phi \in L$ if and only if there exists $\pi$ satisfying $P(\pi, \phi)$, which we read as ``$\pi$ is a \emph{$P$-proof} of $\phi$''. $P$ is assumed to be polynomial-time computable to ensure that verifying whether a given object is a valid proof of a proposition is \emph{computationally easy}, while the equivalence condition encodes the standard requirements of \emph{soundness} and \emph{completeness} of $P$ with respect to $L$.  

\subsection{Propositional Incompleteness}

In its usual sense, proof systems cannot be incomplete, as they are sound and complete by definition. However, by shifting the focus from the ``non-existence of a proof'' to the ``non-existence of a \emph{short} proof,'' where short means polynomial in the size of the formula we want to prove, a new form of incompleteness emerges. Formally, given a logic $L$, we call a proof system for $L$ \emph{polynomially bounded} (or \emph{p-bounded}) if there exists a polynomial $p$ such that any $\phi \in L$ has a $P$-proof $\pi$ whose size is bounded by $p(|\phi|)$. In this way, a p-bounded proof system can be interpreted as the analogue of a \emph{complete first-order theory}, where the condition of the ``existence of a proof for any true statement'' is replaced by the ``existence of a short proof for any formula in $L$.''

Similar to the quest for first-order completeness, we may ask whether a p-bounded proof system exists for a given logic $L$. It is easy to see that a logic $L$ has a p-bounded proof system if and only if $L \in \mathbf{NP}$. Therefore, classical (resp.\ intuitionistic) propositional logic, $\mathsf{CPC}$ (resp.\ $\mathsf{IPC}$), has a p-bounded proof system if and only if $\mathbf{CoNP} = \mathbf{NP}$ (resp.\ $\mathbf{PSPACE} = \mathbf{NP}$). Since these equalities are widely believed to be false, it is similarly believed that no p-bounded proof system exists for $\mathsf{CPC}$ or $\mathsf{IPC}$. One may view this conjectural non-existence as the propositional analogue of \emph{Gödel's incompleteness theorem}, which asserts that any sufficiently strong and consistent theory is necessarily incomplete. As these forms of \emph{propositional incompleteness} are equivalent to long-standing open problems in computational complexity, their full proofs remain beyond our current reach. Nevertheless, this does not prevent us from demonstrating that some specific proof systems, such as the sequent calculi $\mathbf{LK}$ for $\mathsf{CPC}$ and $\mathbf{LJ}$ for $\mathsf{IPC}$, are not p-bounded. Fortunately, a general technique exists for proving such results, known as \emph{feasible interpolation}. This chapter is dedicated to presenting this method and its applications.

\subsection{Feasible Interpolation}

Roughly speaking, a proof system $P$ for $\mathsf{CPC}$ has feasible interpolation if we can \emph{easily} compute a Craig interpolant for an implicational tautology from any of its $P$-proofs. Now, subject to certain conjectures in computational complexity, we claim that if $P$ has feasible interpolation, then it cannot be p-bounded. To prove that, first, by invoking some conjectures (or occasionally results) from computational complexity, we construct an implicational tautology with hard-to-compute Craig interpolants. Next, since $P$ has feasible interpolation, if we have access to a $P$-proof of the implication, obtaining the hard-to-compute interpolant becomes easy. Therefore, the $P$-proofs must be too long, implying that $P$ is not p-bounded. Note that feasible interpolation effectively reduces a question in proof complexity to one in computational complexity, where known hardness assumptions or results on \emph{computation} can be leveraged to establish the difficulty of \emph{proving} certain theorems.

For classical logic, feasible interpolation as a general method was introduced by Krajíček \cite{krajivcek1994lower,krajivcek1997interpolation} and has since been successfully applied to several proof systems for $\mathsf{CPC}$, including the cut-free sequent calculus, the resolution refutation system \cite{krajivcek1997interpolation}, and cutting planes \cite{pudlak1997lower}. However, if a proof system is too strong, it may yield short proofs for implications whose interpolants are believed (usually by cryptographic assumptions) to be hard-to-compute. Thus, by reversing the earlier argument, we can conclude that such systems cannot have feasible interpolation. 
For instance, $\mathbf{LK}$ fails to have feasible interpolation unless the Diffie-Hellman key exchange protocol is insecure against polynomial-size circuits \cite{krajivcek1995someII,krajiivcek1998some,bonet1997no,bonet2004non}. Since feasible interpolation is the main technique we rely on, this limitation might initially seem like a setback. However, the failure of feasible interpolation leads to an interesting consequence: under a natural technical condition, it implies that the proof system is \emph{non-automatable}, meaning that it is impossible to find a proof of a tautology in time polynomially bounded by the sum of the size of its shortest proof and the size of the formula itself. As a result, we can show that $\mathbf{LK}$ is not automatable, unless the Diffie-Hellman key exchange protocol is insecure against polynomial-size circuits.
For a comprehensive discussion on feasible interpolation and its connections to computational complexity and cryptography, see \cite{krajivcek2019proof}.

For non-classical logics, a similar role to Craig interpolants is played by a generalization of the disjunction property, which we refer to as \emph{disjunctive interpolants}. This notion was first developed by Buss and Pudlák \cite{buss2001computational}, following the feasible version of the disjunction property proved by Buss and Mints \cite{bussMints}. The concept of feasible disjunctive interpolation was later generalized by Hrubeš \cite{hrubevs2007lower, hrubevs2007lowerII, hrubevs2009lengths}, who used it to show that $\mathbf{LJ}$ and any Hilbert-style proof system for intuitionistic logic, or for modal logics below $\mathsf{S4}$ or $\mathsf{GL}$, are not p-bounded. These results were subsequently extended to superintuitionistic logics and modal logics of infinite branching by Jeřábek \cite{jevrabek2009substitution} and by Jalali \cite{jalali2021proof} to any logic between Full Lambek logic $\mathsf{FL}$ and a superintuitionistic logic of infinite branching. For additional results on non-classical proof complexity, see \cite{jevrabek2017proof, jevrabek2023proof, jevrabek2025simplified}, and for an insightful survey, refer to \cite{beyersdorff2010proof}.

\subsection{Outline of the Chapter}

The structure of this chapter is as follows. After introducing the preliminaries in Section \ref{Sec: Preliminaries}, Section \ref{Sec: Propositional World} presents tools from descriptive complexity that enable the translation of first-order formulas over finite structures into propositional formulas. We also utilize the propositional language to describe computations over finite domains through circuits and discuss relevant topics in cryptography and disjoint $\mathbf{NP}$ pairs. In Section \ref{Sec: Proof Systems}, we introduce the fundamentals of proof systems, provide examples, and compare them in terms of their relative strength. Section \ref{Section: Interpolation} uses hard disjoint $\mathbf{NP}$ pairs to construct implicational classical tautologies with hard-to-compute Craig interpolants. It also defines disjunctive interpolants and constructs hard-to-compute instances for them. Section \ref{Sec: FI, Classical} introduces the concept of feasible interpolation in the classical setting, including its monotone version and automatability. We show that the cut-free sequent calculus and the resolution system have (monotone) feasible interpolation and are therefore not p-bounded. Additionally, assuming certain cryptographic hardness conjectures, we demonstrate that some strong proof systems lack feasible interpolation and are thus not automatable. Finally, in Section \ref{Sec: FI, Non-Classical}, we extend these ideas to superintuitionistic and modal sequent-style and Hilbert-style proof systems.\\

\noindent \textbf{A Word to the Reader:}
%\subsubsection*{A Word to the Reader}
The literature on proof complexity often assumes fluency in computational complexity and bounded theories of arithmetic, which can be challenging for newcomers. To make this material more accessible, we present it with the assumption of only a basic familiarity with propositional, modal, and first-order logic, as well as a basic understanding of key concepts in computational complexity, such as the definitions of the classes $\mathbf{NP}$ and $\mathbf{PSPACE}$. Any additional concepts will be introduced and explained as needed.\\

\noindent \textbf{Acknowledgment:}
This work was supported by the Austrian Science Fund (FWF) project P33548, and the Dutch Research Council (NWO) project OCENW.M.22.258.

\section{Preliminaries} \label{Sec: Preliminaries}

By $\mathcal{L}_b$ we mean the propositional language $\{\top, \bot, \wedge, \vee, \neg\}$. The language $\mathcal{L}^u_b = \{\top, \bot, \bigwedge, \bigvee, \neg\}$ is defined similarly to $\mathcal{L}_b$, with the difference that it allows conjunctions and disjunctions of arbitrary finite arity rather than just binary ones. The language $\mathcal{L}_p$ is defined as $\{\top, \bot, \wedge, \vee, \to\}$, and $\mathcal{L}_{\Box} = \mathcal{L}_p \cup \{\Box\}$. 
For any $\mathfrak{L} \in \{\mathcal{L}_b, \mathcal{L}^u_b, \mathcal{L}_p, \mathcal{L}_{\Box}\}$, we define \emph{$\mathfrak{L}$-formulas}, or simply \emph{formulas in $\mathfrak{L}$}, in the usual way. When the context is clear, we may omit specifying the language and simply refer to them as formulas. In the languages $\mathcal{L}_p$ and $\mathcal{L}_{\Box}$, by $\neg \phi$ and $\boxdot \phi$ we mean $\phi \to \bot$ and $(\Box \phi \wedge \phi)$, respectively. For any formula $\phi$, we denote by $V(\phi)$ the set of atomic formulas occurring in $\phi$. 
Over $\mathcal{L}_b$ and $\mathcal{L}_b^u$, a \emph{literal} is either a variable $p$ or its negation $\neg p$, and a \emph{clause} is a finite sequence of literals, interpreted as their disjunction. More generally, a formula in $\mathcal{L}_b$ is said to be in \emph{negation normal form} if the only negations used in the formula appear on atomic formulas, $\top$, or $\bot$. It is clear that any formula in the languages $\mathcal{L}_b$, $\mathcal{L}_b^u$, and $\mathcal{L}_p$ is classically equivalent to one in negation normal form, using De Morgan’s laws. For a set $P$ of atomic formulas, a formula is called \emph{monotone} in $P$ if $\neg p$ does not occur in its negation normal form, for any $p \in P$. A formula in any of the four languages is called \emph{monotone} if it contains no negation or implication.
For any $\mathcal{L}^u_b$-formula $\phi$, define the depth $dp(\phi)$ as follows: $dp(p)=dp(\top)=dp(\bot)=0$, $dp(\neg \phi) = dp(\phi)$, $dp(\bigwedge_i \phi_i)=1+\max_i dp(\phi_i)$ and $dp(\bigvee_i \phi_i)=1+\max_i dp(\phi_i)$.

Let $\mathfrak{L} \in \{\mathcal{L}_b, \mathcal{L}_b^u, \mathcal{L}_p\}$ and let $P = \{p_1, \ldots, p_m\}$ be a set of atomic formulas. By a \emph{Boolean assignment} for $P$, we mean a tuple $\bar{a} \in \{0, 1\}^m$ that assigns $a_i$ to $p_i$ for each $1 \leq i \leq m$. When $m$ is implicit, we simply write $\bar{a} \in \{0, 1\}$. An $\mathfrak{L}$-formula is called \emph{valid} if it evaluates to $1$ under all Boolean assignments to its atomic formulas. A formula (resp. a set of clauses) is called \emph{satisfiable} if there exists a Boolean assignment for all involved atomic formulas such that the formula (resp. the disjunction of any of the clauses) evaluates to $1$.
We define $\mathsf{CPC}$ as the set of all valid $\mathcal{L}_p$-formulas. By a slight abuse of notation, we also use $\mathsf{CPC}$ to refer to the set of valid $\mathfrak{L}$-formulas for $\mathfrak{L} \in \{\mathcal{L}_b, \mathcal{L}_b^u\}$. 
The logic $\mathsf{CPC}$ has the \emph{Craig Interpolation Property}, meaning that for any valid implication $\phi \to \psi$, there exists a formula $I$ such that $V(I) \subseteq V(\phi) \cap V(\psi)$ and both $(\phi \to I)$ and $(I \to \psi)$ are valid.
Sometimes, the Craig interpolation property is stated in terms of unsatisfiable clauses: that is, for any unsatisfiable set $\{C_i\}_i \cup \{D_j\}_j$ of clauses, there exists a formula $I$ using only the variables common to $\{C_i\}_i$ and $\{D_j\}_j$ such that both $\{C_i\}_i \cup \{I\}$ and $\{D_j\}_j \cup \{\neg I\}$ are unsatisfiable.

For any $\mathfrak{L} \in \{\mathcal{L}_b, \mathcal{L}_b^u, \mathcal{L}_p, \mathcal{L}_{\Box}\}$, a \emph{sequent} over $\mathfrak{L}$ is an expression of the form $\Gamma \Rightarrow \Delta$, where $\Gamma$ and $\Delta$ are finite sequences of $\mathfrak{L}$-formulas. A sequent is called \emph{single-conclusion} if $\Delta$ contains at most one $\mathfrak{L}$-formula. We say that a sequent $\Gamma \Rightarrow \Delta$ over $\mathfrak{L} \in \{\mathcal{L}_b, \mathcal{L}_b^u, \mathcal{L}_p\}$ is \emph{valid} if the formula $\bigwedge \Gamma \to \bigvee \Delta$ is valid.
The sequent calculus $\mathbf{LK}$ over $\mathcal{L}_p$ is defined as in Figure~\ref{GSTN}.
\begin{figure}[t]
\begin{center}
 \begin{tabular}{c c c}
 \AxiomC{}
 \UnaryInfC{$\phi \Rightarrow \phi$}
 \DisplayProof 
 &
 \small \AxiomC{}
\small \UnaryInfC{$ \bot \Rightarrow $}
 \DisplayProof 
&
\small  \AxiomC{}
\small \UnaryInfC{$ \Rightarrow \top$}
 \DisplayProof\\[3ex]
\end{tabular}

 \begin{tabular}{cccccc}
\small \AxiomC{$\Gamma, \phi, \psi, \Pi \Rightarrow \Delta$}
\small \RightLabel{\footnotesize $Le$}
\small \UnaryInfC{$\Gamma, \psi, \phi, \Pi \Rightarrow \Delta$}
 \DisplayProof 
 &
\small \AxiomC{$\Gamma \Rightarrow \Delta, \phi, \psi, \Lambda$}
\small \RightLabel{\footnotesize $Re$}
\small \UnaryInfC{$\Gamma \Rightarrow \Delta, \psi, \phi, \Lambda $}
 \DisplayProof 
&
\small \AxiomC{$\Gamma \Rightarrow \phi, \Delta$}
\small \AxiomC{$\Gamma, \phi \Rightarrow \Delta$}
\small \RightLabel{\footnotesize $cut$}
 \BinaryInfC{$\Gamma \Rightarrow \Delta$}
 \DisplayProof
 \\[3ex]
 \end{tabular}
 
 \begin{tabular}{cccccc}
\small \AxiomC{$\Gamma \Rightarrow \Delta$}
\small \RightLabel{\footnotesize $Lw$}
\small \UnaryInfC{$\Gamma, \phi \Rightarrow \Delta$}
 \DisplayProof 
 &
\small \AxiomC{$\Gamma \Rightarrow \Delta$}
\small \RightLabel{\footnotesize $Rw$}
\small \UnaryInfC{$\Gamma \Rightarrow \phi, \Delta $}
 \DisplayProof 
 &
 \small \AxiomC{$\Gamma, \phi, \phi \Rightarrow \Delta$}
 \small \RightLabel{\footnotesize $Lc$}
\small \UnaryInfC{$\Gamma, \phi \Rightarrow \Delta$}
 \DisplayProof 
 &
\small \AxiomC{$\Gamma \Rightarrow, \phi, \phi, \Delta$}
\small \RightLabel{\footnotesize $Rc$}
\small \UnaryInfC{$\Gamma \Rightarrow \phi, \Delta $}
 \DisplayProof 
\\[3ex]
 \end{tabular}
 \begin{tabular}{ccc}
\small \AxiomC{$\Gamma, \phi \Rightarrow \Delta$}
\small \RightLabel{\footnotesize $L\wedge_1$}
\small \UnaryInfC{$\Gamma, \phi \wedge \psi \Rightarrow \Delta$}
 \DisplayProof 
&
\small \AxiomC{$\Gamma, \psi \Rightarrow \Delta$}
\small \RightLabel{\footnotesize $L\wedge_2$}
\small \UnaryInfC{$\Gamma, \phi \wedge \psi \Rightarrow \Delta$}
 \DisplayProof
&
\small \AxiomC{$\Gamma \Rightarrow \phi, \Delta$}
\small \AxiomC{$\Gamma \Rightarrow \psi, \Delta$}
\small \RightLabel{\footnotesize $R\wedge$}
 \BinaryInfC{$\Gamma \Rightarrow \phi \wedge \psi, \Delta$}
 \DisplayProof
  \\[3ex]
   \end{tabular}
 \begin{tabular}{ccc}
\small \AxiomC{$\Gamma, \phi \Rightarrow \Delta$}
\small \AxiomC{$\Gamma, \psi \Rightarrow \Delta$}
\small \RightLabel{\footnotesize $L\vee$}
 \BinaryInfC{$\Gamma, \phi \vee \psi \Rightarrow \Delta$}
 \DisplayProof
 &
\small \AxiomC{$\Gamma \Rightarrow \phi, \Delta$}
\small \RightLabel{\footnotesize $R\vee_1$}
\small \UnaryInfC{$\Gamma \Rightarrow \phi \vee \psi, \Delta$}
 \DisplayProof
&
\small  \AxiomC{$\Gamma \Rightarrow \psi, \Delta$}
\small \RightLabel{\footnotesize $R\vee_2$}
\small \UnaryInfC{$\Gamma \Rightarrow \phi \vee \psi, \Delta$}
 \DisplayProof
 \\[3ex]
\end{tabular}

 \begin{tabular}{ccc}
\small \AxiomC{$\Gamma \Rightarrow \phi, \Delta$}
\small \AxiomC{$\Gamma, \psi \Rightarrow \Delta$}
\small \RightLabel{\footnotesize $L\!\to$}
 \BinaryInfC{$\Gamma, \phi \to \psi \Rightarrow \Delta$}
 \DisplayProof
 &
\small \AxiomC{$\Gamma, \phi \Rightarrow \psi, \Delta$}
\small \RightLabel{\footnotesize $R\!\to$}
\small \UnaryInfC{$\Gamma \Rightarrow \phi \to \psi, \Delta$}
 \DisplayProof
\end{tabular}
\caption{The sequent calculus $\mathbf{LK}$}
\label{GSTN}
\end{center}
\end{figure}
Define $\mathbf{LJ}$ over $\mathcal{L}_p$ similarly to $\mathbf{LK}$, restricting it to operate only on single-conclusion sequents. Over $\mathcal{L}_{\Box}$, consider the following list of modal rules:
\begin{center}
\begin{tabular}{cccccc}
\small \AxiomC{$\Gamma \Rightarrow \phi$}
\small \RightLabel{\footnotesize $K$}
 \UnaryInfC{$\Box \Gamma \Rightarrow \Box \phi$}
 \DisplayProof
 &
\small \AxiomC{$\Gamma \Rightarrow \Delta$}
\small \RightLabel{\footnotesize $D$}
\small \UnaryInfC{$\Box \Gamma \Rightarrow \Box \Delta$}
 \DisplayProof
&
\small  \AxiomC{$\Box \Gamma, \Gamma \Rightarrow \phi$}
\small \RightLabel{\footnotesize $K4$}
\small \UnaryInfC{$\Box \Gamma \Rightarrow \Box \phi$}
 \DisplayProof
 &
\small  \AxiomC{$\Box \Gamma, \Gamma \Rightarrow \Delta$}
\small \RightLabel{\footnotesize $KD4$}
\small \UnaryInfC{$\Box \Gamma \Rightarrow \Box \Delta$}
 \DisplayProof
 \\[3ex]
\end{tabular}
\begin{tabular}{cccccc}
\small  \AxiomC{$\Gamma, \phi \Rightarrow \Delta$}
\small \RightLabel{\footnotesize $LS4$}
\small \UnaryInfC{$\Gamma, \Box \phi \Rightarrow \Delta$}
 \DisplayProof
 &
\small  \AxiomC{$\Box \Gamma \Rightarrow \phi$}
\small \RightLabel{\footnotesize $RS4$}
\small \UnaryInfC{$\Box \Gamma \Rightarrow \Box \phi$}
 \DisplayProof
&
\small  \AxiomC{$\Box \Gamma, \Gamma, \Box \phi \Rightarrow \phi$}
\small \RightLabel{\footnotesize $GL$}
\small \UnaryInfC{$\Box \Gamma \Rightarrow \Box \phi$}
 \DisplayProof
 \\[3ex]
\end{tabular}
\end{center}
where in rules $(D)$ and $(KD4)$, the sequence $\Delta$ contains at most one formula. Now, consider the sequent calculi listed in Table~\ref{tab:sequent_calculi}.
% Start the table floating environment
\begin{table}[htpb] % h=here, t=top, b=bottom, p=page (placement suggestion)
  \centering % Center the contents (tabular) within the table environment
  \begin{tabular}{|c|c|c|} % Defines 3 columns with vertical lines
    \hline % Top horizontal line
    % Header Row
    \textbf{Calculus Name} & \textbf{Definition} & \textbf{Logic Name} \\
    \hline % Line below header
    % Data Rows
    $\mathbf{K}$  & $\mathbf{LK}+\mathrm{K}$        & $\mathsf{K}$   \\
    \hline % Line after row
    $\mathbf{D}$  & $\mathbf{LK}+\mathrm{D}$        & $\mathsf{D}$   \\
    \hline % Line after row
    $\mathbf{KT}$ & $\mathbf{LK}+\mathrm{K}+\mathrm{LS4}$ & $\mathsf{KT}$  \\
    \hline % Line after row
    $\mathbf{K4}$ & $\mathbf{LK}+\mathrm{K4}$       & $\mathsf{K4}$  \\
    \hline % Line after row
    $\mathbf{KD4}$ & $\mathbf{LK}+\mathrm{KD4}$      & $\mathsf{KD4}$ \\
    \hline % Line after row
    $\mathbf{S4}$ & $\mathbf{LK}+\mathrm{LS4}+\mathrm{RS4}$ & $\mathsf{S4}$  \\
    \hline % Line after row
    $\mathbf{GL}$ & $\mathbf{LK}+\mathrm{GL}$       & $\mathsf{GL}$ \\
    \hline % Bottom horizontal line
  \end{tabular}
  \vspace{10pt}
  \caption{Some sequent calculi and their logics.}
  \label{tab:sequent_calculi}
\end{table}
In each of the sequent calculi introduced above, the cut rule can be eliminated without affecting the set of provable sequents. We denote the cut-free version of a calculus by appending a superscript minus sign to its name. For example, the system $\mathbf{LK}$ without the cut rule is denoted by $\mathbf{LK}^-$.
For any of the sequent calculi $G$ defined above, by its \emph{logic} we mean the set of all formulas $\phi$ such that $G \vdash \, \Rightarrow \phi$. We usually denote sequent calculi with boldface letters and their corresponding logics with the same letters in sans-serif font. For instance, the logic $\mathsf{S4}$ is the logic of the calculus $\mathbf{S4}$. The only exceptions are the logics $\mathsf{IPC}$ and $\mathsf{CPC}$, which are the logics of $\mathbf{LJ}$ and $\mathbf{LK}$, respectively.

There is a connection between $\mathsf{IPC}$ and $\mathsf{CPC}$ that we will use in this chapter, called \emph{Glivenko's theorem}. It states that for any sequence $\Gamma \cup \{\phi\}$ of formulas, if $\Gamma \vdash_{\mathsf{CPC}} \phi$, then $\neg \neg \Gamma \vdash_{\mathsf{IPC}} \neg \neg \phi$.

A \emph{superintuitionistic} or \emph{si logic} is a set $L$ of $\mathcal{L}_p$-formulas that is closed under substitution and modus ponens, such that $\mathsf{IPC} \subseteq L \subseteq \mathsf{CPC}$. For any set $\Phi$ of $\mathcal{L}_p$-formulas, the logic $\mathsf{IPC} + \Phi$ refers to the smallest si logic containing $\Phi$. For any si logic $L$ and any set $\Gamma \cup \{\phi\}$ of $\mathcal{L}_p$-formulas, we write $\Gamma \vdash_L \phi$ if $(\, \Rightarrow \phi)$ is provable in $\mathbf{LJ}$ from the sequents $\{\, \Rightarrow \gamma\}_{\gamma \in \Gamma} \cup \{\,\Rightarrow \psi \mid \psi \in L \}$.
A \emph{normal modal logic}, or simply a \emph{modal logic}, is defined as a set $L \supseteq \mathsf{K}$ of $\mathcal{L}_{\Box}$-formulas closed under substitution, modus ponens, and the necessitation rule $\phi/\Box \phi$. It is clear that any logic defined in Table \ref{tab:sequent_calculi} is a modal logic. For any set $\Phi$ of $\mathcal{L}_{\Box}$-formulas, the logic $\mathsf{K} + \Phi$ refers to the smallest modal logic containing $L \cup \Phi$. For any modal logic $L$ and any set $\Gamma \cup \{\phi\}$ of $\mathcal{L}_{\Box}$-formulas, we write $\Gamma \vdash_L \phi$ if $(\, \Rightarrow \phi)$ is provable in $\mathbf{K}$ from the sequents $\{\, \Rightarrow \gamma\}_{\gamma \in \Gamma} \cup \{\, \Rightarrow \psi \mid \psi \in L\}$.
Consider the modal logics $\mathsf{K} + \{\Box p \leftrightarrow p\}$ and $\mathsf{K} + \{\Box p\}$. Both of these modal logics are conservative over classical logic. Moreover, by Makinson's theorem \cite{makinson1971some}, any consistent modal logic is a subset of one of these two logics. There are also two translation functions, $f$ and $c$, from $\mathcal{L}_{\Box}$ to $\mathcal{L}_p$ that preserve every constant, atom, and propositional connectives and defined on the modality by: $(\Box \phi)^f = \phi^f$ and $(\Box \phi)^c = \top$. The function $f$ \emph{forgets} boxes, while $c$ \emph{collapses} them into $\top$, corresponding to the use of the axioms $\Box p \leftrightarrow p$ and $\Box p$, respectively.

Define $\mathsf{T}_k$ as the si logic
\[
\mathsf{IPC} + 
\{\bigwedge_{i=0}^{k} 
\left( 
(p_i \rightarrow \bigvee_{j \neq i} p_j)
\rightarrow 
\bigvee_{j} p_j
\right)
\rightarrow \bigvee_{i} p_i\}.
\]
For the reader familiar with Kripke frames, the logic $\mathsf{T}_k$ is the set of $\mathcal{L}_p$-formulas valid in all finite Kripke frames with branching at most $k$.
Similarly, in the modal case, define $\mathsf{K4BB}_k$ as
\[
\mathsf{K} + \{\Box p \to \Box \Box p, \Box \left( \bigvee_{i \leq k} \Box \left( \boxdot p_i \rightarrow \bigvee_{j \neq i} p_j \right) \rightarrow \bigvee_{i \leq k} \boxdot p_i \right) \rightarrow \bigvee_{i \leq k} \Box \bigvee_{j \neq i} p_j\}.
\]
Again, $\mathsf{K4BB}_k$ is the set of $\mathcal{L}_{\Box}$-formulas that are valid in all finite transitive Kripke frames with branching at most $k$.
We say that a si (resp. modal) logic $L$ has \emph{branching $k$} if $L \supseteq \mathsf{T}_k$ (resp. $L \supseteq \mathsf{K4BB}_k$). A logic is called of \emph{infinite branching} if it does not have branching $k$ for any $k \in \mathbb{N}$.
In the case of si logics, Jeřábek \cite{jevrabek2009substitution} provided an interesting characterization of the logics of infinite branching. First, recall that the $\mathcal{L}_p$-formulas $\mathrm{BD}_n$ are defined by: $\mathrm{BD}_0 := \bot$ and $\mathrm{BD}_{n+1} := p_n \vee (p_n \to \mathrm{BD}_n)$. Then, consider the logics $\mathsf{IPC}+\{\mathrm{BD_2}\}$ and $\mathsf{IPC}+\{\mathrm{BD_3}, \neg p \vee \neg \neg p\}$ and denote them by $\mathsf{BD_2}$ and $\mathsf{BD_3}+\mathsf{KC}$, respectively. 
Jeřábek's characterization states that a si logic $L$ is of infinite branching if and only if $\mathsf{L} \subseteq \mathsf{BD_2}$ or $\mathsf{L} \subseteq \mathsf{BD_3}+\mathsf{KC}$.

Finally, some points about binary strings and computation. By $\{0, 1\}^*$, we mean the set of all binary strings. For any $w \in \{0, 1\}^*$, let $|w|$ denote the length of $w$, and for $i \leq |w|$, let $w_i$ represent the $i$-th bit of $w$. By a \emph{sequence of short binary strings} $\{w_n\}_{n=n_0}^\infty$, we mean a sequence of binary strings such that there exists a polynomial $p$ with $|w_n| \leq p(n)$ for any $n \geq n_0$. We can encode any type of finitary object, such as numbers, formulas, and proofs, by binary strings. Therefore, we can apply similar definitions to these objects via their representations.
By a \emph{language} $L$, we mean any subset of $\{0, 1\}^*$, and $L_n$ is defined as the set $\{w \in L \mid |w| = n\}$.
The class $\mathbf{FP}$ consists of all functions $f: \{0, 1\}^* \to \{0, 1\}^*$ that can be computed by a deterministic Turing machine in polynomial time. We define the classes $\mathbf{P}$ and $\mathbf{PSPACE}$ as the set of languages decidable by a deterministic Turing machine in polynomial time and space, respectively. Similarly, $\mathbf{NP}$ refers to the set of languages decidable by a non-deterministic Turing machine in polynomial time. The class $\mathbf{CoNP}$ is defined as the set of languages whose complement belongs to $\mathbf{NP}$.
It is known that a language $L$ is in $\mathbf{NP}$ if and only if there exists a polynomial time decidable predicate $P \subseteq \{0, 1\}^* \times \{0, 1\}^*$ and a polynomial $p$ such that $w \in L$ if and only if there exists $u \in \{0, 1\}^*$ with $|u| \leq p(|w|)$ and $P(u, w)$. 
It is widely believed that $\mathbf{NP} \neq \mathbf{CoNP}$, and hence $\mathbf{NP} \neq \mathbf{P}$. Moreover, it is believed that $\mathbf{NP} \subsetneq \mathbf{PSPACE}$.

\section{The Propositional World} \label{Sec: Propositional World}

As promised in the introduction, in this section we show that the propositional language is sufficiently powerful to address various matters related to \emph{finite} structures. We will accomplish this in two subsections. First, we explore the \emph{expressive} power of the propositional language. Then, in the second subsection, we use this language to define circuits and examine the \emph{computational} power of the language in computing functions between finite sets.

\subsection{Descriptive Power} \label{Subsection: descriptive power}

Fix a multi-sorted first-order language $\mathcal{L}$ with equality and without any function symbols and choose an interpretation $I$ that assigns to each sort $\sigma$ a finite, non-empty set $I(\sigma)$. We then expand the language $\mathcal{L}$ to a new language $\mathcal{L}(I)$ by adding a constant symbol for each element in these sets. Note that we leave all relational symbols (except for equality) unspecified.
Given an \emph{$\mathcal{L}(I)$-sentence} $\phi$, we define a propositional formula $\llbracket \phi \rrbracket$ that expresses the same content as $\phi$. The key idea is this: for each relational symbol (except equality) of the form $R(x_1^{\sigma_1}, \ldots, x_k^{\sigma_k})$, we introduce a propositional variable $p_{a_1a_2\ldots a_k}$ to represent the truth value of $R(a_1, \ldots, a_k)$ for every $a_i \in I(\sigma_i)$. Equality statements $a = b$ are interpreted by evaluating whether $a$ and $b$ denote the same element. Propositional connectives are mapped to themselves, and quantifiers are translated into finite conjunctions and disjunctions.
More formally, the translation $\llbracket{-}\rrbracket$ from $\mathcal{L}(I)$-sentences to propositional formulas is defined recursively as follows:
\vspace{3pt}
\begin{itemize}
  \item[$\bullet$] $\llbracket c \rrbracket = c$ for any $c \in \{\top, \bot\}$;
  \item[$\bullet$] $\llbracket a = b \rrbracket = \top$ if $a = b$, and $\llbracket a = b \rrbracket = \bot$ otherwise;
  \item[$\bullet$] $\llbracket R(a_1, \ldots, a_k) \rrbracket = p_{a_1a_2\ldots a_k}$ for any $a_i \in I(\sigma_i)$;
  \item[$\bullet$] $\llbracket \phi \circ \psi \rrbracket = \llbracket \phi \rrbracket \circ \llbracket \psi \rrbracket$ for any logical connective $\circ \in \{\wedge, \vee, \to\}$;
  \item[$\bullet$] $\llbracket \forall x^{\sigma} \, \phi(x) \rrbracket = \bigwedge_{a \in I(\sigma)} \llbracket \phi(a) \rrbracket$ and $\llbracket \exists x^{\sigma} \, \phi(x) \rrbracket = \bigvee_{a \in I(\sigma)} \llbracket \phi(a) \rrbracket$.
\end{itemize}
\vspace{4pt}
For any extension of $I$ to a full interpretation $I'$ (one that also interprets the relational symbols), we can define a Boolean assignment that maps each propositional variable $p_{a_1a_2\ldots a_k}$ to true if and only if $R(a_1, \ldots, a_k)$ holds under $I'$, for all $a_i \in I(\sigma_i)$. Conversely, given any Boolean assignment to the atoms $p_{a_1a_2\ldots a_k}$, we can construct a corresponding extension of $I$ to a full interpretation.
It is straightforward to verify that under this correspondence, the $\mathcal{L}(I)$-sentence $\phi$ holds under the interpretation $I'$ if and only if the propositional formula $\llbracket \phi \rrbracket$ evaluates to true under the associated Boolean assignment. In this sense, $\phi$ and $\llbracket \phi \rrbracket$ express the same content.

Now, we use the propositional translation $\llbracket - \rrbracket$ to introduce the propositional counterparts of certain first-order sentences that will play a role later in the chapter. These examples will also help illustrate how the translation operates in concrete cases.

\begin{example}\label{FunctionExample}
Let $\mathcal{L}$ be a two-sorted language with sorts $\sigma$ and $\tau$, containing no function symbols and a single binary relation symbol $R(x^{\sigma}, y^{\tau})$. Define $I(\sigma)=\{1, \ldots, m\}$ and $I(\tau)=\{1, \ldots, n\}$.
The formula  
\[
\phi=\forall x^{\sigma} \forall y_1^{\tau} y_2^{\tau} \left(R(x^{\sigma}, y_1^{\tau}) \wedge R(x^{\sigma}, y_2^{\tau}) \to y^{\tau}_1 = y^{\tau}_2\right)
\]
expresses that the relation $R$ is functional.
To transform this sentence into a propositional formula, we introduce propositional atoms $p_{ij}$ to represent the truth of $R(i, j)$, for all $i \leq m$ and $j \leq n$. Then, we have
\[
\llbracket \phi \rrbracket=\bigwedge_{i \leq m} \bigwedge_{j_1 \leq n} \bigwedge_{j_2 \leq n}((p_{ij_1} \wedge p_{ij_2}) \to \llbracket j_1=j_2\rrbracket).
\]
When $j_1=j_2$, the formula $\llbracket \phi \rrbracket$ is trivially true as $\llbracket j_1=j_2\rrbracket=\top$. Hence, it is enough to restrict to $j_1 \neq j_2$. As $\llbracket j_1=j_2\rrbracket=\bot$ in this case, $\llbracket \phi \rrbracket$ is equivalent to  
\[
\bigwedge_{i \leq m} \bigwedge_{j_1 \neq j_2 \leq n} (\neg p_{ij_1} \vee \neg p_{ij_2}).
\] 
For the second example, consider
\[
\psi=\forall x_1^{\sigma} x_2^{\sigma} \forall y^{\tau} \left(R(x_1^{\sigma}, y^{\tau}) \wedge R(x^{\sigma}_2, y^{\tau}) \to x^{\sigma}_1 = x^{\sigma}_2\right)
\]
stating that $R$ is an injective relation. Using a similar argument as above, we have:  
\[
\llbracket \psi \rrbracket=\bigwedge_{i_1 \neq i_2 \leq m} \bigwedge_{j \leq n} (\neg p_{i_1j} \vee \neg p_{i_2j}).
\]
As the final example, consider the sentence $\theta=\forall x^{\sigma} \exists y^{\tau} R(x^{\sigma}, y^{\tau})$ stating that $R$ is a total relation. Therefore, we have  
$
\llbracket \theta \rrbracket=
\bigwedge_{i \leq m} \bigvee_{j \leq n} p_{ij}.
$
\end{example}

\begin{example}\label{ExampleClique}
Let $\mathcal{L}$ be the first-order language with one sort $\sigma$, no function symbols, and the binary relation $E(x^{\sigma}, y^{\sigma})$. Define $I(\sigma)=\{1, \ldots, n\}$, and interpret any element of $I(\sigma)$ as a vertex of a graph with $n$ vertices. Now, if we interpret $E(x^{\sigma}, y^{\sigma})$ as ``there is an edge from $x$ to $y$,'' then any extension of $I$ to a full interpretation corresponds to a directed graph on $n$ vertices.
First, consider the sentence  
\[
\phi=\forall x^{\sigma} \neg E(x^{\sigma}, x^{\sigma}) \wedge \forall x_1^{\sigma} x_2^{\sigma} (E(x^{\sigma}, y^{\sigma}) \to E(y^{\sigma}, x^{\sigma})),
\]
which states that the graph is symmetric and has no loops; i.e., it is essentially a simple graph. 
To translate $\phi$ into a propositional formula, we introduce propositional atoms $s_{ij}$ to encode whether $E(i, j)$ holds, for all $i, j \leq n$. Then,  
\[
\llbracket \phi \rrbracket = (\bigwedge_{i \leq n} \neg s_{ii}) \wedge \bigwedge_{i_1, i_2 \leq n} (s_{i_1i_2} \to s_{i_2i_1}).
\]
To simplify the propositional translation, we use the following trick: Instead of $s_{ij}$'s, use the atoms $p_{ij}$ for any pair $\{i, j\} \subseteq \{1, \ldots, n\}$ that contains exactly two elements. It is clear that the assignments for these new atoms are in one-to-one correspondence with the assignments that make $\llbracket \phi \rrbracket$ true. Using this trick, we can indirectly refer to simple graphs without writing out any propositional formula.

Now, let us extend $\mathcal{L}$ by adding a new sort $\tau$ and a new relational symbol $R(z^{\tau}, x^{\sigma})$. Define $I(\tau) = \{1, \ldots, k\}$. From Example \ref{FunctionExample}, we know how to write a sentence $\psi$ that states that $R(z^{\tau}, x^{\sigma})$ is an injective function from $\{1, \ldots, k\}$ into $\{1, \ldots, n\}$. 
Then, consider the sentence  
\[
\theta = \psi \wedge \forall z_1^{\tau} z_2^{\tau} \forall x_1^{\sigma} x_2^{\sigma} \left[ (R(z_1^{\tau}, x_1^{\sigma}) \wedge R(z_2^{\tau}, x_2^{\sigma}) \wedge z_1^{\tau} \neq z_2^{\tau}) \to E(x_1^{\sigma}, x_2^{\sigma}) \right],
\]
which states that the function $R$ maps different elements to connected vertices. This simply says that $R$ describes a $k$-clique inside the graph.
Using $q_{ui}$ to encode whether $R(u, i)$ holds, for $u \leq k$ and $i \leq n$, the propositional formula $\llbracket \theta \rrbracket$ is the conjunction of the following:
\begin{itemize}
  \item[$\bullet$] 
  $\bigvee_{i \leq n} q_{ui}$, for all $u \leq k$,
  \item[$\bullet$] 
  $\neg q_{ui_1} \vee \neg q_{ui_2}$, for all $u \leq k$ and any $i_1 \neq i_2 \leq n$,
  \item[$\bullet$] 
  $\neg q_{u_1i} \vee \neg q_{u_2i}$, for all $u_1 \neq u_2 \leq k$ and any $i \leq n$,
  \item[$\bullet$] 
  $\neg q_{u_1i_1} \vee \neg q_{u_2i_2} \vee p_{i_1i_2}$, for all $u_1 \neq u_2 \leq k$ and any set $\{i_1, i_2\} \subseteq \{1, \ldots, n\}$ with two elements.
\end{itemize}
Note that the formula is monotone in $\bar{p}$. This reflects the fact that if additional edges are added to the graph, a $k$-clique selected by $\bar{q}$ remains a $k$-clique.
\end{example}

\begin{example}\label{ExampleColor}
We continue with a similar language as in Example \ref{ExampleClique}, consisting of two sorts, $\sigma$ and $\tau$, and binary relational symbols $E(x^{\sigma}, y^{\sigma})$ and $R(x^{\sigma}, z^{\tau})$. Define $I(\sigma)=\{1, \ldots, n\}$ and $I(\tau)=\{1, \ldots, l\}$. From Example \ref{FunctionExample}, we know how to write a sentence $\phi$ that states that $R(x^{\sigma}, z^{\tau})$ is a function from $\{1, \ldots, n\}$ into $\{1, \ldots, l\}$. 
Then, consider the sentence  
\[
\psi = \phi \wedge \forall z_1^{\tau} z_2^{\tau} \forall x_1^{\sigma} x_2^{\sigma} \left[(R(x^{\sigma}_1, z^{\tau}_1) \wedge R(x^{\sigma}_2, z^{\tau}_2) \wedge E(x^{\sigma}_1, x^{\sigma}_2)) \to z^{\tau}_1 \neq z^{\tau}_2 \right],
\]
which states that $R$ maps connected vertices into different elements. This simply says that $R$ describes an $l$-coloring of the graph.
Using $p_{ij}$'s as in Example \ref{ExampleClique} to encode the simple graph and $r_{ia}$ to encode $R(i, a)$, for any $i \leq n$ and $a \leq l$, the formula $\llbracket \psi \rrbracket$ is the conjunction of the following:
\begin{itemize}
  \item[$\bullet$] 
  $\bigvee_{a \leq l} r_{ia}$, for all $i \leq n$,
  \item[$\bullet$] 
  $\neg r_{ia_1} \vee \neg r_{ia_2}$, for all $a_1 \neq a_2 \leq l$ and any $i \leq n$,
  \item[$\bullet$] 
  $\neg r_{i_1a} \vee \neg r_{i_2a} \vee \neg p_{i_1i_2}$, for all $a \leq l$ and any set $\{i_1, i_2\} \subseteq \{1, \ldots, n\}$ with two elements.
\end{itemize}
\end{example}

So far, we have seen that any first-order formula with uninterpreted relational symbols, provided it refers exclusively to finite domains, can be translated into an equivalent propositional formula. By systematically ranging over all finite possibilities, one can express the finite content of a first-order formula as a sequence of propositional formulas.
More precisely, let $\mathcal{L}$ be a multi-sorted first-order language with equality and without any function symbols. For each sort $\sigma$ in $\mathcal{L}$, fix a polynomial $P_{\sigma}$ in one variable, and for each natural number $n \in \mathbb{N}$, define the interpretation $I_n$ by setting $I_n(\sigma) = \{1, \ldots, P_{\sigma}(n)\}$. Let $\llbracket \phi \rrbracket_n$ denote the propositional translation of the first-order $\mathcal{L}$-sentence $\phi$ with respect to $I_n$. The resulting sequence $\{\llbracket \phi \rrbracket_n\}_n$ can then be viewed as a propositional encoding of $\phi$ across all finite domains, parametrized polynomially by $n$. It is also customary to restrict to $n \geq n_0$, for some given $n_0 \in \mathbb{N}$ to encode $\phi$ evaluated over all \emph{large enough} finite domains.
Since the size of each domain $I_n(\sigma)$ grows polynomially with $n$, it follows that the size of $\llbracket \phi \rrbracket_n$ is also polynomially bounded in $n$. Consequently, this translation produces a sequence of short propositional formulas (see Section \ref{Sec: Preliminaries} for the definition). In cases where all extensions of $I_n$ to a full interpretation make $\phi$ true, the translation gives us a sequence of short classical \emph{tautologies}. A canonical example of such a situation is:

\begin{example}[Propositional Pigeonhole Principle] \label{Example: PHP}
Let $n_0 \in \mathbb{N}$ be a number and $m(n)$ be a polynomially bounded function satisfying $m(n) > n$, for any $n \geq n_0$. By the pigeonhole principle, there is no injective function from the set $\{1, \ldots, m(n)\}$ to the set $\{1, \ldots, n\}$, if $n \geq n_0$. We want to express this fact propositionally. As illustrated in Example~\ref{FunctionExample}, to propositionally express that a binary relation $R \subseteq \{1, \ldots, m(n)\} \times \{1, \ldots, n\}$ is an injective function, it suffices to use atoms $p_{ij}$ to indicate that $R(i, j)$, for any $i \leq m(n)$ and $j \leq n$. Then, the conjunction of the following clauses is the propositional translation we needed:
\begin{itemize}
    \item[$\bullet$] 
    $\neg p_{ij_1} \vee \neg p_{ij_2}$, for all $i \leq m(n)$ and $j_1, j_2 \leq n$ with $j_1 \neq j_2$,
    \item[$\bullet$]  
    $\neg p_{i_1j} \vee \neg p_{i_2j}$, for all $i_1, i_2 \leq m(n)$ and $j \leq n$ with $i_1 \neq i_2$,
    \item[$\bullet$] 
    $\bigvee_{j \leq n} p_{ij}$, for all $i \leq m(n)$.
\end{itemize}
Since every assignment for the above clauses corresponds one-to-one with an injective function $R \subseteq \{1, \ldots, m(n)\} \times \{1, \ldots, n\}$, and there is no such $R$, we conclude that the above conjunction is unsatisfiable. We define the \emph{propositional pigeonhole principle}, denoted by $\mathrm{PHP}^{m(n)}_n$, as the negation of this conjunction. Therefore, $\{\mathrm{PHP}^{m(n)}_n\}_{n \geq n_0}$ is a sequence of short tautologies. Notable choices of $m(n)$ discussed in the literature include $m(n) = n + 1$, $m(n) = 2n$, and $m(n) = n^2$ \cite{krajivcek2019proof}.
\end{example}

Using the explained translation, we transformed first-order formulas into sequences of short propositional formulas, establishing a connection between two types of \emph{description}. We can also shift the first-order side of this connection to something computational to transform any language in $\mathbf{NP}$ into a sequence of short propositional formulas. The key is to find an appropriate way to \emph{describe} the languages in $\mathbf{NP}$.
To this goal, interpret any binary string $w \in \{0, 1\}^*$ of length $n$ as a unary relation $R_w$ on the set $\{1, \ldots, n\}$, where $i \in R_w$ if and only if $w_i = 1$, for any $i \leq n$. Then, one direction of the celebrated Fagin's Theorem states:

\begin{theorem}[Fagin \cite{immerman1998descriptive}] \label{Fagin}
Let $L \in \mathbf{NP}$ be a language. Then, there exists a one-sorted first-order language $\mathcal{L}$ with no function symbols and a first-order formula $\phi(\bar{S}, R)$, where $R$ is a unary predicate such that
$(\{1, \ldots, n\}, R_w) \vDash \exists \bar{S} \, \phi(\bar{S}, R)$
iff $w \in L$, for every $w \in \{0, 1\}^*$.
\end{theorem}

Using this descriptive characterization of $\mathbf{NP}$ and the translation process we explained, we obtain the following:

\begin{theorem}\label{NPToPropositional}
For any language $L \in \mathbf{NP}$ (resp. $L \in \mathbf{CoNP}$), there exists a sequence $\{\phi_n(p_1, \ldots, p_n, \bar{q})\}_n$ of short propositional formulas such that for every $n \in \mathbb{N}$ and every $w \in \{0, 1\}^n$, the formula $\phi_n(w_1, \ldots, w_n, \bar{q})$ is satisfiable (resp. a tautology) iff $w \in L$.   
\end{theorem}
\begin{proof}
We will prove the case for $\mathbf{NP}$ only; the other case follows immediately as a consequence. By Theorem \ref{Fagin}, there exists a one-sorted first-order language $\mathcal{L}$ and a first-order formula $\phi(R, \bar{S})$ in $\mathcal{L}$ such that for any string $w \in \{0,1\}^n$, we have $w \in L$ iff $(\{1, \ldots, n\}, R_w) \vDash \exists \bar{S} \, \phi(R, \bar{S})$.
Fix $n \in \mathbb{N}$, use the interpretation $I_n$ that maps the only sort to $\{1, \ldots, n\}$, and translate the formula $\phi(R, \bar{S})$ into its propositional counterpart $\phi_n = \llbracket \phi \rrbracket_n$. Let $p_1, \ldots, p_n$ and $\bar{q}$ denote the propositional variables in $\phi_n$ that encode the relations $R$ and $\bar{S}$, respectively.
Since there is a one-to-one correspondence between interpretations of the relations in $\bar{S}$ (resp.\ $R$) and Boolean assignments to $\bar{q}$ (resp.\ $\bar{p}$), it follows that $(\{1, \ldots, n\}, R_w) \vDash \exists \bar{S} \, \phi(R, \bar{S})$ iff the Boolean assignment $w_1, \ldots, w_n$ to the atoms $p_1, \ldots, p_n$ makes $\phi_n(p_1, \ldots, p_n, \bar{q})$ satisfiable.
\end{proof}

\subsection{Computational Power} \label{Subsec: Computational Power}

In the previous subsection, we showed that the propositional language is expressive enough to represent first-order formulas over finite domains. Moreover, by focusing on satisfiability, we saw that any language in $\mathbf{NP}$ can be encoded as a sequence of short propositional formulas. Similarly, the propositional language can be used to \emph{compute} any function between finite domains. To formalize such computations, we introduce \emph{circuits} as our model of computation:

\begin{definition}[$\mathfrak{L}$-Circuit]
Let $\mathfrak{L} \in \{\mathcal{L}_b, \mathcal{L}_p, \mathcal{L}_{\Box}\}$. A \emph{multi $\mathfrak{L}$-circuit} $C$ with $n$ inputs, $m$ outputs, and size $s$ is a directed acyclic graph $C = (V, E, t, O)$ where:
\begin{itemize}
  \item[$\bullet$] $V$ is the set of \emph{gates} with $s$ elements, and $E \subseteq V \times V$ is the set of directed \emph{edges},
  \item[$\bullet$] $t: V \to \{x_1, \dots, x_n\} \cup \mathcal{L}$ assigns a type to each gate, where the arity of the type matches the gate's indegree (the arity of $x_i$'s is zero),
  \item[$\bullet$] $O$ is a sequence of $m$ gates that serve as the \emph{output} gates.
\end{itemize}
A multi $\mathfrak{L}$-circuit is called \emph{monotone} if it contains no negation or implication gates. It is called an \emph{$\mathfrak{L}$-circuit} if it has exactly one output.
\end{definition}

Observe that any $\mathfrak{L}$-formula can be viewed as an $\mathfrak{L}$-circuit in which every gate other than the output gate has outdegree one. Moreover, given any $\mathfrak{L}$-circuit, there is a canonical way to convert it into an $\mathfrak{L}$-formula by duplicating shared sub-circuits as needed. We denote this formula by $[C]$.
We call a (multi) $\mathcal{L}_b$-circuit simply a \emph{(multi) circuit}. Any (multi) $\mathcal{L}_p$-circuit can be transformed into a (multi) circuit by replacing $\to$ gates with combinations of $\neg$ and $\vee$ gates. Conversely, any (multi) circuit can be transformed into a (multi) $\mathcal{L}_p$-circuit by rewriting $\neg$ using $\to$ and $\bot$. These transformations incur only a linear increase in (multi) circuit size. Therefore, we treat (multi) $\mathcal{L}_p$-circuits and (multi) circuits as essentially equivalent.

Given any (multi) $\mathcal{L}_{\Box}$-circuit $C$, we can erase all $\Box$ gates and continue the edge coming into the gate to all edges going out of it to obtain a (multi) $\mathcal{L}_p$-circuit, which we denote by $C^f$. Alternatively, we may eliminate all incoming edges to each $\Box$ gate and replace each such gate with a constant $\top$ gate, yielding another (multi) $\mathcal{L}_p$-circuit denoted by $C^c$. Note that the size of $C^f$ and $C^c$ are bounded by the size of $C$. These two operations correspond to the forgetful and collapse translations introduced in Section~\ref{Sec: Preliminaries}.

If $C$ is a multi circuit with $n$ inputs and $m$ outputs and $w \in \{0, 1\}^n$, we write $C(w)$ to denote the string output on the sequence of output gates when $C$ runs on the input $x_i = w_i$, for any $1 \leq i \leq n$. A function $f : \{0,1\}^n \to \{0,1\}^m$ is said to be \emph{computed} by $C$ if $C(w) = f(w)$ for every $w \in \{0,1\}^n$.
It is easy to verify that any (monotone) function can be computed by a (monotone) multi circuit, although the size of the multi circuit may be exponential in $n$ and $m$. Since any finite set can be embedded into $\{0,1\}^n$ for some $n \in \mathbb{N}$, we conclude, just as we did in the context of first-order descriptions, that the propositional language is expressive enough to compute any function between finite domains.

Similar to the previous subsection, we can use \emph{sequences} of multi circuits to compute functions defined over all binary strings. Before formalizing this idea, we introduce a technical device. For any string $w \in \{0, 1\}^l$ and any fixed $m \geq l$, we define the padded version of $w$ of length $2m$ by $pad_m(w) = 1w_1 1w_2 \ldots 1w_l 00\ldots 0$, where the number of zeros is $2m - 2l$.

\begin{definition}
A sequence $\{C_n\}_n$ of multi circuits is said to compute a function $f : \{0, 1\}^* \to \{0, 1\}^*$ if each $C_n$ has $n$ inputs and $2out_f(n)$ outputs, and for every $n \in \mathbb{N}$ and every $w \in \{0, 1\}^n$, we have $|f(w)| \leq out_f(n)$ and $C_n(w) = \text{pad}_{out_f(n)}(f(w))$.
We say that the function $f : \{0, 1\}^* \to \{0, 1\}^*$ is \emph{computable by poly-size circuits} if it is computed by such a sequence $\{C_n\}_n$ where the size of each $C_n$ is bounded by a polynomial in $n$. Otherwise, we say that $f$ is \emph{hard for poly-size circuits}.
\end{definition}

\begin{remark}
The padding is necessary because, in general, the length of $f(w)$ may vary across different inputs of the same length, while the number of output gates in $C_n$ must be fixed for each $n$. Padding ensures that the values $f(w)$ for $w \in \{0, 1\}^n$ have the same length. However, if they already have the same length, then without loss of generality, we may define the computation without padding by requiring $C_n(w) = f(w)$ for every $n \in \mathbb{N}$ and $w \in \{0, 1\}^n$. The reason is that when we use padding and the multi circuit $C_n$ has additional output gates to add constant zeros and ones to the actual output, these additional outputs can be safely ignored. This simplification does not affect the multi circuit's size, which is the primary parameter of interest for us.
\end{remark}

\begin{definition}[Class $\mathbf{P/poly}$]
A language $L$ is in $\mathbf{P/poly}$ if its characteristic function $\chi_L: \{0, 1\}^* \to \{0, 1\}$ is computable by poly-size circuits, i.e., there exists a sequence of circuits $\{C_n\}_n$ such that the size of each $C_n$ is polynomially bounded in $n$, and for every $w \in \{0, 1\}^n$, we have $w \in L$ if and only if $C_n(w) = 1$. 
\end{definition}

Computation with poly-size circuits is the \emph{non-uniform} analogue of polynomial time computation, as it allows for an arbitrary choice of the multi circuit $C_n$ for each input length $n$, as long as the size of these multi circuits remains polynomial in $n$. As expected, non-uniform computation must subsume uniform computation. In fact, if a function $f: \{0, 1\}^* \to \{0, 1\}^*$ is in $\mathbf{FP}$, it is computable by poly-size circuits, which implies $\mathbf{P} \subseteq \mathbf{P/poly}$. It is widely believed that $\mathbf{NP} \nsubseteq \mathbf{P/poly}$. Since this would imply $\mathbf{NP} \neq \mathbf{P}$, this belief motivates the combinatorial approach to attack the conjecture $\mathbf{NP} \neq \mathbf{P}$ by identifying languages in $\mathbf{NP}$ that require super-polynomial-size circuits. Despite decades of effort, this remains one of the most prominent open problems in theoretical computer science. However, for \emph{monotone} circuits, such lower bounds have been successfully established. For instance, following Example \ref{ExampleClique}, any binary string of length $\binom{n}{2}$ can be interpreted as the adjacency matrix of a simple graph on $n$ vertices. Let $L^n$ be the set of all such binary strings that, when read as graphs, contain a $\lfloor n^{1/4} \rfloor$-clique, and define $L = \bigcup_n L^n$. It is clear that $L \in \mathbf{NP}$. However:
\begin{theorem}[Razborov \cite{razborov1985lower}, Alon-Boppana \cite{alon1987monotone}]
For any sequence $\{C_n\}_n$ of monotone circuits computing $L$, the size of $C_{\binom{n}{2}}$ is at least $2^{\Omega(n^{1/8}\log n)}$.
\end{theorem}

\subsubsection{A Bit of Cryptography}

In this subsubsection, we introduce some basic concepts and examples from cryptography, which interestingly employs the seemingly negative phenomenon of \emph{computational hardness} to achieve the positive goal of \emph{secure communication}. For our purposes, we simplify everything to the level needed in this chapter.

We begin by developing the intuition behind a notion we refer to as a \emph{one-way protocol}, starting with its simplest setting. Let $P \subseteq \{0, 1\}^*$ be a polynomial-time decidable predicate, and let $h: P \to \{0, 1\}^*$ be a polynomial-time computable function. Intuitively, an element $u \in P$ represents some \emph{source data} that we wish to keep secret, yet must still communicate about through public channels. To overcome this seemingly contradictory situation, we apply $h$ to $u$, producing $h(u)$, and transmit $h(u)$ publicly instead of $u$ itself. Naturally, we prefer $h$ to be injective, since otherwise the transformation may obscure or distort the original data $u$. Crucially, the utility of $h$ lies in its assumed \emph{one-wayness}, namely that it is computationally hard to invert, making it infeasible for a potential eavesdropper to recover $u$ from $h(u)$.

We now generalize this setup. Retain $P$ and $h$ as defined, but relax the requirement that $h$ be injective. Instead, introduce a polynomial-time computable function $k: P \to \{0, 1\}^*$ such that
$
h(u_1) = h(u_2)$ implies $k(u_1) = k(u_2)$,
for all $u_1, u_2 \in P$. We refer to this condition as \emph{semi-injectivity}, a weakening of injectivity. In this scenario, the goal is not necessarily to communicate the source data $u$, but rather some \emph{derived information} $k(u)$. This allows for certain distortions in $u$, provided that the derived output $k(u)$ remains intact. Semi-injectivity guarantees precisely this. Note that when $k$ is the identity function, this notion collapses to ordinary injectivity of $h$.
In this generalized framework, it is no longer sufficient to merely assume that $h$ is hard to invert. Instead, we must assume that computing $k(u)$ from $h(u)$ is computationally hard. Since $k$ is efficiently computable, this assumption automatically implies the one-wayness of $h$.
We are now ready to formally define a one-way protocol:

\begin{definition}[One-Way Protocol]
We call a triple $(P, h, k)$ a \emph{one-way protocol}\footnote{This is a non-standard adaptation of \emph{one-way permutations} tailored for the purposes of this chapter. It differs from the conventional definition in three key ways. First, we introduce an auxiliary function $k$ to generalize the setting and accommodate examples such as the Diffie–Hellman protocol. Second, while standard cryptographic hardness is defined against probabilistic polynomial-time adversaries, we adopt a circuit-based notion of hardness more appropriate for our context. Third, in one-way permutations, we have $|h(u)| = |u|$. We have relaxed this condition to boundedness, which helps to present protocols like RSA and Diffie-Hellman as one-way protocols with only minor adjustments.} if $P$ is a polynomial-time decidable predicate and $h, k: \{0, 1\}^* \to \{0, 1\}^*$ are polynomial-time computable functions satisfying:
\begin{itemize}
    \item[$\bullet$] \emph{semi-injectivity}, i.e., for all $u_1, u_2 \in P$, if $h(u_1) = h(u_2)$, then $k(u_1) = k(u_2)$,
    \item[$\bullet$] \emph{boundedness}, i.e., there exists a polynomial-time computable function $l: \{0, 1\}^* \to \{0, 1\}^*$ and a polynomial $p$ such that $|k(u)| = l(h(u))$ and $|u| \leq p(|h(u)|)$ for any $u \in \{0, 1\}^*$,
    \item[$\bullet$] \emph{security against $\mathbf{P/poly}$ adversaries}, i.e., computing $k(u)$ from $h(u)$ is hard for poly-size circuits. Formally, this means that there is no function $f: \{0, 1\}^* \to \{0, 1\}^*$ computable by poly-size circuits such that $k(u)=f(h(u))$ for all $u \in \{0, 1\}^*$.
\end{itemize}
\end{definition}

The existence of a one-way protocol implies that $\mathbf{NP} \neq \mathbf{P}$. To prove that, assume $\mathbf{NP} = \mathbf{P}$ and let $(P, h, k)$ be a one-way protocol. We reach a contradiction. Since $h$ and $k$ are polynomial-time computable, the language
\[
\{(v, i) \mid \exists u (|u| \leq p(|v|) \wedge i \leq |k(u)| \wedge k(u)_i = 1)\}
\]
is in $\mathbf{NP} \subseteq \mathbf{P}$. Hence, we can effectively compute the $i$-th bit of $k(u)$ by reading $h(u)$ and $i \leq |k(u)|$. Given that $|k(u)| \leq |u|^{O(1)} \leq |h(u)|^{O(1)}$, we can compute all bits of $k(u)$ in time $|h(u)|^{O(1)}$.
This implies the existence of a polynomial-time computable function $f: \{0, 1\}^* \to \{0, 1\}^*$ such that $f(h(u)) = k(u)$ for any $u \in \{0, 1\}^*$. Since $f$ is in $\mathbf{FP}$, it is also computable by poly-size circuits which contradicts the security of the one-way protocol. Therefore, we conclude that $\mathbf{NP} \neq \mathbf{P}$.

As the existence of a one-way protocol implies $\mathbf{NP} \neq \mathbf{P}$, we currently do not know whether any one-way protocol exists. Nevertheless, there are concrete examples of triples $(P, h, k)$ that are widely believed to form one-way protocols based on standard cryptographic assumptions.
In the remainder of this subsubsection, we introduce two central examples from classical cryptography that we will refer to later: the RSA protocol and the Diffie–Hellman protocol.

\begin{example}[RSA Protocol]\label{RSA}
First, let us explain RSA in its original form as a way for a user to communicate a number to us securely.
Let $N$ be a number and $1 < e < N$ such that $(e, \phi(N)) = 1$, where $\phi(N)$ is Euler’s function, which returns the number of integers less than $N$ that are coprime to $N$. The pair $(N, e)$ is called the \emph{public key} and is known to everyone. However, the factorization of $N$ is kept secret and is known only to us.
Since $(e, \phi(N)) = 1$, there exists $d < N$ such that $ed \equiv 1 \pmod{\phi(N)}$. This $d$ is called the \emph{secret key}. It is widely believed that one cannot efficiently compute $d$ from $N$ and $e$ alone. However, if the factorization of $N$ is known, then $\phi(N)$ can be computed easily, and hence so can $d$. Therefore, we are able to compute $d$, while the public cannot.
Now, let $u < N$ be a number such that $(u, N) = 1$. To communicate $u$ to us securely, the user computes the remainder $v$ of $u^e$ modulo $N$ and sends $v$ through public channels. Since $(u, N) = 1$, Euler’s theorem gives $u^{\phi(N)} \equiv 1 \pmod{N}$. Therefore,
$v^d \equiv (u^e)^d = u^{ed} \equiv u \pmod{N}$.
This allows us to recover the original message $u$ from the ciphertext $v$.

Now, with a slight technical adjustment, one can view RSA as a one-way protocol. After suitable encoding of the tuples of binary strings as a binary string, define $P$ as the set of tuples $(N, M, e, u)$, where $(e, M) = 1$, $u < N$, $1 < e < N$, and $u^M \equiv 1 \pmod{N}$. Note that $M$ plays the role of $\phi(N)$ to make the process simpler. As we have the polynomial-time Euclidean algorithm to compute the greatest common divisor and modular exponentiation is polynomial-time computable, it is clear that $P$ is a polynomial-time predicate. Then, define the function $h: P \to \{0, 1\}^*$ by
$
h(N, M, e, u) = (N, e, v),
$
where $v$ is the remainder of $u^e \bmod N$. Also, define $k(N, M, e, u) = \overline{0u}$, where $\overline{0u}$ is a sequence of zeros added to the binary expansion of $u$ such that $|\overline{0u}|=|N|$. Padding $u$ with zeros is a technicality to make $|k(u)|$ easily computable from $h(u)$.
Both $h$ and $k$ are polynomial-time computable. For semi-injectivity, if $h(N_1, M_1, e_1, u_1) = h(N_2, M_2, e_2, u_2)$, then $N_1 = N_2 = N$ and $e_1 = e_2 = e$. Moreover, $u_1^e \equiv u_2^e \pmod{N}$. Since $(e, M) = 1$, there is $d < M$ such that $ed \equiv 1 \pmod{M}$. Therefore, as $u_1^M \equiv u_2^M \equiv 1 \pmod{N}$, we have 
\[
u_1 = u_1^1 \equiv u_1^{ed} \equiv u_2^{ed} \equiv u_2^1 = u_2 \pmod{N}.
\]
As $u_1, u_2 < N$, we reach $u_1=u_2$. Moreover, as $N_1=N_2=N$, the number of zeros needed to pad $u_1=u_2$ into a string with the length $|N|$ are equal. Therefore, $k(N_1, M_1, e_1, u_1) = \overline{0u_1}=\overline{0u_2}=k(N_2, M_2, e_2, u_2)$.
For boundedness, as $M, u \leq N$, we have 
\[
|(N, M, e, u)| \leq O(|(N, e, v)|)=O(|h(N, M, e, u)|).
\]
Moreover, note that $|k(N, M, e, u)|=|\overline{0u}|=|N|$ is computable from the value $h(N, M, e, u)=(N, e, v)$ in polynomial time, as we have a direct access to $N$. For the hardness, it is widely believed that RSA is secure against $\mathbf{P/poly}$ adversaries \cite{krajivcek2019proof}. Under this assumption, RSA qualifies as a one-way protocol.
\end{example}

\begin{example}[Diffie–Hellman Key Exchange Protocol]\label{Diffie-Hellman}
First, let us explain the Diffie–Hellman key exchange protocol in its original form as a way to set a secret key between Alice and Bob over public channels.
Let $N$ and $g$ be publicly known natural numbers. To establish a shared secret key, Alice chooses a secret number $a \leq N$ and sends the remainder $X$ of $g^a$ modulo $N$ to Bob over public channels. Similarly, Bob chooses a secret number $b \leq N$ and sends the remainder $Y$ of $g^b$ modulo $N$ to Alice over public channels. Note that the values $X$ and $Y$ are also publicly known, but $a$ is known only to Alice and $b$ only to Bob.
Since Bob knows $b$, he can compute $X^b \equiv g^{ab} \pmod{N}$, and similarly, Alice can compute $Y^a \equiv g^{ab} \pmod{N}$. The resulting number $g^{ab}$ is the shared secret key known to both parties. The central premise of the Diffie–Hellman protocol is that computing the remainder of $g^{ab}$ modulo $N$ from the public data $N$, $g$, $X$, and $Y$ is computationally infeasible. In fact, when $N = pq$ for primes $p$ and $q$, and $(g, N) = 1$, it is known that breaking the Diffie-Hellman protocol is harder than factoring $N$~\cite{bonet1997no}.
Again, this protocol can be seen as a one-way protocol. After encoding tuples of binary strings as a binary string, define $P$ as the set of all tuples $(N, g, a, b)$, where $a, b \leq N$, and note that $P$ is a polynomial-time predicate. Then, define a function $h: P \to \{0, 1\}^*$ by
$
h(N, g, a, b) = (N, g, X, Y),
$
where $X$ and $Y$ are the remainders of $g^a$ and $g^b$ modulo $N$, respectively. Also define 
$
k(N, g, a, b) = \bar{0}Z,
$
where $Z$ is the binary expansion of the remainder of $g^{ab}$ modulo $N$ and $\bar{0}$ is a sequence of zeros such that $|\bar{0}Z|=|N|$. Both $h$ and $k$ are clearly computable in polynomial time as modular exponentiation is in polynomial time.
For semi-injectivity, if $h(N_1, g_1, a_1, b_1) = h(N_2, g_2, a_2, b_2)$, then $N_1 = N_2 = N$, $g_1 = g_2 = g$, $g^{a_1} \equiv g^{a_2} \pmod{N}$, and $g^{b_1} \equiv g^{b_2} \pmod{N}$. Therefore,
\[
g^{a_1 b_1} \equiv (g^{a_1})^{b_1} \equiv (g^{a_2})^{b_1} \equiv (g^{b_1})^{a_2} \equiv (g^{b_2})^{a_2} \equiv g^{a_2 b_2} \pmod{N},
\]
which implies that $k(N_1, g_1, a_1, b_1) = k(N_2, g_2, a_2, b_2)$.
For boundedness, as $a, b \leq N$, we have 
\[
|(N, g, a, b)| \leq O(|(N, g, X, Y)|) = O(|h(N, g, a, b)|).
\]
Moreover, note that $|k(N, g, a, b)| = |\bar{0}Z| = |N|$ is computable from the value $h(N, g, a, b) = (N, g, X, Y)$ in polynomial time, as we have direct access to $N$.
For the hardness condition, it is widely believed that the Diffie–Hellman protocol is secure against $\mathbf{P/poly}$ adversaries. Under this assumption, the Diffie–Hellman key exchange protocol qualifies as a one-way protocol.
\end{example}

\subsubsection{Disjoint $\mathbf{NP}$ Pairs} \label{subsubsec: DNP}

In this subsubsection, we introduce and explain the concept of a disjoint $\mathbf{NP}$ pair. Such pairs arise naturally in computational complexity and cryptography. Through the translation of languages in $\mathbf{NP}$ into propositional formulas, they are also closely connected to Craig interpolation for $\mathsf{CPC}$, as will be explained in Section~\ref{Section: Interpolation}. We begin with the formal definition:

\begin{definition}[Disjoint $\mathbf{NP}$ Pair]
A \emph{disjoint $\mathbf{NP}$ pair} is a pair of languages $(L, R)$ such that $L, R \in \mathbf{NP}$ and $L \cap R = \emptyset$. A \emph{separator} for such a pair is a language $S$ satisfying
$L \subseteq S$ and $S \cap R = \emptyset$:
\vspace{-8pt}
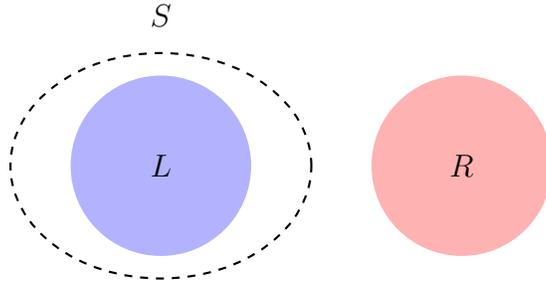
\begin{center}
% Wrap the tikzpicture environment in \scalebox
\scalebox{1}{%
  \begin{tikzpicture}
      % Draw the sets L and R
      \fill[blue!30] (-2,0) circle (1.2cm) node[black] {$L$};
      \fill[red!30] (2,0) circle (1.2cm) node[black] {$R$};

      % Separator S as a dashed ellipse around L but avoiding R
      \draw[black, thick, dashed] (-2,0) ellipse (2cm and 1.5cm);

      % Label for separator
      \node[black] at (-2,2.0) {$S$};
  \end{tikzpicture}% Note the % here to avoid extra space
} % End \scalebox
\end{center}
A disjoint $\mathbf{NP}$ pair is called \emph{hard} if it has no separators in $\mathbf{P/poly}$.
\end{definition}

A natural source of disjoint $\mathbf{NP}$ pairs comes from languages in the class $\mathbf{NP} \cap \mathbf{CoNP}$. Specifically, if $L \in \mathbf{NP} \cap \mathbf{CoNP}$, then the pair $(L, L^c)$ forms a disjoint $\mathbf{NP}$ pair. Note that $L$ is the only separator for this pair. Therefore, if $\mathbf{NP} \cap \mathbf{CoNP} \nsubseteq \mathbf{P/poly}$, then there exists a disjoint $\mathbf{NP}$ pair that has no separator in $\mathbf{P/poly}$. Since it is believed that $\mathbf{NP} \cap \mathbf{CoNP} \nsubseteq \mathbf{P/poly}$, we conjecture that a hard disjoint $\mathbf{NP}$ pair must exist.
However, the existence of such pairs remains unproven, as it would imply $\mathbf{NP} \neq \mathbf{P}$. Indeed, if $\mathbf{NP} = \mathbf{P}$, then for any disjoint $\mathbf{NP}$ pair $(L, R)$, the language $L \in \mathbf{NP} = \mathbf{P} \subseteq \mathbf{P/poly}$ provides an obvious separator in $\mathbf{P/poly}$.

There are some concrete settings that lead to the existence of hard disjoint $\mathbf{NP}$ pairs, where we rely on stronger complexity-theoretic conjectures. In return, we gain concrete instances of such pairs, which can be exploited later. A rich source of hard disjoint $\mathbf{NP}$ pairs arises in cryptography.
Let $(P, h, k)$ be a one-way protocol and let $p$ be the polynomial that $|u| \leq p(|h(u)|)$, for any $u \in \{0, 1\}^*$. We want to construct a hard disjoint $\mathbf{NP}$ pair.
Consider the following languages:
\begin{align*}
L &= \{(v, i) \mid \exists u \in P\, [(|u| \leq p(|v|)) \wedge (i \leq |k(u)|) \wedge (h(u) = v) \wedge (k(u))_i = 1]\}, \\
R &= \{(v, i) \mid \exists u \in P\, [(|u| \leq p(|v|)) \wedge (i \leq |k(u)|) \wedge (h(u) = v) \wedge (k(u))_i = 0]\}.
\end{align*}
Since $h$ and $k$ are polynomial-time computable, it is clear that $L, R \in \mathbf{NP}$. Moreover, $L \cap R = \varnothing$. To see this, suppose for contradiction that $(v, i) \in L \cap R$. Then there exist $u_1, u_2 \in P$ such that $h(u_1) = v = h(u_2)$. By semi-injectivity, we have $k(u_1) = k(u_2)$, which implies $k(u_1)_i = k(u_2)_i$. However, $k(u_1)_i = 1$ and $k(u_2)_i = 0$, yielding a contradiction. Thus, $(L, R)$ is indeed a disjoint $\mathbf{NP}$ pair.

Now, using the assumed security of $(P, h, k)$ against $\mathbf{P/poly}$ adversaries, we show that $(L, R)$ has no separator in $\mathbf{P/poly}$. Suppose, for contradiction, that $S \in \mathbf{P/poly}$ is a separator for $(L, R)$. Then, there exists a sequence of poly-size circuits $\{C_n\}_n$ that computes $S$. First, we show that for any $u \in \{0, 1\}^*$ with $|h(u)| = n$ and any $i \leq |k(u)|$, we have $C_{n+|i|}(h(u), i) = k(u)_i$.
The reasoning is as follows: if $C_{n+|i|}(h(u), i) = 1$, then $(h(u), i) \in S$, hence $(h(u), i) \notin R$, as $S \cap R = \emptyset$. Given that $i \leq |k(u)|$ and $|u| \leq p(|h(u)|)$, it must be that $k(u)_i \neq 0$, so $k(u)_i = 1$. Similarly, if $C_{n+|i|}(h(u), i) = 0$, then $k(u)_i = 0$. 

Now, as $|u| \leq p(|h(u)|)$ and $|k(u)| \leq |u|^{O(1)}$, there exists a polynomial $q(n)$ such that $|k(u)| \leq q(|h(u)|)$ for any $u \in \{0, 1\}^*$. Moreover, let $l: \{0, 1\}^* \to \{0, 1\}^*$ be the polynomial-time computable function satisfying $|k(u)|=l(h(u))$, for any $u \in \{0, 1\}^*$. 
As $l$ is polynomial-time computable, the predicate $i \leq l(w)$ is decidable in polynomial time. Therefore, it is in $\mathbf{P/poly}$. Let $\{D_n\}_n$ be a sequence of polynomial-size circuits that decide $i \leq l(w)$. For any $i \leq q(n)$, consider the circuit $D_{n+|i|}(x_1, \ldots, x_n, i)$.
Now, define the circuit $E_n(x_1, \ldots, x_n)$ in the following way: First, make the disjoint union of the circuit $C_{n+|i|}(x_1, \ldots, x_n, i)$ with the output $y_i$ and $D_n^i(x_1, \ldots, x_n)$ with the output $z_i$, for any $i \leq q(n)$. Then, add some $\wedge$ gates to compute $w_i=y_i \wedge z_i$, for any $i \leq q(n)$, and finally, for the output gates of $E_n$, pick the sequence $z_1w_1z_2w_2\ldots z_nw_n$. Since $q$ is a polynomial, the size of $E_n$ is polynomially bounded in $n$. Moreover, if $|h(u)| = n$, then for any $i \leq q(n)$, the output of $E_n(h(u))$ on $z_i$ decides if $i \leq l(h(u))=|k(u)|$, and its output on $y_i$ is $k(u)_i$ when $i \leq |k(u)|$. Therefore, if $i \leq |k(u)|$, we get the bit $k(u)_i$ on $w_i$, and when $i > |k(u)|$, we get zero on $w_i$. Therefore, the output of $E_n(h(u))$ is $pad_{2q(n)}(k(u))$. This means that $E_n$ computes $k(u)$ from $h(u)$ using poly-size multi circuits and considering the required padding. This contradicts the security of $(P, h, k)$. Therefore, $(L, R)$ has no separator in $\mathbf{P/poly}$.

\begin{corollary}\label{OneWayToDNP}
If one-way protocols exist, then hard disjoint $\mathbf{NP}$ pairs also exist.
\end{corollary}

\begin{example}\label{ExampleOfDNPOfRSA}
Apply the above construction to the RSA and Diffie–Hellman (believed to be) one-way protocols, as presented in Example~\ref{RSA} and Example~\ref{Diffie-Hellman}. Denote the corresponding disjoint $\mathbf{NP}$ pairs by $(\mathrm{RSA}_1, \mathrm{RSA}_0)$ and $(\mathrm{DH}_1, \mathrm{DH}_0)$, respectively. If these protocols are secure against $\mathbf{P/poly}$-adversaries, then $(\mathrm{RSA}_1, \mathrm{RSA}_0)$ and $(\mathrm{DH}_1, \mathrm{DH}_0)$ must be hard.
\end{example}

We have seen that the existence of a hard disjoint $\mathbf{NP}$ pair is an open problem. However, if we restrict the computational devices to monotone circuits, some known hardness theorems exist. First, we introduce the most important disjoint $\mathbf{NP}$ pair in this paper:

\begin{example}\label{Example: Clique-Color DNP}
Let $n_0 \in \mathbb{N}$, and let $K(n)$ and $L(n)$ be functions computable in time $n^{O(1)}$ such that $L(n) < K(n) \leq n$ for all $n \geq n_0$. 
Following Example~\ref{ExampleClique}, any binary string of length $\binom{n}{2}$ can be interpreted as the adjacency matrix of a simple graph on $n$ vertices. For any $k, l \in \mathbb{N}$, define $\mathrm{Clique}_{k,n}$ as the set of all such binary strings that, when interpreted as graphs, contain a $k$-clique, and define $\mathrm{Clique}_K = \bigcup_{n \geq n_0} \mathrm{Clique}_{K(n),n}$. Similarly, define $\mathrm{Color}_{l,n}$ as the set of graphs that are $l$-colorable, and set $\mathrm{Color}_L = \bigcup_{n \geq n_0} \mathrm{Color}_{L(n),n}$. It is clear that both $\mathrm{Clique}_K$ and $\mathrm{Color}_L$ are in $\mathbf{NP}$. Since $L(n) < K(n)$, it is impossible for a graph on $n$ vertices to simultaneously contain a $K(n)$-clique and be $L(n)$-colorable. Hence, $\mathrm{Clique}_K \cap \mathrm{Color}_L = \emptyset$, and the pair $(\mathrm{Clique}_K, \mathrm{Color}_L)$ forms a disjoint $\mathbf{NP}$ pair.
As a concrete special case, fix a real number $0 < \epsilon < 1/3$ and define $K(n) = \lfloor n^{2/3} \rfloor$ and $L(n) = \lfloor n^{2/3 - 2\epsilon} \rfloor$. These functions are computable in time $n^{O(1)}$ and satisfy $L(n) < K(n)$ for sufficiently large $n$. We denote $\mathrm{Clique}_K$ and $\mathrm{Color}_L$ in this setting by $\mathrm{Clique}_\epsilon$ and $\mathrm{Color}_\epsilon$, respectively.
\end{example}

Here is a well-known exponential lower bound on the size of monotone circuits separating Clique-Color pairs:

\begin{theorem}[Razborov \cite{razborov1985lower}, Alon-Boppana \cite{alon1987monotone}] \label{Alon-Boppana}
Let $l, k, n \in \mathbb{N}$ be numbers such that $3 \leq l < k$ and $\sqrt{l}\,k \leq \frac{n}{8 \log n}$. Any monotone circuit $C$ on $\binom{n}{2}$ variables computing a separator of $(\mathrm{Clique}_{k,n}, \mathrm{Color}_{l,n})$ has size at least $\frac{1}{8} 2^{\frac{\sqrt{l}+1}{2}}$.
\end{theorem}

\begin{corollary}\label{AlonBoppanaForEpsilon}
Let $0 < \epsilon < 1/3$ be a real number, and let $\{C_n\}_{n}$ be a sequence of monotone circuits computing a separator for the disjoint $\mathbf{NP}$ pair $(\mathrm{Clique}_\epsilon, \mathrm{Color}_\epsilon)$. Then, the size of $C_{\binom{n}{2}}$ is at least $2^{\Omega(n^{\frac{1}{3} - \epsilon})}$.  
\end{corollary} 

\begin{proof}
Recall that to define $(\mathrm{Clique}_\epsilon, \mathrm{Color}_\epsilon)$, we used $K(n) = \lfloor n^{2/3} \rfloor$ and $L(n) = \lfloor n^{2/3 - 2\epsilon} \rfloor$. 
As $\sqrt{L(n)}\,K(n) \leq n^{1/3 - \epsilon} n^{2/3} = n^{1 - \epsilon}$ and $0 < \epsilon < 1/3$, we eventually have $\sqrt{L(n)}\,K(n) \leq \frac{n}{8 \log n}$ and $\lfloor n^{2/3 - 2\epsilon} \rfloor \geq 3$. Therefore, we can apply Theorem~\ref{Alon-Boppana}.
\end{proof}

\section{Proof Systems} \label{Sec: Proof Systems}

In the previous section, we saw that the propositional language is sufficiently powerful to describe finite domains and to compute over them. This encompasses both the \emph{descriptive} and \emph{computational} aspects of the propositional language. We now turn to its third fundamental aspect: \emph{proofs}.

We use the term ``proof'' in the broadest possible sense. To explain this usage, consider two types of \emph{finitary objects}, denoted by $\mathbf{F}$ and $\mathbf{Pr}$ (for example, formulas, circuits, inequalities, polynomials, or sequences of such objects). Then, for any set $L \subseteq \mathbf{F}$, we interpret some of the elements of $\mathbf{Pr}$ as proofs or witnesses certifying that an object from $\mathbf{F}$ belongs to $L$.
The key insight, due to Cook and Reckhow~\cite{cook1979relative}, is that a meaningful notion of proof requires only two essential properties: \emph{soundness and completeness} with respect to $L$, and \emph{efficient verifiability}. The latter means that one can efficiently check whether a given object in $\mathbf{Pr}$ constitutes a valid proof for a given object in $\mathbf{F}$. Since any finitary object can be encoded as a binary string, this idea can be formalized as follows:

\begin{definition}
Let $L \subseteq \{0,1\}^*$. A \emph{proof system for $L$} is a polynomial-time decidable relation $P \subseteq \{0,1\}^* \times \{0,1\}^*$ such that, for every $w \in \{0, 1\}^*$, we have $w \in L$ if and only if there exists $u \in \{0, 1\}^*$ with $P(u, w)$. We call $u$ a $P$-proof of $w$.
\end{definition}

\begin{remark}
Usually, $L$ is the set of binary representations of formulas in a logic, such as $\mathsf{CPC}$ or $\mathsf{IPC}$. However, in some interesting cases, $L$ may simply be a set of finitary objects. For example, $L$ could consist of a sequence of inconsistent sets of clauses or formulas of a specific form within a certain fragment of a language. In such cases, and only in these cases, we place emphasis on the nature of the finitary objects used. Additionally, when working with inconsistent sets of clauses, ``proofs'' are usually referred to as \emph{refutations}.
\end{remark}

\begin{example}
The usual Hilbert-style systems, natural deduction systems, and sequent calculi (with or without the cut rule) for classical and intuitionistic logics, as well as the modal logics in Table \ref{tab:sequent_calculi}, are all proof systems for their respective logics. Since these systems are designed to be sound and complete, it suffices to observe that verifying whether a given string $u$ is a proof of another string $w$ involves checking if the sequence adheres to the local rules for constructing proofs and ultimately concludes with $w$. This verification process can be completed in polynomial time relative to the lengths of $u$ and $w$.
\end{example}

In the following, we will introduce some proof systems for si or modal logics, and many for classical logic $\mathsf{CPC}$. In each case, it is easy to verify that they are indeed proof systems for the corresponding logics.

\begin{example}[Frege Systems]
Let $L$ be a si or modal logic. An \emph{inference system} $P$ for $L$ is a finite set of inference rules, each of the form:
\begin{center}
 \AxiomC{$\phi_1$}
  \AxiomC{$\phi_2$}
  \AxiomC{$\ldots$}
  \AxiomC{$\phi_k$}
  \QuaternaryInfC{$\phi$}
 \DisplayProof
\end{center}
An inference rule with no assumptions is called an \emph{axiom}.
A $P$-proof $\pi$ of a formula $\phi$ from the assumptions in $\Gamma$ is defined as a sequence $\pi = \theta_1, \theta_2, \ldots, \theta_n$ of formulas, such that $\theta_n=\phi$ and for each $i \leq n$, either $\theta_i \in \Gamma$, or $\theta_i$ is the result of an instance of one of the rules in $P$ applied to the formulas $\theta_j$ for $j < i$. We write $\Gamma \vdash_P \phi$, when there is a $P$-proof of $\phi$ from $\Gamma$.
A \emph{Frege system} for the logic $L$ is an inference system that is both \emph{sound} and \emph{strongly complete}; that is, $\vdash_P \phi$ implies $\phi \in L$, and $\phi_1, \ldots, \phi_m \vdash_L \phi$ implies $\phi_1, \ldots, \phi_m \vdash_P \phi$.
A Frege proof system is called \emph{standard} if it is \emph{strongly sound}, meaning that if $\phi_1, \ldots, \phi_m \vdash_P \phi$, then $\phi_1, \ldots, \phi_m \vdash_L \phi$.
Frege systems are formalizations of what is typically referred to as a Hilbert-style system.
\end{example}

\begin{example}[Substitution Frege Systems]
Let $L$ be a si or modal logic. By a (standard) \emph{substitution Frege} system $P$ for $L$, we mean an inference system extending a (standard) Frege system with the following \emph{substitution rule}:
\begin{center}
  \AxiomC{$\phi(p_1, \ldots, p_n)$}
  \UnaryInfC{$\phi(\psi_1, \ldots, \psi_n)$}
 \DisplayProof 
\end{center}
Note that the substitution rule is inherently different from the rules in a Frege system.
\end{example}

\begin{example}[Bounded depth $\mathbf{LK}$]
Define $\mathbf{LK}^u$ as a sequent calculus over $\mathcal{L}^u_b$, similar to $\mathbf{LK}$, but with the conjunction, disjunction and implication rules replaced by the following:
\begin{center}
\renewcommand{\arraystretch}{1.5}
\begin{tabular}{cc}
\AxiomC{$\Gamma, \phi_a \Rightarrow \Delta$}
\RightLabel{\footnotesize{$L\bigwedge_a$}}
\UnaryInfC{$\Gamma, \bigwedge_{i \in I} \phi_i \Rightarrow \Delta$}
\DisplayProof \qquad \qquad 
&
\AxiomC{$\{ \Gamma \Rightarrow \phi_i, \Delta \}_{i \in I}$}
\RightLabel{\footnotesize{$R\bigwedge$}}
\UnaryInfC{$\Gamma \Rightarrow \bigwedge_{i \in I} \phi_i, \Delta$}
\DisplayProof
\\[3ex]
\AxiomC{$\{ \Gamma, \phi_i \Rightarrow \Delta \}_{i \in I}$}
\RightLabel{\footnotesize{$L\bigvee$}}
\UnaryInfC{$\Gamma, \bigvee_{i \in I} \phi_i \Rightarrow \Delta$}
\DisplayProof \qquad \qquad 
&
\AxiomC{$\Gamma \Rightarrow \phi_a, \Delta$}
\RightLabel{\footnotesize{$R\bigvee_a$}}
\UnaryInfC{$\Gamma \Rightarrow \bigvee_{i \in I} \phi_i, \Delta$}
\DisplayProof
\\[3ex]
\AxiomC{$\Gamma \Rightarrow \phi, \Delta$}
\RightLabel{\footnotesize{$L\neg$}}
\UnaryInfC{$\Gamma, \neg \phi \Rightarrow \Delta$}
\DisplayProof \qquad \qquad 
&
\AxiomC{$\Gamma, \phi \Rightarrow \Delta$}
\RightLabel{\footnotesize{$R\neg$}}
\UnaryInfC{$\Gamma \Rightarrow \neg \phi, \Delta$}
\DisplayProof
\end{tabular}
\end{center}
where $I$ is any finite set and $a \in I$ is arbitrary. It is straightforward to observe that $\mathbf{LK}^u$ is a proof system for $\mathsf{CPC}$ when defined over $\mathcal{L}^u_b$. Now, let $d \in \mathbb{N}$. By an $\mathbf{LK}_d$-proof, we refer to an $\mathbf{LK}^u$-proof in which the depth of every formula occurring in the proof is at most $d$. It is well known that $\mathbf{LK}_d$ is sound and complete with respect to valid sequents in which every formula has depth at most $d$. Consequently, we can consider $\mathbf{LK}_d$ as a proof system for $\mathsf{CPC}_d$, the set of all valid $\mathcal{L}^u_b$-formulas of depth bounded by $d$.
\end{example}

\begin{example}[Negation Normal $\mathbf{LK}$]
Define $\mathbf{LK}_n$ as the sequent calculus over $\mathcal{L}_b$ consisting of the following axioms:
\begin{center}
\begin{tabular}{c c c c c c c}
 \AxiomC{}
 \UnaryInfC{$p \Rightarrow p$}
 \DisplayProof 
 &
  \AxiomC{}
 \UnaryInfC{$p, \neg p \Rightarrow $}
 \DisplayProof 
 &
  \AxiomC{}
 \UnaryInfC{$ \Rightarrow p, \neg p$}
 \DisplayProof 
 &
 \small \AxiomC{}
\small \UnaryInfC{$ \bot \Rightarrow $}
 \DisplayProof 
&
\small  \AxiomC{}
\small \UnaryInfC{$ \Rightarrow \top$}
 \DisplayProof
 &
\small \AxiomC{}
\small \UnaryInfC{$ \neg \top \Rightarrow $}
 \DisplayProof 
&
\small  \AxiomC{}
\small \UnaryInfC{$ \Rightarrow \neg \bot$}
 \DisplayProof
\end{tabular}
\end{center}
along with the structural rules (including the cut), as well as the conjunction and disjunction rules of $\mathbf{LK}$. Note that implication rules are not included. We define $\mathbf{LK}_n^-$ analogously to $\mathbf{LK}_n$, but without the cut rule. Both systems are sound and complete with respect to the valid sequents consisting of formulas in negation normal form. Therefore, we can consider $\mathbf{LK}_n$ and $\mathbf{LK}_n^-$ as proof systems for $\mathsf{CPC}_n$, the set of all valid $\mathcal{L}_b$-formulas in negation normal form.
\end{example}

\begin{example}[Resolution $\mathbf{R}$]
A \emph{resolution refutation} $\pi$ for a sequence of clauses $\mathcal{C}$ is a sequence $\pi = C_1, \dots, C_n$ of clauses, where $C_n = \varnothing$, and each $C_i$ is either a member of $\mathcal{C}$ or is derived from earlier clauses $C_j$ and $C_k$ (with $j, k < i$) via the resolution inference rule:
\begin{center}
\AxiomC{$C \cup \{p\}$}
\AxiomC{$D \cup \{\neg p\}$}
\BinaryInfC{$C \cup D$}
\DisplayProof
\end{center}
For example, the following refutation:
\begin{center}
\AxiomC{$q$}
\AxiomC{$p, \neg q$}
\AxiomC{$\neg p$}
\BinaryInfC{$\neg q$}
\BinaryInfC{$\varnothing$}
\DisplayProof
\end{center}
is a refutation for the inconsistent sequence of clauses $\{p, \neg q\}$, $\{\neg p\}$, and $\{q\}$. Since resolution is complete for inconsistent sequences of clauses, we consider it a proof system for the set of inconsistent sequences of clauses. 
\end{example}

Not all proof systems are dynamic or purely logical entities, as is traditionally assumed. Some proof systems are \emph{static}, meaning they do not involve a sequence of inferential or transformational steps. Others are \emph{non-logical}, operating not on formulas or sequents but on other mathematical structures, such as polynomial equalities or linear inequalities. In the following, we present examples that illustrate both of these alternative approaches: static and non-logical, individually as well as in combination.

\begin{example}[Truth Table $\mathbf{TT}$]
The most straightforward example of a static proof system for $\mathsf{CPC}$ is the \emph{truth table}. Formally, define $\mathbf{TT}(u, w)$ to hold if the string $w \in \{0, 1\}^*$ is a code for a propositional formula $\phi$, and $u$ is the code for the truth table of $\phi$ that assigns the value $1$ to $\phi$ under all possible truth assignments to its atomic variables. It is evident that $\mathbf{TT}$ is a proof system for $\mathsf{CPC}$.
%Furthermore, since verifying whether $\pi$ is a valid truth table for $\phi$ and confirming that it evaluates $\phi$ to 1 is feasible in polynomial time with respect to the size of the table, $\mathbf{TT}$ qualifies as a proof system for $\mathsf{CPC}$.
\end{example}

\begin{example} [Cutting Planes $\mathbf{CP}$] Interpret any clause 
\[
C = \{p_{i_1}, \ldots, p_{i_m}, \neg p_{j_1}, \ldots, \neg p_{j_n}\}
\]
as the linear inequality $I_C$ defined by
$\sum_{r=1}^m x_{i_r} - \sum_{s=1}^n x_{j_s} \geq 1 - n$.
It is straightforward to see that an assignment $\bar{a} \in \{0, 1\}$ satisfies $C$ if and only if $I_C$ holds after substituting $\bar{x}$ by $\bar{a}$.
A \emph{$\mathbf{CP}$-refutation} for the sequence $\mathcal{C}$ of clauses is a sequence of the following rules, starting from the inequalities in $\{I_C \mid C \in \mathcal{C}\}$ and ending with the inequality $0 \geq 1$:
{\small
\begin{center}
\begin{tabular}{ccc}
\AxiomC{$ $}
\UnaryInfC{$x_i \geq 0$}
\DisplayProof
&
\AxiomC{$ $}
\RightLabel{\quad(Axioms)}
\UnaryInfC{$-x_i \geq -1$}
\DisplayProof
&
\AxiomC{$\sum_i a_i x_i \geq b$}
\AxiomC{$\sum_i c_i x_i \geq d$}
\RightLabel{\quad(Addition)}
\BinaryInfC{$\sum_i (a_i + c_i) x_i \geq b + d$}
\DisplayProof\\[3ex]
\end{tabular}

\begin{tabular}{ccc}
\AxiomC{$\sum_i a_i x_i \geq b$}
\RightLabel{\quad(Multiplication)}
\UnaryInfC{$\sum_i c a_i x_i \geq c b$}
\DisplayProof
\hskip 2em
\AxiomC{$\sum_i a_i x_i \geq b$}
\RightLabel{\quad(Division)}
\UnaryInfC{$\sum_i \frac{a_i}{c} x_i \geq \left\lfloor \frac{b}{c} \right\rfloor$}
\DisplayProof
\end{tabular}
\end{center}
}
where $c \geq 0$ in the multiplication rule and $c > 0$ and $c$ divides all $a_i$'s in the division rule. For instance, the following is a $\mathbf{CP}$-refutation for the set of clauses $\{p_1, \neg p_2\}, \{\neg p_1\}, \{p_2\}$:
\begin{center}
\AxiomC{$x_2 \geq 1$}
  \AxiomC{$x_1 - x_2 \geq 0$}
 \AxiomC{$-x_1 \geq 0$}
  \BinaryInfC{$-x_2 \geq 0$}
  \BinaryInfC{$0 \geq 1$}
 \DisplayProof 
\end{center}
It is known that $\mathbf{CP}$ is sound and complete for unsatisfiable sequences of clauses. Soundness is easy by the connection between $C$ and $I_C$. For completeness, the easiest proof is by simulating resolution refutations in $\mathbf{CP}$ to get completeness of $\mathbf{CP}$ from the completeness of $\mathbf{R}$ \cite{krajivcek2019proof}.
\end{example}

\begin{example}[Nullstellensatz $\mathbf{NS}$]
Let $F$ be either $\mathbb{Q}$ or a finite field. We translate any $\mathcal{L}_b$-formula $\phi(p_1, \ldots, p_n)$ into a polynomial $P_\phi \in F[x_1, \ldots, x_n]$ as follows:
\begin{align*}
P_{\top} &= 0, \quad &P_{\bot} &= 1, \quad &P_{p_i} &= 1 - x_i, \\
P_{\neg \phi} &= 1 - P_\phi, \quad &P_{\phi \vee \psi} &= P_\phi \cdot P_\psi, \quad &P_{\phi \wedge \psi} &= P_\phi + P_\psi - P_\phi \cdot P_\psi.
\end{align*}
The key property of this translation is that for any assignment $\bar{a} \in \{0, 1\}$ to the atomic variables in $\phi$, $\phi$ holds if and only if $P_\phi(\bar{a}) = 0$. Therefore, given a finite sequence $\Phi$ of $\mathcal{L}_b$-formulas, the set of polynomials
$\{P_{\phi} \mid \phi \in \Phi\} \cup \{x_i^2 - x_i \mid i \in \mathbb{N}\}$
has a common root in $F$ if and only if $\Phi$ is satisfiable.
We now define an \emph{$\mathbf{NS}$-refutation} of a sequence $\mathcal{C} = C_1, \ldots, C_r$ of clauses over variables $\{p_1, \ldots, p_s\}$ as a sequence of polynomials $g_1, \ldots, g_r, h_1, \ldots, h_s$ over $F$ such that
$\sum_{i=1}^r P_{\phi_i} g_i + \sum_{i=1}^s (x_i^2 - x_i) h_i = 1$,
where $\phi_i = \bigvee C_i$.
We claim that $\mathbf{NS}$ is a proof system for the set of unsatisfiable sequences of clauses. To establish the soundness and completeness of this system, we rely on Hilbert's Nullstellensatz, which asserts that a set $\mathcal{F}$ of polynomials over $F$ has no common root in the algebraic closure of $F$ if and only if there exist polynomials $g_1, \ldots, g_k$ over $F$ such that
$\sum_i f_i g_i = 1$
for some $f_i \in \mathcal{F}$. Since we are only concerned with Boolean assignments, passing to the algebraic closure of $F$ does not affect the correctness of the system. 
For a concrete example, let $F = \mathbb{Q}$ and consider the unsatisfiable sequence of clauses $\{p_1, \neg p_2\}, \{\neg p_1\}, \{p_2\}$. This sequence translates into the following sequence of polynomials: $(1 - x_1)x_2, x_1, 1 - x_2$. Then, the sequence $(1, x_2, 1)$ is an $\mathbf{NS}$-refutation since 
$(1 - x_1)x_2 \cdot 1 + x_1 \cdot x_2 + (1 - x_2) \cdot 1 = 1$.
\end{example}

\begin{remark}
The previous example illustrates that the soundness and completeness of a proof system may rely on a mathematical theorem, in this case Hilbert's Nullstellensatz. Motivated by this observation, one can generalize the notion of a proof system by allowing any sufficiently strong, sound mathematical first-order theory $T$ with a polynomial-time axiomatization to serve as a proof system for $\mathsf{CPC}$. Specifically, we may define a proof of a propositional formula $\phi$ as a first-order proof in $T$ of the encoding of the sentence ``the formula $\phi$ is a tautology.'' For example, taking $T = \mathsf{ZFC}$ yields a powerful proof system that permits the use of all of classical mathematics in a propositional proof.
\end{remark}

\subsection{Simulations and Relative Efficiency}
For any language $L \subseteq \{0, 1\}^*$, all proof systems for $L$ are, by definition, sound and complete. That is, they prove exactly the same set of strings, i.e., the ones in $L$. This naturally raises the question: in what sense can we meaningfully distinguish between different proof systems for the same $L$? The key lies in their \emph{efficiency}, that is, how succinct the proofs they produce are. For example, our collective experience is that the presence of the cut rule in $\mathbf{LK}$ allows for significantly shorter proofs compared to its cut-free counterpart $\mathbf{LK}^-$. This makes $\mathbf{LK}$ \emph{stronger} than $\mathbf{LK}^-$. The following formalizes this notion of relative strength:

\begin{definition}
Let $L_1 \subseteq L_2$ and let $P_1$ and $P_2$ be proof systems for $L_1$ and $L_2$, respectively. We say that $P_2$ \emph{simulates} $P_1$ and write $P_2 \geq P_1$ (or $P_1 \leq P_2$) if there exists a polynomial $p$ such that whenever $P_1(u, w)$ holds, there exists $v \in \{0, 1\}^*$ with $|v| \leq p(|u|, |w|)$ such that $P_2(v, w)$ holds, for any strings $u, w \in \{0, 1\}^*$. We say that $P_1$ and $P_2$ are \emph{equivalent} and write $P_1 \equiv P_2$ if $L_1 = L_2$, $P_1 \geq P_2$, and $P_2 \geq P_1$.
\end{definition}

\begin{remark}
The concrete proof systems we have defined so far operate over different fragments of different languages. However, there are canonical ways to translate formulas from one language into another, which sometimes allow us to map one fragment into another. For instance, a clause can naturally be viewed as a depth-one formula in $\mathcal{L}^u_b$, and arbitrary conjunctions and disjunctions in an $\mathcal{L}^u_b$-formula can be simulated using only binary connectives of $\mathcal{L}_b$. When comparing proof systems via simulations, we take such translations into account and consider the proof systems up to these transformations.
\end{remark}

Some simulations are trivial. For instance, it is clear that $\mathbf{LK}_d \leq \mathbf{LK}_{d+1}$, as we can read a proof in $\mathbf{LK}_d$ as a proof in $\mathbf{LK}_{d+1}$. Similarly, using the above-mentioned transformations of formulas, we have $\mathbf{LK}_d \leq \mathbf{LK}$, and by transforming refutations into proofs, one can easily see that $\mathbf{R} \leq \mathbf{LK}_1$. In the following, we present less trivial simulations. For more, see Figure~\ref{fig:ProofSystemFig}.

\begin{theorem}[\cite{reckhow1975lengths,jevrabek2009substitution}]\label{Reckhow}
Let $L$ be a si or modal logic. Then, any two standard Frege systems for $L$ are equivalent. The same also holds for standard substitution Frege systems.
\end{theorem}
\begin{proof}
We prove only the Frege case; the other is similar. Let $\mathbf{F}_1$ and $\mathbf{F}_2$ be two standard Frege systems for $L$. We show that $\mathbf{F}_2 \geq \mathbf{F}_1$. 
Let $R$ be a rule of inference in $\mathbf{F}_1$ that derives $\phi$ from the assumptions $\phi_1, \ldots, \phi_k$. Since $\mathbf{F}_1$ is strongly sound, we have $\phi_1, \ldots, \phi_k \vdash_L \phi$. By the strong completeness of $\mathbf{F}_2$, there exists an $\mathbf{F}_2$-proof of $\phi$ from $\phi_1, \ldots, \phi_k$. Fix such a proof and denote it by $\pi_R$. 
Let $\pi = \psi_1, \ldots, \psi_m$ be an $\mathbf{F}_1$-proof of the formula $\psi$. To construct an $\mathbf{F}_2$-proof $\pi'$ of $\psi$, proceed step-by-step: for each line $\psi_i$ of the proof, find the rule $R \in \mathbf{F}_1$ and the substitution $\sigma$ such that $\psi_i$ is obtained from earlier formulas $\psi_j$ via $R$ and $\sigma$. Then, add the $\mathbf{F}_2$-proof $\sigma(\pi_R)$ and continue with the next line of $\pi$. 
Since $\mathbf{F}_1$ has only finitely many rules, there are only finitely many $\mathbf{F}_2$-proofs $\pi_R$ whose substitutions we need. As substitutions expand a proof only polynomially, the size of $\pi'$ is polynomially bounded in the size of $\pi$. Hence, $\mathbf{F}_2 \geq \mathbf{F}_1$. Similarly, $\mathbf{F}_1 \geq \mathbf{F}_2$, which implies $\mathbf{F}_1 \equiv \mathbf{F}_2$.
\end{proof}

\begin{remark}
Since any two standard Frege systems for $L$ are equivalent, we may refer to \emph{the} standard Frege system for $L$ and denote it by \emph{$L$-Frege} or \emph{$L$-$\mathbf{F}$}. The same applies to standard substitution Frege systems for $L$, which we denote by $L$-$\mathbf{SF}$.
\end{remark}

\begin{theorem}\label{LKEquivFrege}
$\mathbf{LK}$ (resp. $\mathbf{LJ}$) and $\mathsf{CPC}$-Frege (resp. $\mathsf{IPC}$-Frege) are equivalent. The same also holds for all modal calculi of Table \ref{tab:sequent_calculi} and their standard Frege systems.
\end{theorem}
\begin{proof}
The standard proof of the equivalence between Hilbert-style systems and sequent calculi for these logics can be carried out in polynomial time, and therefore yields a polynomially bounded simulation. For more details, see \cite{krajivcek2019proof}.
\end{proof}

\begin{theorem}\label{LK=LKN}
$\mathbf{LK}_n$ and $\mathbf{LK}$ are equivalent. The same also holds for $\mathbf{LK}^-_n$ and $\mathbf{LK}^-$.
\end{theorem}
\begin{proof}
We only prove the first case, as the second is analogous.
First, observe that any $\mathbf{LK}_n$-proof is, in fact, an $\mathbf{LK}$-proof, if we read $\neg \phi$ as $\phi \to \bot$. The only difference lies in the axioms: one needs to verify that the axioms 
$(p, \neg p \Rightarrow\,)$, $(\,\Rightarrow p, \neg p)$, $(\,\Rightarrow \neg \bot)$, and $(\neg \top \Rightarrow\,)$
are provable in $\mathbf{LK}$, which is clearly the case. Note that the canonical $\mathbf{LK}$-proofs of any instance of these axioms are of polynomial size in the end sequent. Hence, the entire transformation yields a polynomially bounded $\mathbf{LK}$-proof.
For the converse direction, it suffices to transform any $\mathbf{LK}$-proof of a sequent $\Gamma_1, \Gamma_2 \Rightarrow \Delta_1, \Delta_2$
into an $\mathbf{LK}_n$-proof of the sequent
$
\Gamma^n_1, (\neg \Delta_1)^n \Rightarrow (\neg \Gamma_2)^n, \Delta^n_2,
$
where $\phi^n$ denotes the negation normal form of $\phi$.
Constructing such a transformation with polynomially bounded size is straightforward.
\end{proof}

%\begin{remark}
%For any language $L \subseteq \{0,1\}^*$ and any proof system $P$ for $L$, we can define the predicate $Q$ by setting $Q(u, w)$ to hold iff $u$ encodes a pair $(v, w)$ and $P(v, w)$ holds. It is easy to verify that $Q$ is a proof system for $L$ and that $P \equiv Q$. Since, in $Q$, we can efficiently read $w$ from its proofs, w.l.o.g., we may assume that all proof systems enjoy this useful property.
%\end{remark}

\begin{figure}[hptb] % Use placement specifiers like t, b, h, p
\centering % Center the figure

% Using previous scale/font settings
\begin{tikzpicture}[
    scale=0.5, 
    every node/.append style={font=\small}, 
    node distance=.5cm and 2cm, 
    sys/.style={ 
        draw, rectangle, thick, align=center, 
        minimum height=0.4cm, 
        minimum width=1cm,
        fill=white, inner sep=4pt
    },
    sys_small/.style={ 
        draw, rectangle, thick, align=center, 
        minimum height=0.4cm, minimum width=1cm,    
        fill=white, inner sep=2pt
    },
    edge/.style={ 
        draw, thick
    }
]

% === 1. Define ALL Content Node Positions FIRST ===
\node (zfc)  [sys]                 {$\mathbf{ZFC}$};
\node (sf)   [sys, below=of zfc]   {$\mathsf{CPC}$-$\mathbf{SF}$};
\node (f-lk) [sys, below=of sf]    {$\mathsf{CPC}$-$\mathbf{F} \equiv \mathbf{LK}$};
\node (cp)   [sys, below=of f-lk]  {$\mathbf{CP}$};
\node (ac0f) [sys, left=of cp]     {$\{\mathbf{LK}_d \leq \mathbf{LK}_{d+1}\}_{d \geq 1}$}; 
\node (ns)   [sys, right=of cp]    {$\mathbf{NS}$};
\node (tt)   [sys, below=1.4cm of cp] {$\mathbf{TT}$}; 
% Define R node position relative to AC0F/TT path during edge drawing later,
% but we need to include its anticipated position in the fit calculation.
% Let's define a coordinate for its approximate position for fitting purposes.
\coordinate (r_pos) at ($(ac0f.south)!0.5!(tt.north west)$);

% === 2. Create a Bounding Box around ALL relevant nodes ===
% Include all system nodes and the anticipated position of R
\node (bounding_box) [fit=(zfc) (sf) (f-lk) (ac0f) (cp) (ns) (tt) (r_pos), inner sep=0pt] {};

% === 3. Calculate Coordinates based on Bounding Box and specific nodes ===
% Mid level calculation remains the same
\coordinate (mid_level) at ($(f-lk.south)!0.5!(cp.north)$); 

% Use bounding_box anchors for horizontal extent + padding
\coordinate (line_start) at ([xshift=-0.7cm] bounding_box.west |- mid_level); 
\coordinate (line_end)   at ([xshift=0.7cm] bounding_box.east |- mid_level);   

% Use bounding_box anchors for background corners + padding
\coordinate (bg_top_left)     at ([yshift=0.6cm, xshift=-0.7cm] bounding_box.north west); 
\coordinate (bg_bottom_right) at ([yshift=-0.6cm, xshift=0.7cm] bounding_box.south east); 

% Mid points for rectangle division
\coordinate (bg_mid_right)    at (line_end |- mid_level);
\coordinate (bg_mid_left)     at (line_start |- mid_level);

% === 4. Draw Background Layer ===
\begin{scope}[on background layer]
    % Use the newly calculated coordinates for the fills
    \fill[myLightRed] (bg_top_left) rectangle (bg_mid_right); 
    \fill[myLightBlue] (bg_mid_left) rectangle (bg_bottom_right);
    % Draw the dashed line using the new extents
    \draw [dashed, thick] (line_start) -- (line_end);
\end{scope}

% === 5. Draw Edges and Intermediate Nodes === 
\draw [edge] (zfc.south) -- (sf.north);
\draw [edge] (sf.south)  -- (f-lk.north);
\draw [edge] (f-lk.south) -- (ac0f.north);
\draw [edge] (f-lk.south) -- (cp.north);
\draw [edge] (f-lk.south) -- (ns.north);

% Define R node properly here while drawing the edge
\draw [edge] (ac0f.south) -- node[sys_small, pos=0.5] (r) {$\mathbf{R}$} (tt.north west); 
\draw [edge] (cp.south)   -- (tt.north);
\draw [edge] (ns.south)   -- (tt.north east); 
\draw [edge] (r) -- (cp); 

% === 6. Nodes were defined in Step 1 and are drawn on top ===

\end{tikzpicture} 

\caption{Some proof systems for $\mathsf{CPC}$ \cite{krajivcek2019proof}. If there is a line between two proof systems, the higher one simulates the lower. For the proof systems in the blue region, an exponential lower bound is known. For those in the red region, it is not even known whether they are not p-bounded.} 
\label{fig:ProofSystemFig} 

\end{figure}
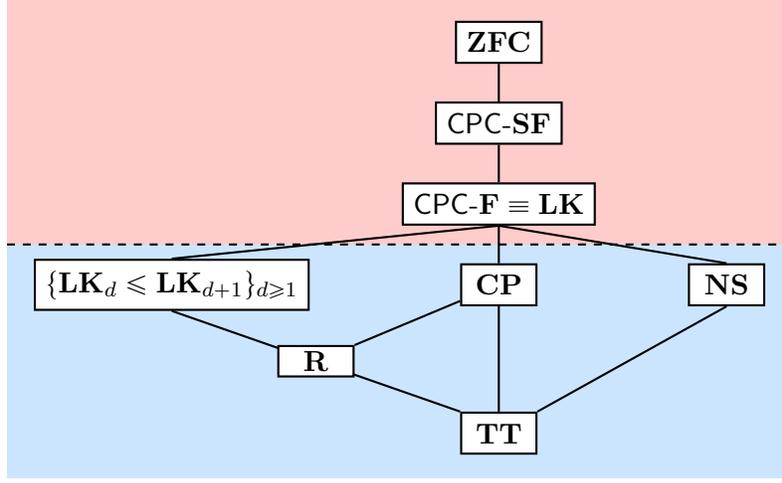

\subsection{Absolute Efficiency}
In the previous subsection, we introduced the relative power of proof systems in terms of their efficiency. A natural question is: given a language $L$, is there an \emph{absolutely efficient} proof system for $L$ in which every $w \in L$ has a \emph{short proof}? Of course, the term ``short'' must be understood relative to the length of $w$ itself, since proofs are typically at least as long as the statements they establish.

\begin{definition}
A proof system $P$ for a language $L$ is called \emph{polynomially bounded} (or \emph{p-bounded}, for short) if there exists a polynomial $p$ such that, for every $w \in L$, there exists $u \in \{0, 1\}^*$ satisfying $P(u, w)$ and $|u| \leq p(|w|)$. 
\end{definition}

\begin{theorem}
$L$ has a p-bounded proof system iff $L \in \mathbf{NP}$.   
\end{theorem}
\begin{proof}
Let $P$ be a p-bounded proof system for a language $L$. Then, there exists a polynomial $p$ such that for any $w \in \{0, 1\}^*$, we have $w \in L$ iff there exists $u \in \{0, 1\}^*$ such that $|u| \leq p(|w|)$ and $ P(u, w)$.
The right-hand side clearly defines an $\mathbf{NP}$ predicate, so we conclude that $L \in \mathbf{NP}$.
For the converse direction, suppose $L \in \mathbf{NP}$. Then there exists a polynomial-time predicate $R$ and a polynomial $p$ such that
$w \in L$ iff $\exists u \, [|u| \leq p(|w|)\wedge R(u, w)]$.
Set $P(u, w) := R(u, w) \wedge (|u| \leq p(|w|))$.
It is clear that $P$ is a p-bounded proof system for $L$.
\end{proof}

\begin{corollary}\label{CooksTheorem}
$\mathsf{CPC}$ (resp. $\mathsf{IPC}$ or any modal logic in Table \ref{tab:sequent_calculi}) has a p-bounded proof system iff $\mathbf{CoNP}=\mathbf{NP}$ (resp. $\mathbf{PSPACE}=\mathbf{NP}$).
\end{corollary}
\begin{proof}
As $\mathsf{CPC}$ is $\mathbf{CoNP}$-complete, we have $\mathsf{CPC} \in \mathbf{NP}$ if and only if $\mathbf{CoNP} = \mathbf{NP}$. For $\mathsf{IPC}$ or any modal logic in Table~\ref{tab:sequent_calculi}, the argument is similar, using the $\mathbf{PSPACE}$-completeness of these logics instead.
\end{proof}
It is widely believed that $\mathbf{NP} \neq \mathbf{CoNP}$ and $\mathbf{NP} \neq \mathbf{PSPACE}$. This implies that there can be no p-bounded proof system for $\mathsf{CPC}$, $\mathsf{IPC}$, or any modal logic in Table~\ref{tab:sequent_calculi}. However, since the non-existence of such systems is equivalent to these major complexity-theoretic conjectures, we currently lack the tools to prove it.
Of course, this does not prevent us from proving that specific, well-known proof systems for these logics (such as Resolution, Cutting Planes, or $L$-Frege) are not p-bounded. The general setting for proving that a proof system for $L \subseteq \{0, 1\}^*$ is not p-bounded is as follows. First, a definition:
\begin{definition}
Let $n_0 \in \mathbb{N}$ be a number, $L \subseteq \{0, 1\}^*$, $P$ be a proof system for $L$, and $\{w_n\}_{n \geq n_0}$ be a sequence of short binary strings in $L$. We say that the sequence $\{w_n\}_{n \geq n_0}$ has a \emph{short $P$-proof} if there exists a sequence $\{u_n\}_{n \geq n_0}$ of binary strings such that $P(u_n, w_n)$ and $|u_n| \leq p(n)$ for all $n \geq n_0$.
\end{definition}
To show that $P$ is not p-bounded, it suffices to find a sequence of short strings in $L$ with no short $P$-proofs, since in any p-bounded proof system, every sequence of short binary strings in $L$ must have a short $P$-proof. 
Arguably, the only sufficiently general technique for constructing such sequences is \emph{feasible interpolation}, as explained in the introduction. In the remainder of this chapter, we will explain this method and use it to show that some of the proof systems we have encountered so far are not p-bounded.

\section{Interpolants and their Computation} \label{Section: Interpolation}

In this section, we examine the complexity of computing a Craig interpolant for a given implication in classical logic. We then introduce alternative forms of Craig interpolation appropriate for intuitionistic logic and certain modal logics, and investigate the computational complexity of these interpolants.

\subsection{Computing Craig Interpolants in Classical Logic}

Let $\{\phi_n(\bar{p}, \bar{q}) \to \psi_n(\bar{p}, \bar{r})\}_{n \geq n_0}$ be a sequence of short classical tautologies. A sequence $\{C_n(\bar{p})\}_{n \geq n_0}$ of circuits computes\footnote{The most natural candidate for computing an interpolant is an algorithm that reads a valid implication and returns an interpolant of the implication as the output. However, this candidate is too universal and not fine-grained enough to distinguish between the implications with easy and hard interpolants to compute. Therefore, we designed the more local version presented here.} a Craig interpolant for $\{\phi_n(\bar{p}, \bar{q}) \to \psi_n(\bar{p}, \bar{r})\}_{n \geq n_0}$ when the formula $[C_n(\bar{p})]$ serves as a Craig interpolant for $\phi_n(\bar{p}, \bar{q}) \to \psi_n(\bar{p}, \bar{r})$, for all $n \geq n_0$. Spelling out, this means that for any Boolean assignment $\bar{a} \in \{0, 1\}$ to the atoms in $\bar{p}$, if $C_n(\bar{a}) = 0$, then $\neg \phi_n(\bar{a}, \bar{q})$ is a tautology, and if $C_n(\bar{a}) = 1$, then $\psi_n(\bar{a}, \bar{r})$ is a tautology. We say that $\{\phi_n(\bar{p}, \bar{q}) \to \psi_n(\bar{p}, \bar{r})\}_{n \geq n_0}$ has \emph{an easy Craig interpolant} if there is a sequence $\{C_n\}_{n \geq n_0}$ of circuits that computes an interpolant such that the size of $C_n$ is polynomially bounded by $n$, for any $n \geq n_0$. Otherwise, we say that it has \emph{hard Craig interpolants}. 

The most obvious way to compute a Craig interpolant for the sequence $\{\phi_n(\bar{p}, \bar{q}) \to \psi_n(\bar{p}, \bar{r})\}_{n \geq n_0}$ is to define the circuit $C_n(\bar{p})$ as either the formula $\bigvee_{\bar{b} \in \{0, 1\}} \phi_n(\bar{p}, \bar{b})$ or the formula $\bigwedge_{\bar{c} \in \{0, 1\}} \psi_n(\bar{p}, \bar{c})$. It is straightforward to verify that both of these formulas serve as interpolants. However, the size of these formulas grows exponentially with $n$. Naturally, we are keen to find more efficient methods to compute an interpolant, ideally using poly-size circuits. However, there are some indications that such computations may be inherently hard. Let us explain why, using disjoint $\mathbf{NP}$ pairs:

\begin{lemma}\label{DNPToInterpolant}
Let $(L, R)$ be a disjoint $\mathbf{NP}$ pair, and let $\{\phi_n(\bar{p}, \bar{q})\}_n$ and $\{\psi_n(\bar{p}, \bar{r})\}_n$ be sequences of short propositional formulas representing $L$ and $R$, respectively, as given by Theorem~\ref{NPToPropositional}. Then the sequence $\{\phi_n(\bar{p}, \bar{q}) \to \neg \psi_n(\bar{p}, \bar{r})\}_n$ consists of short tautologies. Moreover, any sequence $\{C_n\}_n$ of circuits computing Craig interpolants for this sequence also computes a separator for the pair $(L, R)$.
\end{lemma}
\begin{proof}
By construction in Theorem \ref{NPToPropositional}, $w \in L$ (resp. $w \in R$) if and only if $\phi_n(w, \bar{q})$ (resp. $\psi_n(w, \bar{r})$) is satisfiable, where $n = |w|$. Since $L \cap R = \emptyset$, it follows that no $w$ can simultaneously belong to both $L$ and $R$, for any $w \in \{0, 1\}^*$. Therefore, the conjunction $\phi_n(\bar{p}, \bar{q}) \wedge \psi_n(\bar{p}, \bar{r})$ is unsatisfiable, because any assignment that maps $\bar{p}$ to $w \in \{0, 1\}^*$ would make both $\phi_n(w, \bar{q})$ and $\psi_n(w, \bar{r})$ satisfiable, which is a contradiction. Thus, we derive a sequence $\{\phi_n(\bar{p}, \bar{q}) \to \neg \psi_n(\bar{p}, \bar{r})\}_n$ of short tautologies.
Now, suppose that $\{C_n(\bar{p})\}_n$ is a family of circuits that computes an interpolant for the sequence $\{\phi_n(\bar{p}, \bar{q}) \to \neg \psi_n(\bar{p}, \bar{r})\}_n$. Therefore, for any $w \in \{0, 1\}^n$, if $C_n(w) = 0$, then $\neg \phi_n(w, \bar{q})$ is a tautology, implying that $w \in L^c$. Similarly, if $C_n(w) = 1$, then $w \in R^c$. Therefore, the set $S = \{w \in \{0, 1\}^* \mid C_{|w|}(w) = 1\}$ acts as a separator for $(L, R)$ and $\{C_n\}_n$ computes $S$.
\end{proof}

\begin{corollary}\label{DNPImpliesHardInt}
If hard disjoint $\mathbf{NP}$ pairs exist or, in particular, one-way protocols exist, then there is a sequences of short propositional tautologies that has hard Craig interpolants.
\end{corollary}
\begin{proof}
Assuming that hard disjoint $\mathbf{NP}$ pairs exist, the claim follows immediately from Lemma~\ref{DNPToInterpolant}. For the one-way protocol variant, it suffices to apply Corollary~\ref{OneWayToDNP}.
\end{proof}

One can apply Lemma~\ref{DNPToInterpolant} to the disjoint $\mathbf{NP}$ pairs introduced earlier to obtain concrete sequences of tautologies that are believed to have hard Craig interpolants:

\begin{example}[RSA and Diffie-Hellman Tautologies] \label{RSASequenceHaveIsHardInt}
Let $\mathrm{RSA}^1_n(\bar{p}, \bar{q})$ and $\mathrm{RSA}^0_n(\bar{p}, \bar{r})$ be the propositional formulas that encode the disjoint $\mathbf{NP}$ pair $(\mathrm{RSA}_1, \mathrm{RSA}_0)$ defined in Example~\ref{ExampleOfDNPOfRSA}. Then, by Lemma~\ref{DNPToInterpolant}, we obtain the sequence $\{\mathrm{RSA}^1_n(\bar{p}, \bar{q}) \to \neg \mathrm{RSA}^0_n(\bar{p}, \bar{r})\}_n$ of short tautologies.
Similarly, if $\mathrm{DH}^1_n(\bar{p}, \bar{q})$ and $\mathrm{DH}^0_n(\bar{p}, \bar{r})$ are the propositional translations of the disjoint $\mathbf{NP}$ pair $(\mathrm{DH}_1, \mathrm{DH}_0)$, we obtain the sequence $\{\mathrm{DH}^1_n(\bar{p}, \bar{q}) \to \neg \mathrm{DH}^0_n(\bar{p}, \bar{r})\}_n$ of short tautologies.
If the RSA and Diffie-Hellman protocols are secure against $\mathbf{P/poly}$-adversaries, then the corresponding disjoint $\mathbf{NP}$ pairs must be hard, by Example~\ref{ExampleOfDNPOfRSA}. Therefore, both sequences would have hard Craig interpolants, by Lemma~\ref{DNPToInterpolant}.
\end{example}

\begin{example}[Clique-Coloring Tautologies]
Let $0 < \epsilon < 1/3$ be a real number. Using Examples \ref{ExampleClique} and \ref{ExampleColor}, define the families $\mathrm{Clique}_n^{\epsilon}(\bar{p}, \bar{q})$ and $\mathrm{Color}^{\epsilon}_n(\bar{p}, \bar{r})$ of clauses, where $\bar{p}$ has $\binom{n}{2}$ atoms and describes a simple graph $G$ on $n$ vertices, $\bar{q}$ represents a $\lfloor n^{2/3} \rfloor$-clique inside $G$, and $\bar{r}$ represents an $\lfloor n^{2/3-2\epsilon} \rfloor$-coloring of $G$. Then, for large enough $n$, $\mathrm{Clique}_n^{\epsilon}(\bar{p}, \bar{q}) \cup \mathrm{Color}^{\epsilon}_n(\bar{p}, \bar{r})$ is an unsatisfiable sequence of clauses, as $\lfloor n^{2/3} \rfloor > \lfloor n^{2/3-2\epsilon} \rfloor$. Similarly, by transforming the clauses into disjunctions and slightly abusing notation, we can write $\mathrm{Clique}_n^{\epsilon}(\bar{p}, \bar{q}) \to \neg \mathrm{Color}^{\epsilon}_n(\bar{p}, \bar{r})$ as the corresponding tautology.
\end{example}

For the Clique-Coloring tautologies, we can unconditionally show that computing a Craig interpolant is hard if we restrict our computational device to monotone circuits:

\begin{theorem}\label{LowerBoundClassicalInt}
Let $0 < \epsilon < 1/3$ be a real number and let $\{C_n\}_n$ be a sequence of monotone circuits computing a Craig interpolant for $\{\mathrm{Clique}^{\epsilon}_{n}(\bar{p}, \bar{q}) \to \neg \mathrm{Color}^{\epsilon}_{n}(\bar{p}, \bar{r})\}_n$. Then, the size of $C_n$ must be at least $2^{\Omega(n^{1/3-\epsilon})}$.
\end{theorem}
\begin{proof}
Similar to the proof of Lemma \ref{DNPToInterpolant}, it is easy to see that $C_n$ separates the pair $(\mathrm{Clique}_{\lfloor n^{2/3} \rfloor,n}, \mathrm{Color}_{\lfloor n^{2/3-2\epsilon} \rfloor,n})$. The lower bound follows from Theorem \ref{Alon-Boppana}.
\end{proof}

\subsection{Alternatives for Craig Interpolation}

In classical logic, a Craig interpolant for $(\phi(\bar{p}, \bar{q}) \to \psi(\bar{p}, \bar{r})) \equiv (\neg \phi(\bar{p}, \bar{q}) \vee \psi(\bar{p}, \bar{r}))$ can be viewed as a mechanism that, given an assignment $\bar{a}$ to the shared variables $\bar{p}$, determines whether $\neg \phi(\bar{a}, \bar{q})$ or $\psi(\bar{a}, \bar{r})$ is a tautology. Thus, computing a Craig interpolant corresponds to deciding between these two disjuncts based on the assignment $\bar{a}$, a task that we have argued is computationally hard, by leveraging the hardness of disjoint $\mathbf{NP}$ pairs.

However, in many non-classical logics, there is often some form of the \emph{disjunction property}. That is, if $\phi \vee \psi$ is derivable in the logic, then either $\phi$ or $\psi$ must be derivable. In such settings, deciding between the disjuncts no longer depends on the assignment $\bar{a}$, making this task computationally trivial.
Consequently, in non-classical contexts, we need to replace Craig interpolants with more refined notions that preserve the same intuitive purpose but are better suited to the specific logic in question. In this subsection, we follow \cite{hrubevs2009lengths} and introduce two such alternatives: one for the propositional language and one for the modal language.

\begin{definition}[Propositional Disjunctive Interpolation]
Let $L$ be a si logic. An \emph{$L$-propositional disjunctive interpolant} ($L$-$\mathrm{PDI}$, for short) for $(\phi \to \psi \vee \theta) \in L$ is a pair $(I, J)$ of formulas such that $V(I), V(J) \subseteq V(\phi)$, and the formulas $\phi \to (I \vee J)$, $I \to \phi$, and $J \to \psi$ are all in $L$. 
An \emph{$L$-monotone propositional disjunctive interpolant} ($L$-$\mathrm{mPDI}$, for short) is defined similarly, requiring additionally that $\phi$, $I$, and $J$ are all monotone.
We say that $L$ has the \emph{(monotone) propositional disjunctive interpolation property} if any $(\phi \to \psi \vee \theta) \in L$ (resp. for monotone $\phi$) has an $L$-$\mathrm{PDI}$ (resp. $L$-$\mathrm{mPDI}$).
We denote the propositional disjunctive interpolation property by $\mathrm{PDIP}$ and its monotone version by $\mathrm{mPDIP}$.\footnote{This notion is frequently employed in the literature on the proof complexity of non-classical logics as a technical tool. We believe that, like Craig interpolation, both this notion and its modal counterpart are of independent interest and deserve further logical exploration. As a first step in this direction, we propose a name for them and briefly examine their existence.} 
\end{definition}

\begin{remark}
For si logics, PDIP generalizes the disjunction property. To see why, it is enough to set $\phi = \top$. Then, both $I$ and $J$ must be variable-free, and hence they are intuitionistically equivalent to either $\top$ or $\bot$. Since $I \vee J$ is provable, one of $I$ or $J$ must be equivalent to $\top$. Therefore, either $\psi$ or $\theta$ is provable.
Using this observation, one may argue that PDIP is actually a generalization of the disjunction property rather than of Craig interpolation. This is somewhat true. However, as we will see, there is a connection between this notion and Craig interpolation. 
Therefore, it is probably fair to say that PDIP is a generalization of the classically-equivalent disjunctive form of Craig interpolation, where we work with disjunctions rather than implications.
\end{remark}

\begin{theorem}\label{thm:PDIP}
$\mathsf{IPC}$ has both $\mathrm{PDIP}$ and $\mathrm{mPDIP}$.
\end{theorem}
\begin{proof}
For $\mathrm{PDIP}$, we prove the following apparently stronger claim: For any $\mathbf{LJ}$-provable sequent $\Gamma \Rightarrow \phi \vee \psi$, there are formulas $I$ and $J$ such that $V(I), V(J) \subseteq V(\Gamma)$, and the sequents $(\Gamma \Rightarrow I \vee J)$ and $(I \Rightarrow \phi)$ and $(J \Rightarrow \psi)$ are all provable in $\mathbf{LJ}$.
The proof is by induction on a cut-free $\mathbf{LJ}$-proof of the sequent $\Gamma \Rightarrow \phi \vee \psi$. The axiom case is easy. For the rules, we only explain the cases where the last rule is $(R\vee_1)$ or $(L\vee)$. The others are similar. For $(R\vee_1)$, let the last rule be in the form:
\begin{prooftree}
    \AxiomC{$\Gamma \Rightarrow \phi $}
    \RightLabel{$R\vee_1$}
    \UnaryInfC{$\Gamma \Rightarrow \phi \vee \psi$}
\end{prooftree}
Then, it is enough to set $I=\bigwedge \Gamma$ and $J=\bot$. If the last rule is $(L\vee)$:
\begin{prooftree}
    \AxiomC{$\Gamma, \theta_1 \Rightarrow \phi \vee \psi$}
    \AxiomC{$\Gamma, \theta_2 \Rightarrow \phi \vee \psi$}
    \RightLabel{$L\vee$}
    \BinaryInfC{$\Gamma, \theta_1 \vee \theta_2 \Rightarrow \phi \vee \psi$}
\end{prooftree}
using the induction hypothesis, there are $I_i$'s and $J_i$'s, for $i \in \{1, 2\}$ such that $V(I_i), V(J_i) \subseteq V(\Gamma \cup \{\theta_i\})$ and $(\Gamma, \theta_i \Rightarrow I_i \vee J_i)$, $(I_i \Rightarrow \phi)$, $(J_i \Rightarrow \psi)$ are all provable in $\mathbf{LJ}$. Define $I=I_1 \vee I_2$ and $J=J_1 \vee J_2$. It is clear that $V(I), V(J) \subseteq V(\Gamma \cup \{\theta_1 \vee \theta_2\})$, $\mathbf{LJ} \vdash I \Rightarrow \phi$ and $\mathbf{LJ} \vdash J \Rightarrow \psi$. Moreover, it is easy to see that $(\Gamma, \theta_i \Rightarrow I \vee J)$ and hence $(\Gamma, \theta_1 \vee \theta_2 \Rightarrow I \vee J)$ is provable in $\mathbf{LJ}$. This completes the proof of the claim and hence of PDIP. For the monotone case, it is easy to see that in the above construction, whenever $\Gamma$ is monotone, so are $I$ and $J$.  
\end{proof}

\begin{definition}[Modal Disjunctive Interpolation]
Let $L$ be a modal logic. An \emph{$L$-modal disjunctive interpolant} ($L$-$\mathrm{MDI}$, for short) for $(\phi \to \Box \psi \vee \Box \theta) \in L$ is a pair $(I, J)$ of formulas such that $V(I), V(J) \subseteq V(\phi)$, and the formulas $\phi \to (I \vee J)$, $I \to \Box \phi$, and $J \to \Box \psi$ are all in $L$.
An \emph{$L$-monotone modal disjunctive interpolant} ($L$-$\mathrm{mMDI}$, for short) is defined similarly, requiring additionally that $\phi$, $I$, and $J$ are all monotone.
We say that $L$ has the \emph{(monotone) modal disjunctive interpolation property} if any $(\phi \to \Box \psi \vee \Box \theta) \in L$ (resp. for monotone $\phi$) has an $L$-$\mathrm{MDI}$ (resp. $L$-$\mathrm{mMDI}$).
We denote the modal disjunctive interpolation property by $\mathrm{MDIP}$ and its monotone version by $\mathrm{mMDIP}$.
\end{definition}

\begin{theorem}
All modal logics in Table \ref{tab:sequent_calculi} have both $\mathrm{MDIP}$ and $\mathrm{mMDIP}$.
\end{theorem}
\begin{proof}
Let $G$ be a sequent calculus in Table \ref{tab:sequent_calculi}. Then, it is enough to use an induction on a cut-free $G$-proof of a sequent $\Gamma \Rightarrow \Box \Delta, \Lambda$, to show the apparently stronger claim that for any partition of $\Delta$ into $\Delta = \Delta_1, \ldots, \Delta_k$, there are formulas $I_1, \ldots, I_k$ such that $V(I_i) \subseteq V(\Gamma \cup \Lambda)$ and the sequents $(\Gamma \Rightarrow I_1, \ldots, I_k, \Lambda)$ and $(I_i \Rightarrow \Box \Delta_i)$, for $i \leq k$, are all provable in $G$. Moreover, if $\Gamma$ is monotone and $\Lambda = \emptyset$, all $I_i$'s are also monotone. The proof of this claim is easy and the crucial point that it uses is that the modal rules in $G$ are all single-conclusion.
\end{proof}

Let $L$ be a si logic and $\{\phi_n \to \psi_n \vee \theta_n\}_{n \geq n_0}$ be a sequence of short formulas in $L$. We say that a sequence $\{(C_n, D_n)\}_{n \geq n_0}$ of pairs of circuits computes an $L$-$\mathrm{PDI}$ of $\{\phi_n \to \psi_n \vee \theta_n\}_{n \geq n_0}$ if $([C_n], [D_n])$ is an $L$-$\mathrm{PDI}$ for $\phi_n \to \psi_n \vee \theta_n$, for every $n \geq n_0$.
Similarly, if $L$ is a modal logic, for a sequence $\{\phi_n \to \Box \psi_n \vee \Box \theta_n\}_{n \geq n_0}$ of short formulas in $L$, we say that a sequence $\{(C_n, D_n)\}_{n \geq n_0}$ of pairs of $\mathcal{L}_{\Box}$-circuits computes an $L$-$\mathrm{MDI}$ of $\{\phi_n \to \Box \psi_n \vee \Box \theta_n\}_{n \geq n_0}$ if $([C_n], [D_n])$ is an $L$-$\mathrm{MDI}$ of $\phi_n \to \Box \psi_n \vee \Box \theta_n$ for every $n \geq n_0$.
Similar to the classical case, if there is a sequence of poly-size circuits for the computation, we say that the sequence $\{\phi_n \to \psi_n \vee \theta_n\}_{n \geq n_0}$ (resp. $\{\phi_n \to \Box \psi_n \vee \Box \theta_n\}_{n \geq n_0}$) has an easy $L$-$\mathrm{PDI}$ (resp. $L$-$\mathrm{MDI}$). Otherwise, we say that the sequence has hard $L$-$\mathrm{PDI}$'s (resp. $L$-$\mathrm{MDI}$'s).

In the following, we show that, similar to the case of Craig interpolants in classical logic, the disjunctive interpolants can be also hard to compute. To prove this, we introduce a way to transform classical implications to theorems of $\mathsf{IPC}$ (resp. $\mathsf{K}$) such that any computation of a propositional (resp. modal) disjunctive interpolant of the latter gives us the computation of a Craig interpolant of the former. First, let us transform the formulas:

\begin{theorem}\label{ClassicalToIPC}
If $\mathsf{CPC} \vdash \phi(\bar{p}, \bar{r}) \to \psi(\bar{p}, \bar{s})$, then  
\begin{itemize}
    \item[$\bullet$] 
$\mathsf{IPC} \vdash \bigwedge_i (p_i \vee \neg p_i) \to (\neg \phi(\bar{p}, \bar{r}) \vee \neg \neg \psi(\bar{p}, \bar{s}))$, and
    \item[$\bullet$] 
$\mathsf{K}   \vdash \bigwedge_i (\Box p_i \vee \Box \neg p_i) \to (\Box \neg \phi(\bar{p}, \bar{r}) \vee \Box \psi(\bar{p}, \bar{s})).$   
\end{itemize}
If either $\phi(\bar{p}, \bar{r})$ or $\psi(\bar{p}, \bar{s})$ is monotone in $\bar{p}$, then  
\begin{itemize}
    \item[$\bullet$] 
$\mathsf{IPC} \vdash \bigwedge_i (p_i \vee q_i) \to (\neg \phi(\neg\bar{p}, \bar{r}) \vee \neg \neg \psi(\bar{q}, \bar{s}))$, and
     \item[$\bullet$] 
$\mathsf{K}   \vdash \bigwedge_i (\Box p_i \vee \Box q_i) \to (\Box \neg \phi(\neg \bar{p}, \bar{r}) \vee \Box \psi(\bar{q}, \bar{s}))$
\end{itemize}    
\end{theorem}
\begin{proof}
We only prove the intuitionistic case; the modal case is similar. Let $\bar{p} = p_1, \ldots, p_m$. Then, for any $I \subseteq \{1, \ldots, m\}$, it is easy to see that either
$
\bigwedge_{i \in I} p_i \wedge \bigwedge_{i \notin I} \neg p_i \to \neg \phi_n(\bar{p}, \bar{r})$ or $\bigwedge_{i \in I} p_i \wedge \bigwedge_{i \notin I} \neg p_i \to \psi_n(\bar{p}, \bar{s})$
is a classical tautology. By Glivenko's theorem, we can see that either
$
\bigwedge_{i \in I} p_i \wedge \bigwedge_{i \notin I} \neg p_i \to \neg \phi_n(\bar{p}, \bar{r})$ or $\bigwedge_{i \in I} p_i \wedge \bigwedge_{i \notin I} \neg p_i \to \neg\neg \psi_n(\bar{p}, \bar{s})$
is an intuitionistic tautology, which implies
\[
\mathsf{IPC} \vdash \bigwedge_{i=1}^m (p_i \vee \neg p_i) \to \neg \phi_n(\bar{p}, \bar{r}) \vee \neg\neg \psi_n(\bar{p}, \bar{s}).
\]
For the monotone case, we explain only the situation where $\psi$ is monotone in $\bar{p}$. By monotonicity, the formula
$
\bigwedge_{i=1}^m (\neg p_i \to q_i) \to (\psi(\neg\bar{p}, \bar{s}) \to \psi(\bar{q}, \bar{s}))
$
is valid. Therefore,
$
\bigwedge_{i=1}^m (p_i \vee q_i) \to (\psi(\neg\bar{p}, \bar{s}) \to \psi(\bar{q}, \bar{s}))
$
is also valid. As $\phi(\bar{p}, \bar{r}) \to \psi(\bar{p}, \bar{s})$ is valid, we obtain
$
\bigwedge_{i=1}^m (p_i \vee q_i) \to (\phi(\neg\bar{p}, \bar{r}) \to \psi(\bar{q}, \bar{s}))
$
as valid.
Now, for any $I \subseteq \{1, \ldots, m\}$, since $\bigwedge_{i \in I} p_i \wedge \bigwedge_{i \notin I} q_i \to \bigwedge_{i} (p_i \vee q_i)$ holds, we also have the validity of
$
\bigwedge_{i \in I} p_i \wedge \bigwedge_{i \notin I} q_i \to (\phi(\neg\bar{p}, \bar{r}) \to \psi(\bar{q}, \bar{s})).
$
Therefore, either
$
\bigwedge_{i \in I} p_i \to \neg \phi(\neg\bar{p}, \bar{r})
\quad \text{or} \quad
\bigwedge_{i \notin I} q_i \to \psi(\bar{q}, \bar{s})
$
is valid. By Glivenko's theorem, we can see that either
$
\bigwedge_{i \in I} p_i \to \neg \phi(\neg\bar{p}, \bar{r})$ or $\bigwedge_{i \notin I} q_i \to \neg\neg \psi(\bar{q}, \bar{s})
$
is provable in $\mathsf{IPC}$. Hence,
$
\mathsf{IPC} \vdash \bigwedge_{i=1}^m (p_i \vee q_i) \to (\neg \phi(\neg\bar{p}, \bar{r}) \vee \neg\neg \psi(\bar{q}, \bar{s})).
$
\end{proof}

\begin{corollary}
For any $0 < \epsilon < 1/3$, we have:
\begin{itemize}
    \item[$\bullet$] 
$\mathsf{IPC} \vdash \bigwedge_i (p_i \vee q_i) \to \neg \mathrm{Clique}^{\epsilon}_n(\neg \bar{p}, \bar{r}) \vee \neg \mathrm{Color}_n^{\epsilon}(\bar{q}, \bar{s})$, and
    \item[$\bullet$] 
$\mathsf{K} \vdash \bigwedge_i (\Box p_i \vee \Box q_i) \to \Box \neg \mathrm{Clique}^{\epsilon}_n(\neg \bar{p}, \bar{r}) \vee \Box \neg \mathrm{Color}_n^{\epsilon}(\bar{q}, \bar{s})$.
\end{itemize}
\end{corollary}
\begin{proof}
The formula $\mathrm{Clique}^{\epsilon}_n(\bar{p}, \bar{r})$ is monotone in $\bar{p}$, by Example \ref{ExampleClique}. Now, it is enough to use Theorem \ref{ClassicalToIPC}.
\end{proof}

\begin{theorem}\label{CraigToDisjunctiveInt}
Let $L_1$ be a si and $L_2$ be a consistent modal logic and $\{\phi_n(\bar{p}, \bar{r}) \to \psi_n(\bar{p}, \bar{s})\}_n$ be a sequence of short classical tautologies that has hard Craig interpolants. Then, the sequence 
\[
\{\bigwedge_i (p_i \vee \neg p_i) \to (\neg \phi_n(\bar{p}, \bar{r}) \vee \neg \neg \psi_n(\bar{p}, \bar{s}))\}_n
\]
have hard $L_1$-$\mathrm{PDI}$'s and the sequence 
\[
\{\bigwedge_i (\Box p_i \vee \Box \neg p_i) \to (\Box \neg \phi_n(\bar{p}, \bar{r}) \vee \Box \psi_n(\bar{p}, \bar{s}))\}_n
\]
have hard $L_2$-$\mathrm{MDI}$'s. 
\end{theorem}
\begin{proof}
For the si case, if $(C_n(\bar{p}), D_n(\bar{p}))$ computes an $L_1$-$\mathrm{PDI}$ for the formula
\[
\bigwedge_{i}(p_i \vee \neg p_i) \to \neg \phi_n(\bar{p}, \bar{r}) \vee \neg \neg \psi_n(\bar{p}, \bar{s}),
\]
then the formulas
$
\bigwedge_{i}(p_i \vee \neg p_i) \to [C_n(\bar{p})] \vee [D_n(\bar{p})]$, $ [C_n(\bar{p})] \to \neg \phi(\bar{p}, \bar{r})$ and $[D_n(\bar{p})] \to \neg \neg \psi(\bar{p}, \bar{s})$
are in $L_1 \subseteq \mathsf{CPC}$ and hence valid. We claim that $D_n(\bar{p})$ computes a classical Craig interpolant for $\phi_n(\bar{p}, \bar{r}) \to \psi_n(\bar{p}, \bar{s})$. The reason is that, as 
\[
\bigwedge_{i}(p_i \vee \neg p_i) \to [C_n(\bar{p})] \vee [D_n(\bar{p})]
\]
is valid, so is $[C_n(\bar{p})] \vee [D_n(\bar{p})]$. Therefore, $\neg [C_n(\bar{p})] \to [D_n(\bar{p})]$ is valid.
Moreover, we have the validity of $\phi_n(\bar{p}, \bar{r}) \to \neg [C_n(\bar{p})]$, which implies the validity of $\phi_n(\bar{p}, \bar{r}) \to [D_n(\bar{p})]$. As $[D_n(\bar{p})] \to \neg \neg \psi(\bar{p}, \bar{s})$ and hence $[D_n(\bar{p})] \to \psi(\bar{p}, \bar{s})$ is valid, the proof is complete. Finally, if $C_n$ and $D_n$ are of polynomial size in $n$, we can compute a classical Craig interpolant for $\{\phi_n(\bar{p}, \bar{r}) \to \psi_n(\bar{p}, \bar{s})\}_n$ by a sequence $\{D_n\}_n$ of poly-size circuits, which is a contradiction.

For the modal case, if $(C_n(\bar{p}), D_n(\bar{p}))$ computes an $L_2$-$\mathrm{MDI}$ for the formula
\[
\bigwedge_i (\Box p_i \vee \Box \neg p_i) \to (\Box \neg \phi_n(\bar{p}, \bar{r}) \vee \Box \psi_n(\bar{p}, \bar{s})),
\]
then the formulas
$
\bigwedge_i(\Box p_i \vee \Box \neg p_i) \to [C_n(\bar{p})] \vee [D_n(\bar{p})]$, $[C_n(\bar{p})] \to \Box \neg \phi_n(\bar{p}, \bar{r})$, and $[D_n(\bar{p})] \to \Box \psi_n(\bar{p}, \bar{s})$
are all in $L_2$.
As $L_2$ is consistent, it is a subset of either the logic $\mathsf{CPC}+\Box \theta \leftrightarrow \theta$ or the logic $\mathsf{CPC}+\Box \theta$. For the first case, by
applying the forgetful translation and the conservativity of $\mathsf{CPC}+\Box \theta \leftrightarrow \theta$ over $\mathsf{CPC}$, we get the validity of $[C^f_n(\bar{p})] \vee [D^f_n(\bar{p})]$, $[C^f_n(\bar{p})] \to \neg \phi_n(\bar{p}, \bar{r})$, and $[D_n^f(\bar{p})] \to \psi_n(\bar{p}, \bar{s})$. Similar to the previous case, it is clear that $D_n^f(\bar{p})$ computes a classical Craig interpolant for $\phi_n(\bar{p}, \bar{r}) \to \psi_n(\bar{p}, \bar{s})$. In the other case, we do the same, but use the collapse translation to show that $D_n^c(\bar{p})$ computes a classical Craig interpolant for $\phi_n(\bar{p}, \bar{r}) \to \psi_n(\bar{p}, \bar{s})$. As the sizes of both $D_n^f(\bar{p})$ and $D_n^c(\bar{p})$ are less than or equal to the size of $D_n$, if the size of $C_n$ and $D_n$ are polynomially bounded in $n$, then, we can compute a classical Craig interpolant for $\{\phi_n(\bar{p}, \bar{r}) \to \psi_n(\bar{p}, \bar{s})\}_n$ by a sequence of poly-size circuits which is a contradiction.
\end{proof}

\begin{corollary}
Let $L_1$ be an si logic, and $L_2$ be a consistent modal logic. If hard disjoint $\mathbf{NP}$ pairs exist or, in particular, one-way protocols exist, then there is a sequence of short $\mathsf{IPC}$-tautologies (resp. $\mathsf{K}$-tautologies) with hard $L_1$-$\mathrm{PDI}$'s (resp. $L_2$-$\mathrm{MDI}$'s). 
\end{corollary}

As with classical Craig interpolants, restricting to monotone circuits yields an unconditional hardness result:

\begin{theorem}\label{Clique-ColorForIPC}
Let $0 < \epsilon < 1/3$ be a real number, $L_1$ be an si logic, and $L_2$ be a consistent modal logic, and let $\{(C_n(\bar{p}, \bar{q}), D_n(\bar{p}, \bar{q}))\}_n$ be a sequence of pairs of monotone circuits (resp. $\mathcal{L}_{\Box}$-circuits) computing an $L_1$-$\mathrm{PDI}$ (resp. $L_2$-$\mathrm{MDI}$) of 
$
\bigwedge_i (p_i \vee q_i) \to \neg \mathrm{Clique}^{\epsilon}_n(\neg \bar{p}, \bar{r}) \vee \neg \mathrm{Color}_n^{\epsilon}(\bar{q}, \bar{s})
$
(resp. 
$
\bigwedge_i (\Box p_i \vee \Box q_i) \to \Box \neg \mathrm{Clique}^{\epsilon}_n(\neg \bar{p}, \bar{r}) \vee \Box \neg \mathrm{Color}_n^{\epsilon}(\bar{q}, \bar{s})
$).
Then, the size of $D_n$ is at least $2^{\Omega(n^{1/3-\epsilon})}$.
\end{theorem}
\begin{proof}
For the si case, similar to the proof of Theorem \ref{CraigToDisjunctiveInt}, it is easy to show that $\{D_n(\bar{\top}, \bar{p})\}_n$ computes a Craig interpolant for $\{\mathrm{Clique}^{\epsilon}_n(\bar{p}, \bar{r}) \to \neg \mathrm{Color}^{\epsilon}_n(\bar{p}, \bar{s})\}_n$. As $D_n$ is monotone, so is $D_n(\bar{\top}, \bar{p})$. As the size of $D_n(\bar{\top}, \bar{p})$ is the same as the size of $D_n$, it is enough to use Theorem \ref{LowerBoundClassicalInt} to prove the lower bound.
For the modal case, we claim that either $\{D^f_n(\bar{\top}, \bar{p})\}_n$ or 
$\{D^c_n(\bar{\top}, \bar{p})\}_n$ computes a Craig interpolant for $\{\mathrm{Clique}^{\epsilon}_n(\bar{p}, \bar{r}) \to \neg \mathrm{Color}^{\epsilon}_n(\bar{p}, \bar{s})\}_n$. 
To prove the claim, as any consistent modal logic is either a subset of $\mathsf{CPC}+\Box \theta \leftrightarrow \theta$ or $\mathsf{CPC}+\Box \theta$, there are two cases to consider:
If $L_2 \subseteq \mathsf{CPC}+\Box \theta \leftrightarrow \theta$, then using an argument similar to that of the proof of Theorem \ref{CraigToDisjunctiveInt}, we can see that $\{D^f_n(\bar{\top}, \bar{p})\}_n$ computes a Craig interpolant for $\{\mathrm{Clique}^{\epsilon}_n(\bar{p}, \bar{r}) \to \neg \mathrm{Color}^{\epsilon}_n(\bar{p}, \bar{s})\}_n$. If $L_2 \subseteq \mathsf{CPC}+\Box \theta$, then $\{D^c_n(\bar{\top}, \bar{p})\}_n$ does the same task. As $D^f_n(\bar{\top}, \bar{p})$ or 
$D^c_n(\bar{\top}, \bar{p})$ are monotone and their size is smaller than or equal to that of $D_n$, the theorem follows from Theorem \ref{LowerBoundClassicalInt}.
\end{proof}

\section{Feasible Interpolation (Classical)} \label{Sec: FI, Classical}

In the previous section, we have seen that computing a Craig interpolant can be a hard task. In such computation, the only information available to us is that the implication in question is a classical tautology. However, the computation may get significantly easier if we also have access to a proof of the implication.
To get an idea, consider the cut-free proof system $\mathbf{LK}^-$. Given a proof in this system, it is relatively straightforward to \emph{extract} an interpolant from the proof using the standard Maehara method. This observation leads to the following definition:

\begin{definition}[Krajíček \cite{krajivcek1997interpolation}]
A proof system $P$ for $\mathsf{CPC}$ is said to have \emph{(monotone) feasible interpolation} if there is a polynomial $s$ such that for any $P$-proof $u \in \{0, 1\}^*$ of an implication $\phi(\bar{p}, \bar{q}) \to \psi(\bar{p}, \bar{r})$ with the code $w$, (where $\phi$ or $\psi$ is monotone in $\bar{p}$), there is a (monotone) circuit $C$ of size bounded by $s(|u|, |w|)$ such that $[C]$ is a Craig interpolant for $\phi(\bar{p}, \bar{q}) \to \psi(\bar{p}, \bar{r})$.
\end{definition}

As discussed in the introduction, feasible interpolation for a proof system $P$ can help to show that $P$ is not p-bounded. The main idea is that if we know that computing a Craig interpolant is hard (e.g., using the existence of a disjoint $\mathbf{NP}$ pair) and if it gets easier with access to $P$-proofs, the $P$-proofs must be too long! This way, feasible interpolation transforms the harder task of finding a lower bound for proof length into a relatively easier task of finding an upper bound for the circuit size for Craig interpolants extracted from the proofs. 

\begin{theorem}\label{FILowerbound}
Let $P$ be a proof system for $\mathsf{CPC}$ that has feasible interpolation. Then, if $\{\phi_n \to \psi_n\}_n$ is a sequence of short tautologies that has hard Craig interpolants, then $\{\phi_n \to \psi_n\}_n$ does not have short $P$-proofs. 
\end{theorem}
\begin{proof}
Let $\{\pi_n\}_n$ be a sequence of short $P$-proofs of $\{\phi_n \to \psi_n\}_n$. Using feasible interpolation of $P$ on $\pi_n$'s, we get a sequence $\{C_n\}_n$ of circuits with size bounded by $(|\pi_n|+|\phi_n \to \psi_n|)^{O(1)} \leq n^{O(1)}$ to compute a Craig interpolant for $\{\phi_n \to \psi_n\}_n$, which is impossible.
\end{proof}

\begin{corollary}\label{MainCorollaryOfFI}
If hard disjoint $\mathbf{NP}$ pairs exist or in particular, if one-way protocols exist, any proof system for $\mathsf{CPC}$ with feasible interpolation is not p-bounded.
\end{corollary}
\begin{proof}
It follows from Corollary \ref{DNPImpliesHardInt} and Theorem \ref{FILowerbound}.
\end{proof}

\begin{corollary}
If $\mathbf{NP} \nsubseteq \mathbf{P/Poly}$, then no proof system for $\mathsf{CPC}$ with feasible interpolation is p-bounded.
\end{corollary}
\begin{proof}
Assume that $P$ is a p-bounded proof system for $\mathsf{CPC}$ that has feasible interpolation. By Corollary \ref{CooksTheorem}, we have $\mathbf{NP} = \mathbf{CoNP}$. Hence, $\mathbf{NP} \cap \mathbf{CoNP} = \mathbf{NP} \nsubseteq \mathbf{P/Poly}$. This implies the existence of a hard disjoint $\mathbf{NP}$ pair. Now, use Corollary \ref{MainCorollaryOfFI} to reach a contradiction.
\end{proof}

For the monotone case, as we have a lower bound for monotone circuits computing a Craig interpolant for Clique-Coloring tautologies by Theorem \ref{LowerBoundClassicalInt}, we can prove an unconditional lower bound for the proof length:

\begin{theorem}\label{MainLowerBoundTheorem}
Let $0 < \epsilon < 1/3$ be a real number. If a proof system $P$ for $\mathsf{CPC}$ has monotone feasible interpolation, then any $P$-proof of $\mathrm{Clique}^{\epsilon}_n(\bar{p}, \bar{q}) \to \neg \mathrm{Color}^{\epsilon}_n(\bar{p}, \bar{r})$ has size at least $2^{\Omega(n^{1/3-\epsilon})}$. Hence, no proof system for $\mathsf{CPC}$ with monotone feasible interpolation is p-bounded.
\end{theorem}
\begin{proof}
The argument is similar to that of Theorem \ref{FILowerbound}.
\end{proof}

In the next subsection, we will use Theorem~\ref{MainLowerBoundTheorem} to show that some of the proof systems for~$\mathsf{CPC}$ introduced earlier are not p-bounded. Before proceeding in that direction, let us first highlight another interesting aspect related to feasible interpolation. As we will see later, there are certain proof systems for which we believe feasible interpolation does not hold. Since feasible interpolation is a key technique for proving lower bounds, its absence might initially seem like a setback. However, the lack of feasible interpolation actually reveals an interesting property of a proof system: its \emph{non-automatability}:

\begin{definition}[Automatability]
A proof system $P$ for $\mathsf{CPC}$ is said to be \emph{automatable} if there exists an algorithm $M$ and a polynomial $k$ such that, reading any classical tautology $\phi$, $M$ finds a $P$-proof of $\phi$ in time $k(|\pi_{\phi}|, |\phi|)$, where $\pi_{\phi}$ denotes the shortest $P$-proof of $\phi$.
\end{definition}

\begin{remark}
To automate a proof system $P$ for $\mathsf{CPC}$, we naturally seek an algorithm that, given a tautology~$\phi$, produces a $P$-proof of~$\phi$. Ideally, we would like this algorithm to run in polynomial time. However, any such algorithm must spend at least $|\pi_{\phi}|$ steps, which may be very large. Therefore, the best we can reasonably ask for is that the algorithm runs in time polynomial in both the lengths of~$\pi_{\phi}$ and~$\phi$, and not just in the length of~$\phi$.
\end{remark}

To state the connection between feasible interpolation and automatability, we need the following definition:

\begin{definition}
A proof system $P$ for $\mathsf{CPC}$ is called \emph{substitutable} if there is a polynomial $l$ such that for any $P$-proof $u \in \{0, 1\}^*$ of $\phi(\bar{p}, \bar{q})$ and any assignment $\bar{a} \in \{0, 1\}$, there is a $P$-proof for $\phi(\bar{p}, \bar{a})$ of size bounded by $l(|u|, |\phi|)$. The proof system $P$ is called \emph{closed under variable-free modus ponens}, if there is a polynomial $m$ such that for any variable-free tautology $\phi$ and any $P$-proof $u \in \{0, 1\}^*$ of $\phi \to \psi$, there is a $P$-proof for $\psi$ of size bounded by $m(|u|, |\phi \to \psi|)$.
\end{definition}

It is easy to see that the proof systems $\mathbf{LK}_d$, $\mathbf{LK}$ and $\mathsf{CPC}$-$\mathbf{SF}$ are substitutable and closed under variable-free modus ponens, for any $d \geq 0$.

\begin{theorem}\label{AutoToFI}
Let $P$ be a substitutable and closed under variable-free modus ponens proof system for $\mathsf{CPC}$. If $P$ does not have feasible interpolation, then it is not automatable.
\end{theorem}
\begin{proof}
Let $P$ be automatable and $M$ and $k$ be the corresponding algorithm and polynomial. Moreover, let $l$ and $m$ be the polynomials witnessing that $P$ is substitutable and closed under variable-free modus ponens, respectively. W.l.o.g., we can assume that $k$, $l$, and $m$ are all monotone.
We first prove that for any $P$-proof $u$ of a tautology in the form $\phi(\bar{p}, \bar{q}) \to \psi(\bar{p}, \bar{r})$ with the code $w$ and any assignment $\bar{a} \in \{0, 1\}$ for $\bar{p}$, if $\phi(\bar{a}, \bar{q})$ is satisfiable, then $\psi(\bar{a}, \bar{r})$ has a $P$-proof of size bounded by $m(l(|u|, |w|), |w|)$. To prove that, let $\bar{b} \in \{0, 1\}$ be a satisfying assignment for $\phi(\bar{a}, \bar{q})$. Then, as $P$ is substitutable, there is a $P$-proof of $\phi(\bar{a}, \bar{b}) \to \psi(\bar{a}, \bar{r})$ of size bounded by $l(|u|, |w|)$. Then, as $\phi(\bar{a}, \bar{b})$ is a variable-free tautology and $P$ is closed under variable-free modus ponens, there is a $P$-proof for $\psi(\bar{a}, \bar{r})$ of size bounded by $m(l(|u|, |w|), |w|)$. This completes the proof of the claim.

Now, define the algorithm $N$ as follows. It reads $u$, $w$, and $v$ as three inputs. There are two cases. First, if $u$ is a $P$-proof of a formula in the form $\phi(\bar{p}, \bar{q}) \to \psi(\bar{p}, \bar{r})$ with the code $w$ and $v$ is an assignment $\bar{a} \in \{0, 1\}$ for $\bar{p}$, then $N$ runs $M$ on $\psi(\bar{a}, \bar{r})$ for $k(m(l(|u|, |w|), |w|), |w|)$ many steps. If $M$ halts before this bound, $N$ returns one; otherwise, zero. Second, if the inputs are not in the mentioned form, $N$ outputs zero. It is clear that $N$ runs in polynomial time. We claim that in the first case, $N$ outputs the value of a Craig interpolant for the implication $\phi(\bar{p}, \bar{q}) \to \psi(\bar{p}, \bar{r})$, i.e., that if $N$ outputs one, then $\psi(\bar{a}, \bar{r})$ is a tautology, and if it outputs zero, then $\neg \phi(\bar{a}, \bar{q})$ is a tautology.
To prove this, if $M$ halts in $k(m(l(|u|, |w|), |w|), |w|)$ many steps, it finds a $P$-proof for $\psi(\bar{a}, \bar{r})$, which implies that $\psi(\bar{a}, \bar{r})$ is a tautology. If it returns zero, we claim that $\psi(\bar{a}, \bar{r})$ does not have a $P$-proof of size bounded by $m(l(|u|, |w|), |w|)$. There are two cases to consider. If $\psi(\bar{a}, \bar{r})$ is not a tautology, it has no $P$-proof and we are done. If it is a tautology, then the time $M$ needs to halt on $\psi(\bar{a}, \bar{r})$ is bounded by $k(|\pi_{\psi(\bar{a}, \bar{r})}|, |w|)$, where $\pi_{\psi(\bar{a}, \bar{r})}$ is the shortest $P$-proof of $\psi(\bar{a}, \bar{r})$.
If a $P$-proof of $\psi(\bar{a}, \bar{r})$ with length bounded by $m(l(|u|, |w|), |w|)$ exists, we must have $|\pi_{\psi(\bar{a}, \bar{r})}| \leq m(l(|u|, |w|), |w|)$, which implies $k(|\pi_{\psi(\bar{a}, \bar{r})}|, |w|) \leq k(m(l(|u|, |w|), |w|), |w|)$. This means that $M$ would have halted before reaching $k(m(l(|u|, |w|), |w|), |w|)$ many steps, which is a contradiction. Now, as $\psi(\bar{a}, \bar{r})$ does not have a $P$-proof of size bounded by $m(l(|u|, |w|), |w|)$, by the above observation, we conclude that $\phi(\bar{a}, \bar{q})$ is unsatisfiable. Hence, $\neg \phi(\bar{a}, \bar{q})$ is a tautology. This completes the proof that $N$ computes a Craig interpolant of $\phi(\bar{p}, \bar{q}) \to \psi(\bar{p}, \bar{r})$.

Finally, as $N$ runs in polynomial time, there is a sequence $\{C_n\}_n$ of poly-size circuits such that $N(u, w, v) = C_{|u| + |w| + |v|}(u, w, v)$. For any $P$-proof $u$ of $\phi(\bar{p}, \bar{q}) \to \psi(\bar{p}, \bar{r})$ with the code $w$ and $n$ many atoms in $\bar{p}$, define the circuit $D_{u, w}$ by $D_{u, w}(x_1, \ldots, x_n) = C_{|u| + |w| + n}(u, w, x_1, \ldots, x_n)$. It is clear that $[D_{u, w}]$ is an interpolant of $\phi(\bar{p}, \bar{q}) \to \psi(\bar{p}, \bar{r})$ and as $n \leq |w|$, its size is bounded by $(|u| + |w|)^{O(1)}$. Hence, $P$ has feasible interpolation.
\end{proof}

To use Theorem \ref{AutoToFI}, we need a way to prove that a proof system $P$ for $\mathsf{CPC}$ fails to have feasible interpolation. For that purpose, it is enough to use Theorem \ref{FILowerbound} in a backward manner. Combining with Theorem \ref{AutoToFI}, we have:

\begin{corollary}\label{MainCorForNoFI}
Let $P$ be a proof system for $\mathsf{CPC}$ and $\{\phi_n \to \psi_n\}_n$ be a sequence of short tautologies that has hard Craig interpolants. Then, if $\{\phi_n \to \psi_n\}_n$ has a short $P$-proof, $P$ does not have feasible interpolation. If $P$ is also substitutable and closed under variable-free modus ponens, then it is not automotable.
\end{corollary}

Note that Corollary \ref{MainCorForNoFI} suggests that if proof systems become strong enough to provide short proofs for implications with hard Craig interpolants, then feasible interpolation starts to fail. In the next subsection, we will present some instances of such strong proof systems that are believed to lack feasible interpolation and, as they are substitutable and closed under variable-free modus ponens, also non-automatable.

\subsection{Applications}
In this subsection, we apply the concept of feasible interpolation to demonstrate that certain weak concrete proof systems are not p-bounded. We also use the failure of feasible interpolation to show that some strong concrete proof systems are non-automatable.

\begin{theorem}[Krajíček \cite{krajivcek1997interpolation}] \label{FIForLK}
$\mathbf{LK}_n^-$ has both feasible and monotone feasible interpolation. 
\end{theorem}
\begin{proof}
By recursion on an $\mathbf{LK}^-_n$-proof $\pi$ of $\Gamma \Rightarrow \Delta$, one can easily construct a circuit $C$ on the variables in $V(\Gamma) \cap V(\Delta)$ such that both $\Gamma \Rightarrow [C]$ and $[C] \Rightarrow \Delta$ are valid, and the size of $C$ is bounded by $|\pi|^{O(1)}$. Since the system $\mathbf{LK}^-_n$ lacks negation and implication rules, it is clear that if $\Gamma(\bar{p}, \bar{q})$ or $\Delta(\bar{p}, \bar{r})$ is monotone in $\bar{p}$, then the constructed circuit will only use conjunctions and disjunctions over the atoms in $\bar{p}$. Hence, $C$ will be monotone.
\end{proof}

One can use the (monotone) feasible interpolation for $\mathbf{LK}^-_n$ to prove the same for resolution. It is enough to simulate $\mathbf{R}$ by $\mathbf{LK}_n^-$. For any clause $C$ in variables $\bar{p} \cup \bar{q}$, define $C^{\bar{p}}$ as the sub-clause of $C$ consisting of the literals constructed from $\bar{p}$, and let $d$ be the dualization operator, i.e., $d(p)=\neg p$, $d(\neg p)=p$, and $C^d=\{l^d \mid l \in C\}$ for any clause $C$. Then:

\begin{lemma}[Krajíček \cite{krajivcek1997interpolation}] \label{SimulationOfRByLK} 
There is a polynomial $P(n)$ such that for any resolution refutation $\pi$ of the clause $E$ from the clauses $\{C_i(\bar{p}, \bar{q})\}_i \cup \{D_j(\bar{p}, \bar{r})\}_j$, there is an $\mathbf{LK}_n^-$-proof of size at most $P(|\pi|)$ of a sequent in the form
$
\Gamma, d(E^{\bar{q}}) \Rightarrow \Delta, E^{\bar{p},\bar{r}},
$
where $\Gamma$ consists of the formulas $\bigvee C_i$ or unsatisfiable formulas in $\bar{q}$, and $\Delta$ consists of the formulas $\bigwedge \neg D_j$ or tautologies in $\bar{p}$ and $\bar{r}$.
\end{lemma}
\begin{proof}
The construction of the claimed $\mathbf{LK}^-_n$-proof is by recursion on the structure of $\pi$. If $\pi$ is just a clause $C_i(\bar{p}, \bar{q})$ or $D_j(\bar{p}, \bar{r})$, it is enough to use the canonical $\mathbf{LK}^-_n$-proofs for either $(\bigvee C_i, d(C^{\bar{q}}_i) \Rightarrow C_i^{\bar{p}})$ or $(\, \Rightarrow D_j, \bigwedge \neg D_j)$. 
For the resolution rule, if the last rule in $\pi$ has the form:
\begin{prooftree}
    \AxiomC{$C \cup \{s\}$}
    \AxiomC{$D \cup \{\neg s\}$}
    \BinaryInfC{$C \cup D$}
\end{prooftree}
then there are two cases to consider: either $s \in \bar{q}$ or $s \in \bar{p} \cup \bar{r}$. 
In the first case, by the induction hypothesis, we have $\mathbf{LK}_n^-$-proofs of $\Gamma, d(C^{\bar{q}}), \neg s \Rightarrow C^{\bar{p}, \bar{r}}, \Delta$ and $\Gamma', d(D^{\bar{q}}), s \Rightarrow D^{\bar{p},\bar{r}}, \Delta'$. Therefore, using the rules of $\mathbf{LK}^-_n$, we can easily reach the sequent 
$
\Gamma, \Gamma', d(C^{\bar{q}}), d(D^{\bar{q}}), s \vee \neg s \Rightarrow C^{\bar{p}, \bar{r}}, D^{\bar{p}, \bar{r}}, \Delta, \Delta'.
$
If $s \in \bar{p} \cup \bar{r}$, then we have $\Gamma, d(C^{\bar{q}}) \Rightarrow s, C^{\bar{p}, \bar{r}}, \Delta$ and $\Gamma', d(D^{\bar{q}}) \Rightarrow D^{\bar{p}, \bar{r}}, \neg s, \Delta'$. Hence, in $\mathbf{LK}^-_n$, we can easily reach the sequent 
$
\Gamma, \Gamma', d(C^{\bar{q}}), d(D^{\bar{q}}) \Rightarrow C^{\bar{p}, \bar{r}}, D^{\bar{p}, \bar{r}}, s \wedge \neg s, \Delta, \Delta'.
$
This completes the simulation. Moreover, note that the transformation is polynomially bounded.
\end{proof}

\begin{corollary}[Krajíček \cite{krajivcek1997interpolation}] \label{FIForR} 
$\mathbf{R}$ has both feasible and monotone feasible interpolation.  
\end{corollary}
\begin{proof}
Let $\pi$ be a resolution refutation for $\{C_i(\bar{p}, \bar{q})\}_i \cup \{D_j(\bar{p}, \bar{r})\}_j$. Then, by Lemma \ref{SimulationOfRByLK}, there is an $\mathbf{LK}^-_n$-proof of size $|\pi|^{O(1)}$ for $\Gamma \Rightarrow \Delta$, where $\Gamma$ consists of the formulas $\bigvee C_i$ or unsatisfiable formulas in $\bar{q}$ and $\Delta$ consists of the formulas $\bigwedge \neg D_j$ or tautologies in $\bar{p}$ and $\bar{r}$. By Theorem \ref{FIForLK}, there is a circuit of size $|\pi|^{O(1)}$ to compute a Craig interpolant for $\Gamma \Rightarrow \Delta$. As unsatisfiable formulas on the right hand side and tautologies on the left hand side of a sequent has no effect on its interpolant, we get a circuit of size $|\pi|^{O(1)}$ to compute an interpolant for $\{C_i(\bar{p}, \bar{q})\}_i \cup \{D_j(\bar{p}, \bar{r})\}_j$. For the monotone case, assume that $\{C_i\}_i$ is monotone in $\bar{p}$. Since $\Gamma$ consists of the formulas $\bigvee_i C_i$ and $s \vee \neg s$ for $s \in \bar{q}$, and $C_i$'s are monotone in $\bar{p}$, the multiset $\Gamma$ is also monotone in $\bar{p}$. Therefore, by the monotone feasible interpolation property of $\mathbf{LK}^{-}_n$, we obtain a monotone circuit to interpolate.
\end{proof}

\begin{corollary}[Krajíček \cite{krajivcek1997interpolation}]
Let $0 < \epsilon < 1/3$ be a real number. Then, any $\mathbf{LK}^-$-proof of $\mathrm{Clique}^{\epsilon}_n(\bar{p}, \bar{q}) \to \neg \mathrm{Color}^{\epsilon}_n(\bar{p}, \bar{r})$ or any resolution refutation of $\mathrm{Clique}^{\epsilon}_n(\bar{p}, \bar{q}) \cup \mathrm{Color}^{\epsilon}_n(\bar{p}, \bar{r})$ has size at least $2^{\Omega(n^{1/3-\epsilon})}$. Consequently, $\mathbf{LK}^-$ and $\mathbf{R}$ are not p-bounded.
\end{corollary}
\begin{proof}
By Theorem \ref{FIForLK} and Corollary \ref{FIForR}, both $\mathbf{LK}^-_n$ and $\mathbf{R}$ have monotone feasible interpolation. Thus, we can apply Corollary \ref{MainLowerBoundTheorem} to obtain the claimed lower bounds for $\mathbf{LK}^-_n$ and $\mathbf{R}$. Then, we can replace $\mathbf{LK}_n^-$ by $\mathbf{LK}^-$ by the equivalence $\mathbf{LK}^- \equiv \mathbf{LK}^-_n$ from Theorem \ref{LK=LKN}.
\end{proof}

\begin{remark}
For many proof systems for $\mathsf{CPC}$, such as $\mathbf{CP}$ and $\mathbf{NS}$, the monotone feasible interpolation technique can be used to demonstrate that they are not p-bounded. However, in each case, it is necessary to transition from monotone circuits to more sophisticated computational models and establish a lower bound on their ability to separate certain disjoint $\mathbf{NP}$ pairs.
For example, in the case of Cutting Planes, Pudlák employs monotone real circuits and proves that $\mathbf{CP}$ admits monotone feasible interpolation with respect to these circuits. A real monotone circuit is one that can use any monotone function on real numbers as gates. Pudlák then establishes an exponential lower bound for the Clique-Color disjoint $\mathbf{NP}$ pair, concluding that $\mathrm{Clique}^{\epsilon}_n(\bar{p}, \bar{q}) \cup \mathrm{Color}^{\epsilon}_n(\bar{p}, \bar{r})$ has exponentially long refutations in $\mathbf{CP}$, and therefore, $\mathbf{CP}$ is not p-bounded \cite{pudlak1997lower}. For more on these variations, see the comprehensive monograph \cite{krajivcek2019proof}.
\end{remark}

To prove the lack of (monotone) feasible interpolation by applying Theorem \ref{MainLowerBoundTheorem} or Corollary \ref{MainCorForNoFI}, we must provide short proofs for implications with hard Craig interpolants. To that goal, we need some upper bound results to show that certain proof systems are sufficiently strong:

\begin{theorem}\label{UpperBounds}
There exist $k, l \in \mathbb{N}$ such that for any sufficiently large $n \in \mathbb{N}$, we have:
\begin{itemize}
    \item[$\bullet$] 
    \cite{maciel2002new,krajivcek2019proof} $\mathrm{PHP}_{n}^{m(n)}$ have $\mathbf{LK}_2$-proofs of size less than $2^{l(\log n)^{k}}$, if $n^2 \leq m(n) \leq 2^{(\log n)^{k}}$.
    
    \item[$\bullet$] 
    \cite{krajiivcek1998some}  $\mathrm{RSA}^1_n(\bar{p}, \bar{q}) \to \neg \mathrm{RSA}^0_n(\bar{p}, \bar{r})$ has $\mathsf{CPC}$-$\mathbf{SF}$ proofs of size less than $n^{k}$.
    
    \item[$\bullet$]
    \cite{bonet1997no,bonet2004non} $\mathrm{DH}^1_n(\bar{p}, \bar{q}) \to \neg \mathrm{DH}^1_n(\bar{p}, \bar{r})$ has $\mathbf{LK}$-proofs of size less than $n^{k}$. Moreover, there exists $d \geq 2$ such that it also has $\mathbf{LK}_d$-proofs of size less than $2^{n^{\epsilon}}$, for some $\epsilon > 0$.
\end{itemize}
\end{theorem}
\begin{proof}
These upper bounds are typically obtained by reasoning in a non-propositional framework (often an appropriate fragment of arithmetic), and then translating the argument (fragment) into propositional proofs. For the first item, one should prove the pigeonhole principle in its usual first-order form via a weak form of induction~\cite{krajivcek2019proof}. 
For the other two items, which correspond to disjoint $\mathbf{NP}$ pairs, the main reason why the implications are tautologies stems from the disjointness of the $\mathbf{NP}$ pairs, itself a consequence of semi-injectivity in the corresponding one-way protocols.  
Thus, to obtain propositional proofs for these implications, it suffices to express the earlier proofs of semi-injectivity for RSA and Diffie–Hellman protocols in a propositional form within the corresponding propositional proof systems \cite{krajivcek2019proof}.
\end{proof}

Now, we are ready to prove the lack of (monotone) feasible interpolation:

\begin{theorem}[Krajíček \cite{krajivcek1997interpolation}]
Let $P$ be a proof system for $\mathsf{CPC}$ such that $P \geq \mathbf{LK}_2$. Then, $P$ does not have monotone feasible interpolation.
\end{theorem}
\begin{proof}
By Theorem~\ref{UpperBounds}, there exists $k \in \mathbb{N}$ such that the formula $\mathrm{PHP}^{l(m)}_m$ has an $\mathbf{LK}_2$-proof of size less than $2^{(\log m)^k}$, provided $m^2 \leq l(m) \leq m^3$.
We use this pigeonhole principle to construct an $\mathbf{LK}_2$-proof of the tautology 
$
\mathrm{Clique}^{1/6}_n(\bar{p}, \bar{q}) \to \neg \mathrm{Color}^{1/6}_n(\bar{p}, \bar{r})
$
of size $2^{(\log n)^{O(1)}}$.
We omit the detailed construction, but the main idea is as follows: Assume 
$
\mathrm{Clique}^{1/6}_n(\bar{p}, \bar{q}) \wedge \mathrm{Color}^{1/6}_n(\bar{p}, \bar{r}).
$
Then, by Examples~\ref{ExampleClique} and~\ref{ExampleColor}, $\bar{q}$ and $\bar{r}$ effectively describe an injective mapping from the $\lfloor n^{2/3} \rfloor$ vertices of an $\lfloor n^{2/3} \rfloor$-clique to the $\lfloor n^{1/3} \rfloor$ colors of a $\lfloor n^{1/3} \rfloor$-coloring. However, since
$
\lfloor n^{1/3} \rfloor^2 \leq \lfloor n^{2/3} \rfloor \leq \lfloor n^{1/3} \rfloor^3,
$
using a short reasoning in $\mathbf{LK}_2$, one can show that this violates an instance of the pigeonhole principle $\mathrm{PHP}^{l(m)}_m$, when $m = \lfloor n^{1/3} \rfloor$ and $l(m) = \lfloor n^{2/3} \rfloor$.
As this pigeonhole principle has an $\mathbf{LK}_2$-proof of size $O(2^{(\log m)^k})$, the instance has an $\mathbf{LK}_2$-proof of size bounded by $2^{(\log n)^{O(1)}}$. This then implies that the tautology $\mathrm{Clique}^{1/6}_n(\bar{p}, \bar{q}) \to \neg \mathrm{Color}^{1/6}_n(\bar{p}, \bar{r})$ has an $\mathbf{LK}_2$-proof of size $2^{(\log n)^{O(1)}}$.
Since $P \geq \mathbf{LK}_2$, the same upper bound holds for $P$-proofs. But if $P$ has monotone feasible interpolation, then by Theorem~\ref{MainLowerBoundTheorem}, any $P$-proof of this tautology must be of size at least $2^{\Omega(n^{1/6})}$, leading to a contradiction.
Therefore, $P$ does not have monotone feasible interpolation.
\end{proof}

\begin{theorem}
\begin{itemize}
    \item[$\bullet$]
\cite{krajiivcek1998some}  
Assume that $\mathrm{RSA}$ is secure against $\mathbf{P/poly}$ adversaries. Then, no proof system for $\mathsf{CPC}$ that simulates $\mathsf{CPC}$-$\mathbf{SF}$ has the feasible interpolation property. In particular, $\mathsf{CPC}$-$\mathbf{SF}$ is not automatable.
    \item[$\bullet$]
\cite{bonet1997no,bonet2004non}  
Assume that the Diffie–Hellman protocol is secure against $\mathbf{P/poly}$ adversaries. Then, no proof system for $\mathsf{CPC}$ that simulates $\mathbf{LK}$ has the feasible interpolation property. In particular, $\mathbf{LK}$ is not automatable.
Moreover, if we assume that the Diffie–Hellman protocol is secure against circuits of size $2^{n^{\epsilon}}$ for any $\epsilon > 0$, then there exists $d \geq 2$ such that $\mathbf{LK}_d$ does not have the feasible interpolation property, and hence is not automatable.
\end{itemize}
\end{theorem}
\begin{proof}
We only prove the first part, as the second follows by a similar argument.  
Let $P$ be a proof system for $\mathsf{CPC}$ that simulates $\mathsf{CPC}$-$\mathbf{SF}$.  
By Theorem~\ref{UpperBounds}, the sequence $\{\mathrm{RSA}^1_n(\bar{p}, \bar{q}) \to \neg \mathrm{RSA}^0_n(\bar{p}, \bar{r})\}_n$ has a short $\mathsf{CPC}$-$\mathbf{SF}$-proof, and hence a short $P$-proof.
Assuming that $\mathrm{RSA}$ is secure against $\mathbf{P/poly}$ adversaries, it follows from Example~\ref{RSASequenceHaveIsHardInt} that this sequence has hard Craig interpolants. Therefore, by Theorem~\ref{MainCorForNoFI}, the system $P$ cannot have the feasible interpolation property.
To conclude non-automatability, recall that $\mathsf{CPC}$-$\mathbf{SF}$ is substitutable and closed under variable-free modus ponens.
\end{proof}

\section{Feasible Interpolation (Non-classical)} \label{Sec: FI, Non-Classical}

In this section, we adapt the feasible interpolation technique to the non-classical setting by replacing Craig interpolants with disjunctive interpolants. Using this modified framework, we demonstrate that many proof systems, including $\mathbf{LJ}$ (or equivalently, $\mathsf{IPC}$-Frege), as well as $L$-Frege for logics $L \subseteq \mathsf{S4}$ or $L \subseteq \mathsf{GL}$, are not p-bounded.

\begin{definition}
Let $L$ be a si logic. A proof system $P$ for $L$ is called to have \emph{(monotone) feasible $\mathrm{PDIP}$} if there is a polynomial $p$ such that for any $P$-proof $u \in \{0, 1\}^*$ of an implication $\phi \to \psi \vee \theta$ (where $\phi$ is monotone) with the code $w$, there are (monotone) circuits $C$ and $D$ of size bounded by $p(|u|, |w|)$ such that $([C], [D])$ is an $L$-$\mathrm{PDI}$ for $\phi \to \psi \vee \theta$. For a modal logic $L$, the definition is similar using formulas in the form $\phi \to \Box \psi \vee \Box \theta$, $\mathcal{L}_{\Box}$-circuits and $L$-$\mathrm{MDI}$, instead.
\end{definition}

\begin{theorem}\label{IFILowerbound}
Let $L$ be a si (resp. consistent modal) logic. If hard disjoint $\mathbf{NP}$ pairs exist, or in particular, if one-way protocols exist, then no proof system for $L$ with feasible $\mathrm{PDIP}$ (resp. $\mathrm{MDIP}$) is p-bounded. 
\end{theorem}
\begin{proof}
We prove only the si case; the modal case is analogous. If hard disjoint $\mathbf{NP}$ pairs exist, or in particular, if one-way protocols exist, Corollary~\ref{DNPImpliesHardInt} implies the existence of a sequence $\{\phi_n(\bar{p}, \bar{r}) \to \psi_n(\bar{p}, \bar{s})\}_n$ of short classical tautologies with hard Craig interpolants. By Theorems~\ref{ClassicalToIPC} and~\ref{CraigToDisjunctiveInt},
$
\{\Phi_n\}_n = \left\{\bigwedge_i (p_i \vee \neg p_i) \to \neg \phi_n(\bar{p}, \bar{r}) \vee \neg \neg \psi_n(\bar{p}, \bar{s})\right\}_n
$
is a sequence of short formulas in $\mathsf{IPC} \subseteq L$ that has hard $L$-$\mathrm{PDI}$'s.
Now, if $P$ is p-bounded, then the sequence $\{\Phi_n\}_n$ has a short $P$-proof. Using the feasible $\mathrm{PDIP}$, this would yield a sequence $\{(C_n, D_n)\}_n$ of pairs of poly-size circuits computing an $L$-$\mathrm{PDI}$ of $\{\Phi_n\}_n$, which is impossible.
\end{proof}

\begin{theorem}\label{LowerBoundForIntCliquevsColor}
Let $0 < \epsilon < 1/3$ be a real number, $L$ be a si (resp. consistent modal) logic, and $P$ be a proof system for $L$. If $P$ has monotone feasible $\mathrm{PDIP}$ (resp. $\mathrm{MDIP}$), then any $P$-proof of 
$
\bigwedge_i (p_i \vee q_i) \to \neg \mathrm{Clique}^{\epsilon}_n(\neg \bar{p}, \bar{r}) \vee \neg \mathrm{Color}^{\epsilon}_n(\bar{q}, \bar{s})
$
(resp. $\bigwedge_i (\Box p_i \vee \Box q_i) \to \Box \neg \mathrm{Clique}^{\epsilon}_n(\neg \bar{p}, \bar{r}) \vee \Box \neg \mathrm{Color}^{\epsilon}_n(\bar{q}, \bar{s})$)
has size at least $2^{\Omega(n^{1/3 - \epsilon})}$. Hence, $P$ is not p-bounded. 
\end{theorem}
\begin{proof}
The proof is a direct consequence of Theorem~\ref{Clique-ColorForIPC}.
\end{proof}

In the next two subsections, we prove monotone feasible $\mathrm{PDIP}$ (resp. $\mathrm{MDIP}$) for certain proof systems for $\mathsf{IPC}$ (resp. consistent modal logics) in order to apply Theorem~\ref{LowerBoundForIntCliquevsColor}.

\subsection{Super-intuitionistic Logics}

Our main objective in this subsection is to follow the strategy presented in \cite{hrubevs2009lengths} to show that the proof system $\mathbf{LJ}$ has monotone feasible $\mathrm{PDIP}$. We begin with a definition and some preliminary observations.

\begin{definition}
An $\mathcal{L}_p$-formula is called an \emph{implicational Horn formula} if it is either an atom or has the form $\bigwedge_{i=1}^k p_i \to r$, where the $p_i$'s and $r$ are atomic formulas.
\end{definition}

\begin{lemma}\label{conservativity}
Let $\Gamma$ be a sequence of implicational Horn formulas, $\Pi$ and $\Delta$ be sequences of monotone $\mathcal{L}_p$-formulas and $q$ and $r$ be atoms. Then:
\begin{description}
    \item[$(i)$] 
    If the sequent $(\Gamma, \Pi \Rightarrow \Delta)$ is valid, then $\Gamma, \Pi \Rightarrow \bigvee \Delta$ is provable in $\mathbf{LJ}$.
    \item[$(ii)$] 
    If the sequent $(\Gamma \Rightarrow q \vee r)$ is valid, then either $(\Gamma \Rightarrow q)$ or $(\Gamma \Rightarrow r)$ is valid.  
\end{description}
\end{lemma}
\begin{proof}
For part $(i)$, we proceed by induction on a cut-free $\mathbf{LK}$-proof of $\Gamma, \Pi \Rightarrow \Delta$. The key observation is that since $\Delta$ contains no implications and the formulas in $\Gamma \cup \Pi$ do not involve implications of depth greater than one, the classical rule $(R\!\to)$ is never applied in the proof. Hence, the classical derivation is already valid intuitionistically.
For part $(ii)$, apply part $(i)$ to obtain $\mathbf{LJ} \vdash \Gamma \Rightarrow q \vee r$. Then, an induction on the cut-free $\mathbf{LJ}$-proof of this sequent shows that either $\mathbf{LJ} \vdash \Gamma \Rightarrow q$ or $\mathbf{LJ} \vdash \Gamma \Rightarrow r$.
\end{proof}

The second property of implicational Horn formulas is the following crucial theorem, proved by Hrubeš~\cite{hrubevs2009lengths}:

\begin{theorem}[Hrubeš \cite{hrubevs2009lengths}]\label{HrubesLemma}
For any sequence $\Gamma$ of implicational Horn formulas, a finite set of atoms $P = \{p_1, \ldots, p_m\}$, and an atom $q$, there exists a monotone circuit $C(x_1, \ldots, x_m)$ of size $(|\Gamma| + m)^{O(1)}$ such that $C(\bar{a}) = 1$ if and only if the sequent $(\Gamma, \{p_i \in P \mid a_i = 1\} \Rightarrow q)$ is valid, for any $\bar{a} \in \{0, 1\}^m$.
\end{theorem}

Our strategy for proving monotone feasible $\mathrm{PDIP}$ for $\mathbf{LJ}$ is to first establish a similar, yet technically simpler, claim in which the disjunction on the right-hand side ranges over atoms, and the left-hand side can have an additional sequence of implicational Horn formulas. We then define a suitable translation function to reduce the general case to this specific setting. This translation is a powerful machine for extracting feasible data from intuitionistic proofs and is of independent interest. For more, see \cite{tabatabai2025universal}.

\begin{theorem}\label{AtomicFI}
For any sequence $\Gamma$ of implicational Horn formulas, a monotone $\mathcal{L}_p$-formula $\phi$ and atoms $q$ and $r$,
if $\Gamma, \phi \Rightarrow q \vee r$ is valid, then there are monotone circuits $C$ and $D$ in the variables of $\phi$ and of size $(|\Gamma|+|\phi|)^{O(1)}$ such that the sequents $(\Gamma, \phi \Rightarrow [C] \vee [D])$, $(\Gamma, [C] \Rightarrow q)$ and $(\Gamma, [D] \Rightarrow r)$ are all valid. 
\end{theorem}
\begin{proof}
Let $P=V(\phi)$ and $C(\bar{x})$ be the monotone circuit from Theorem \ref{HrubesLemma} for $P$, $\Gamma$, and $q$ and define $D(\bar{x})$ similarly but using $r$ instead of $q$. We claim that $C(\bar{p})$ and $D(\bar{p})$ are the circuits we want.
First, it is clear that the size of $C(\bar{p})$ and $D(\bar{p})$ is bounded by $(|\Gamma|+|\phi|)^{O(1)}$. Second, define $T$ as the set of the assignments $\bar{a} \in \{0, 1\}$ for the atoms in $P$ that makes $\phi$ true. For any $\bar{a} \in T$, as $\phi$ is monotone, the sequent $(\{p_i \in P \mid a_i=1\} \Rightarrow \phi)$ is valid. Hence, as $\Gamma, \phi \Rightarrow q \vee r$ is valid, so is  
$(\Gamma, \{p_i \in P \mid a_i=1\} \Rightarrow q \vee r)$. As $\Gamma \cup \{p_i \mid a_i=1\}$ consists of implicational Horn formulas, by Theorem \ref{conservativity}, part $(ii)$, either $\Gamma, \{p_i \in P \mid a_i=1\} \Rightarrow q$ or $\Gamma, \{p_i \in P \mid a_i=1\} \Rightarrow r$ are valid. Therefore, for any $\bar{a} \in T$, either $C(\bar{a})=1$ or $D(\bar{a})=1$. This means that $\phi \Rightarrow [C(\bar{p})] \vee [D(\bar{p})]$ is valid. Third, we claim that $\Gamma, [C(\bar{p})] \Rightarrow q$ is valid. The reason is that for any assignment for the atoms in the sequent that assigns $\bar{a}$ for $\bar{p}$, if $C(\bar{a})=1$, then by definition of $C(\bar{x})$, the sequent $\Gamma, \{p_i \mid a_i=1\} \Rightarrow q$ is valid. Now, using the same assignment, as every atoms in the set $\{p_i \mid a_i=1\}$ becomes true by mapping $p_i$ to $a_i$, the sequent $\Gamma \Rightarrow  q$ becomes true under the assignment. Hence, $\Gamma, [C(\bar{p})] \Rightarrow q$ is valid. Similarly, we can prove that $\Gamma, [D(\bar{p})] \Rightarrow r$ is valid.
\end{proof}

For the translation part, define $\mathcal{L}_p^+$ as the extension of $\mathcal{L}_p$ by new atoms of the form $\langle \phi \rangle$, for each formula $\phi \in \mathcal{L}_p$. Clearly, there is a canonical substitution that substitutes each atom $\langle \phi \rangle$ with the corresponding formula $\phi$. We refer to this substitution as the \emph{standard substitution}, and denote it by $s$. Now, consider the following translation:

\begin{definition} \label{Translation}
Define the translation $t: \mathcal{L}_p \to \mathcal{L}_p^+$ as follows:
$\bot^t = \bot$, $\top^t = \langle \top \rangle$, $p^t = \langle p \rangle$ for any atomic formula $p$, and
$(\phi \circ \psi)^t = (\phi^t \circ \psi^t) \wedge \langle \phi \circ \psi \rangle$,
for any $\circ \in \{\wedge, \vee, \to\}$.  
For a multiset $\Gamma$, define $\Gamma^t = \{\gamma^t \mid \gamma \in \Gamma\}$.
\end{definition}

It is clear that the function $t$ and the substitution $s$ are polynomial-time computable. Moreover, for any formula $\phi \in \mathcal{L}_p$, we have
$\mathbf{LJ} \vdash \phi^t \Rightarrow \langle \phi \rangle$ and $\mathbf{LJ} \vdash (\phi^t)^s \Leftrightarrow \phi$.

\begin{lemma}
\label{ProvPreservation}
There is a polynomial $p$ such that for any $\mathbf{LJ}$-proof $\pi$ of $\Omega \Rightarrow \Lambda$, there is a sequence $\Sigma_{\pi}$ of implicational Horn formulas such that $|\Sigma_{\pi}| \leq p(|\pi|)$, 
$\mathbf{LJ} \vdash \Sigma_{\pi}, \Omega^t \Rightarrow  \Lambda^t$, and 
$\mathbf{LJ} \vdash \; \Rightarrow \bigwedge \Sigma_{\pi}^{s}$.
\end{lemma}
\begin{proof}
The construction of $\Sigma_{\pi}$ in terms of $\pi$ is by recursion on the structure of $\pi$. The case of axioms is easy. If the last rule in $\pi$ is structural (including the cut) or a left rule, it suffices to set $\Sigma_{\pi}$ as the union of the $\Sigma$'s from the immediate subproofs of $\pi$. For the right rules, we explain only the case of $(R\!\to)$; the others are similar. Suppose $\pi$ has the form:
\begin{center}
\begin{tabular}{c}
    \AxiomC{$\pi'$}
    \noLine
    \UnaryInfC{$\Gamma, \phi \Rightarrow \psi$}
    \RightLabel{$R\!\to$}
    \UnaryInfC{$\Gamma \Rightarrow \phi \to \psi$}
    \DisplayProof
\end{tabular}
\end{center}
Then, by the induction hypothesis, we have the multiset $\Sigma_{\pi'}$ of implicational Horn formulas such that $\mathbf{LJ} \vdash \Sigma_{\pi'}, \Gamma^t, \phi^t \Rightarrow \psi^t$ and $\mathbf{LJ} \vdash \, \Rightarrow \bigwedge \Sigma^{s}_{\pi'}$. Hence, $\mathbf{LJ} \vdash \Sigma_{\pi'}, \Gamma^t \Rightarrow \phi^t \rightarrow \psi^t$. Then, set $\Sigma_{\pi} = \Sigma_{\pi'} \cup \{\bigwedge_{\gamma \in \Gamma} \langle \gamma \rangle \to \langle \phi \to \psi \rangle\}$. It is clear that $\Sigma_{\pi}$ consists of implicational Horn formulas and that $\mathbf{LJ} \vdash \, \Rightarrow \bigwedge \Sigma^{s}_{\pi}$. Since $\mathbf{LJ} \vdash \gamma^t \Rightarrow \langle \gamma \rangle$ for any $\gamma \in \Gamma$, we obtain $\mathbf{LJ} \vdash \Sigma_{\pi}, \Gamma^t \Rightarrow \langle \phi \to \psi \rangle$. Therefore, as $(\phi \to \psi)^t = (\phi^t \to \psi^t) \wedge \langle \phi \to \psi \rangle$, we conclude that $\mathbf{LJ} \vdash \Sigma_{\pi}, \Gamma^t \Rightarrow (\phi \to \psi)^t$. This completes the construction of $\Sigma_{\pi}$. By inspecting the construction, it is easy to see that $|\Sigma_{\pi}| \leq |\pi|^{O(1)}$.
\end{proof}

\begin{theorem}\label{LJHasFI}
$\mathbf{LJ}$ has monotone feasible $\mathrm{PDIP}$.
\end{theorem}
\begin{proof}
Let $\pi$ be an $\mathbf{LJ}$-proof of $\, \Rightarrow \phi \to \psi \vee \theta$, where $\phi$ is monotone. Then, using cut, it is easy to construct an $\mathbf{LJ}$-proof $\pi'$ of $\phi \Rightarrow \psi \vee \theta$, such that $|\pi'| \leq |\pi|^{O(1)}$. Then, by Lemma~\ref{ProvPreservation}, there is a sequence of implicational Horn formulas $\Sigma_{\pi'}$ such that $|\Sigma_{\pi'}| \leq |\pi'|^{O(1)} \leq |\pi|^{O(1)}$, $\mathbf{LJ} \vdash \Sigma_{\pi'}, \phi^t \Rightarrow (\psi \vee \theta)^t$, and $\mathbf{LJ} \vdash \, \Rightarrow \bigwedge \Sigma^s_{\pi'}$. By the definition of $t$, we have $\mathbf{LJ} \vdash \Sigma_{\pi'}, \phi^t \Rightarrow \psi^t \vee \theta^t$. As $\mathbf{LJ} \vdash \psi^t \Rightarrow \langle \psi \rangle$ and $\mathbf{LJ} \vdash \theta^t \Rightarrow \langle \theta \rangle$, we get $\mathbf{LJ} \vdash \Sigma_{\pi'}, \phi^t \Rightarrow \langle \psi \rangle \vee \langle \theta \rangle$, which implies the validity of $(\Sigma_{\pi'}, \phi^t \Rightarrow \langle \psi \rangle \vee \langle \theta \rangle)$. Let $q = \langle \psi \rangle$ and $r = \langle \theta \rangle$. Then, as $\phi$ and hence $\phi^t$ are monotone, by Theorem~\ref{AtomicFI}, there exist monotone circuits $C'$ and $D'$ on the variables of $\phi^t$ and of size bounded by $(|\Sigma_{\pi'}| + |\phi^t|)^{O(1)} \leq |\pi|^{O(1)}$, such that the sequents $(\Sigma_{\pi'}, \phi^t \Rightarrow [C'] \vee [D'])$, $(\Sigma_{\pi'}, [C'] \Rightarrow q)$, and $(\Sigma_{\pi'}, [D'] \Rightarrow r)$ are all valid and hence provable in $\mathbf{LJ}$, by Theorem~\ref{conservativity}, part~$(i)$. 
Now note that the atoms of $\phi^t$ are of the form $\langle \alpha \rangle$, where $\alpha$ is a subformula of $\phi$. As these $\alpha$'s are monotone, applying the standard substitution $s$ to $C'$ and $D'$ yields two monotone circuits $C$ and $D$ on the variables of $\phi$ and of size bounded by $|\pi|^{O(1)}$. Since $\mathbf{LJ} \vdash \, \Rightarrow \bigwedge \Sigma^s_{\pi'}$ and $\mathbf{LJ} \vdash (\phi^t)^s \Leftrightarrow \phi$, by applying the standard substitution $s$, the sequents $(\phi \Rightarrow [C] \vee [D])$, $([C] \Rightarrow \psi)$, and $([D] \Rightarrow \theta)$ are all provable in $\mathbf{LJ}$. Hence, $(\phi \rightarrow [C] \vee [D])$, $([C] \rightarrow \psi)$, and $([D] \rightarrow \theta)$ are all provable in $\mathsf{IPC}$.
\end{proof}

\begin{corollary}[Hrubeš \cite{hrubevs2009lengths,hrubevs2007lower,hrubevs2007lowerII}]
Let $0 < \epsilon < 1/3$ be a real number. Then, any $\mathbf{LJ}$-proof of 
$
\bigwedge_i (p_i \vee q_i) \to \neg \mathrm{Clique}^{\epsilon}_n(\neg \bar{p}, \bar{r}) \vee \neg \mathrm{Color}^{\epsilon}_n(\bar{q}, \bar{s})
$
has size at least $2^{\Omega(n^{1/3 - \epsilon})}$. Therefore, $\mathbf{LJ}$ and hence $\mathsf{IPC}$-Frege is not p-bounded.
\end{corollary}
\begin{proof}
For the first part, we apply Theorem \ref{LowerBoundForIntCliquevsColor} and Theorem \ref{LJHasFI}.
For the second part, we recall that $\mathbf{LJ}$ is equivalent to the $\mathsf{IPC}$-Frege system, by Theorem~\ref{LKEquivFrege}.
\end{proof}

For si logics of infinite branching, Jeřábek extended Hrubeš's result by combining reductions with a modification of the above argument:

\begin{theorem}[Jeřábek \cite{jevrabek2009substitution}]
Let $L$ be a si logic of infinite branching, i.e., either $\mathsf{L} \subseteq \mathsf{BD_2}$ or $\mathsf{L} \subseteq \mathsf{KC+BD_3}$. Then, $L$-Frege is not p-bounded.
\end{theorem}

\begin{remark}
Here are two remarks. First, as we have seen, the formulas with exponentially long proofs we provided make serious use of disjunctions, and one may suspect that the use of disjunctions is crucial for such bounds. However, Jeřábek showed that a similar result is possible using only the implicational fragment of the language \cite{jevrabek2017proof,jevrabek2025simplified}. Second, for a logic $L$ weaker than $\mathsf{IPC}$, it may seem obvious that $L$-Frege cannot be p-bounded, since $L$ is weaker than $\mathsf{IPC}$ and even $\mathsf{IPC}$-Frege proofs are already hard for the intuitionistic Clique-Coloring tautologies. However, this is not so straightforward. To show that $L$-Frege is not p-bounded, we must provide a sequence of short formulas \emph{in $L$} that have no short $L$-Frege proofs, and there is no reason to assume that our intuitionistic Clique-Coloring tautologies belong to $L$. To address this issue, Jalali~\cite{jalali2021proof} employed a suitable descending technique to modify the intuitionistic Clique-Coloring tautologies so that they become provable in the Full Lambek substructural logic $\mathsf{FL}$. She then showed that for any logic between $\mathsf{FL}$ and a superintuitionistic logic of infinite branching, the suitably defined $L$-Frege system is not p-bounded.
\end{remark}

\subsection{Modal Logics}

Our task in this subsection is to apply a technique similar to that used in the previous subsection to show that $\mathbf{GL}$ and $\mathbf{S4}$ have monotone feasible $\mathrm{MDIP}$. First, extend the language $\mathcal{L}_{\Box}$ to $\mathcal{L}_{\Box}^+$ by introducing new atoms $\langle \Box \phi \rangle$ for each formula $\phi \in \mathcal{L}_{\Box}$. Clearly, there is a canonical substitution $s$ that substitutes each atom $\langle \Box \phi \rangle$ with $\Box \phi$. We refer to this substitution as the \emph{standard substitution}. Now, consider the following translations:

\begin{definition} \label{ModalTranslation}
For any $i \in \{0, 1\}$, define $t_i: \mathcal{L}_{\Box} \to \mathcal{L}_{\Box}^+$ by:  
$\bot^{t_i} = \bot$, $\top^{t_i} = \top$, and $p^{t_i} = p$, for any atomic formula $p$;  
$(\phi \circ \psi)^{t_i} = (\phi^{t_i} \circ \psi^{t_i})$, for any $\circ \in \{\wedge, \vee, \to\}$; and  
$(\Box \phi)^{t_0} = \langle \Box \phi \rangle$, while $(\Box \phi)^{t_1} = \phi^{t_1} \wedge \langle \Box \phi \rangle$.  
For a multiset $\Gamma$ and any $i \in \{0, 1\}$, define $\Gamma^{t_i} = \{\gamma^{t_i} \mid \gamma \in \Gamma\}$.
\end{definition}

Note that both $\phi^{t_0}$ and $\phi^{t_1}$ are propositional formulas.  
The functions $t_i$, for any $i \in \{0, 1\}$, and the substitution $s$ are polynomial-time computable, and it is clear that  
$\mathbf{K} \vdash (\Box \phi)^{t_i} \Rightarrow \langle \Box \phi \rangle$, for any $i \in \{0, 1\}$, $\mathbf{GL} \vdash (\phi^{t_0})^s \Leftrightarrow \phi$ and $\mathbf{S4} \vdash (\phi^{t_1})^s \Leftrightarrow \phi$,  
for any formula $\phi \in \mathcal{L}_{\Box}$.

\begin{lemma}
\label{ModalProvPreservation}
Let $L_0=\mathsf{GL}$, $L_1=\mathsf{S4}$, $G_0=\mathbf{GL}$, and $G_1=\mathbf{S4}$. Then, for any $i \in \{0, 1\}$, there is a polynomial $p$ such that for any $G_i$-proof $\pi$ of $\Omega \Rightarrow \Lambda$, there is a sequence $\Sigma_{\pi}$ of implicational Horn formulas such that $|\Sigma_{\pi}| \leq |\pi|^{O(1)}$, the sequent 
$(\Sigma_{\pi}, \Omega^{t_i} \Rightarrow \Lambda^{t_i})$ is valid, and $G_i \vdash \; \Rightarrow \bigwedge \Sigma_{\pi}^{s}$.
\end{lemma}
\begin{proof}
For any $i \in \{0, 1\}$, the construction of $\Sigma_{\pi}$ proceeds by recursion on the structure of $\pi$. The case for axioms is straightforward, by setting $\Sigma_{\pi}=\emptyset$. For the rules, if the last rule in $\pi$ is a structural or propositional rule, it suffices to set $\Sigma_{\pi}$ as the union of the $\Sigma$'s of the immediate subproofs. The only interesting cases occur when the last rule in $\pi$ is a modal rule.
For $i=0$, let $\pi$ have the form:
\begin{center}
\begin{tabular}{c}
    \AxiomC{$\pi'$}
    \noLine
    \UnaryInfC{$\Gamma, \Box \Gamma, \Box \phi \Rightarrow \phi$}
    \RightLabel{$GL$}
    \UnaryInfC{$\Box \Gamma \Rightarrow \Box \phi$}
    \DisplayProof
\end{tabular}
\end{center}
Then, define $\Sigma_{\pi}=\{\bigwedge_{\gamma \in \Gamma} \langle \Box \gamma \rangle \to \langle \Box \phi \rangle\}$. It is clear that $\Sigma_{\pi}$ consists of implicational Horn formulas, and the sequent $(\Sigma_{\pi}, (\Box \Gamma)^{t_0} \Rightarrow (\Box \phi)^{t_0})$ is valid. Moreover, we have $\mathbf{GL} \vdash \, \Rightarrow \bigwedge \Sigma_{\pi}^s$.
For $i=1$, if the last rule in $\pi$ is $(LS4)$, set $\Sigma_{\pi}$ as the $\Sigma$ of the immediate subproof. For $(RS4)$, let $\pi$ have the form:
\begin{center}
\begin{tabular}{c}
    \AxiomC{$\pi'$}
    \noLine
    \UnaryInfC{$\Box \Gamma \Rightarrow \phi$}
    \RightLabel{$RS4$}
    \UnaryInfC{$\Box \Gamma \Rightarrow \Box \phi$}
    \DisplayProof
\end{tabular}
\end{center}
Then, by the induction hypothesis, there is a sequence $\Sigma_{\pi'}$ of implicational Horn formulas such that $(\Sigma_{\pi'}, (\Box \Gamma)^{t_1} \Rightarrow \phi^{t_1})$ is valid and $\mathbf{S4} \vdash \, \Rightarrow \bigwedge \Sigma^s_{\pi'}$. Define 
$
\Sigma_{\pi}=\Sigma_{\pi'} \cup \{\bigwedge_{\gamma \in \Gamma} \langle \Box \gamma \rangle \to \langle \Box \phi \rangle\}.
$
It is clear that $\Sigma_{\pi}$ consists of implicational Horn formulas and that $\mathbf{S4} \vdash \, \Rightarrow \bigwedge \Sigma^s_{\pi}$. Since $(\Box \phi)^{t_1}=(\phi^{t_1} \wedge \langle \Box \phi \rangle)$ and $(\Box \gamma)^{t_1}=(\gamma^{t_1} \wedge \langle \Box \gamma \rangle)$ for any $\gamma \in \Gamma$, it is easy to verify the validity of $(\Sigma_{\pi}, (\Box \Gamma)^{t_1} \Rightarrow (\Box \phi)^{t_1})$. Finally, checking the construction, we have $|\Sigma_{\pi}| \leq |\pi|^{O(1)}$. 
\end{proof}

\begin{theorem}\label{S4HasFI}
$\mathbf{S4}$ and $\mathbf{GL}$ have monotone feasible $\mathrm{MDIP}$.
\end{theorem}
\begin{proof}
For $\mathbf{S4}$, let $\pi$ be an $\mathbf{S4}$-proof of $(\, \Rightarrow \phi \rightarrow \Box \psi \vee \Box \theta)$, where $\phi$ is a monotone $\mathcal{L}_{\Box}$-formula. Using the cut rule, it is easy to construct an $\mathbf{S4}$-proof $\pi'$ of $\phi \Rightarrow \Box \psi \vee \Box \theta$ such that $|\pi'| \leq |\pi|^{O(1)}$. By Lemma~\ref{ModalProvPreservation}, we obtain a sequence of implicational Horn formulas $\Sigma_{\pi'}$ such that $|\Sigma_{\pi'}| \leq |\pi'|^{O(1)} \leq |\pi|^{O(1)}$, the sequent $(\Sigma_{\pi'}, \phi^{t_1} \Rightarrow (\Box \psi)^{t_1} \vee (\Box \theta)^{t_1})$ and hence $(\Sigma_{\pi'}, \phi^{t_1} \Rightarrow \langle \Box \psi \rangle \vee \langle \Box \theta \rangle)$ is valid, and $\mathbf{S4} \vdash \, \Rightarrow \bigwedge \Sigma^s_{\pi'}$. 
Since $\phi^{t_1}$ is monotone, by Lemma~\ref{AtomicFI}, there exist monotone circuits $C'$ and $D'$ on the variables of $\phi^{t_1}$ and of size bounded by $(|\Sigma_{\pi'}| + |\phi^{t_1}|)^{O(1)} \leq |\pi|^{O(1)}$, such that the sequents $(\Sigma_{\pi'}, \phi^{t_1} \Rightarrow [C'] \vee [D'])$, $(\Sigma_{\pi'}, [C'] \Rightarrow \langle \Box \psi \rangle)$, and $(\Sigma_{\pi'}, [D'] \Rightarrow \langle \Box \theta \rangle)$ are all valid.
Now, applying the standard substitution, we obtain monotone $\mathcal{L}_{\Box}$-circuits $C$ and $D$ on the variables of $\phi$ and of size bounded by $|\pi|^{O(1)}$, such that the sequents $((\phi^{t_1})^s \Rightarrow [C] \vee [D])$, $([C] \Rightarrow \Box \psi)$, and $([D] \Rightarrow \Box \theta)$ are all provable in $\mathbf{S4}$. Since $\mathbf{S4} \vdash \phi \Leftrightarrow (\phi^{t_1})^s$, it follows that the formulas $(\phi \to [C] \vee [D])$, $([C] \to \Box \psi)$, and $([D] \to \Box \theta)$ are all provable in $\mathsf{S4}$.
For $\mathbf{GL}$, the proof is similar; it suffices to use the translation $t_0$ instead of $t_1$.
\end{proof}

\begin{corollary}[Hrubeš \cite{hrubevs2007lower,hrubevs2007lowerII,hrubevs2009lengths}]
Let $0 < \epsilon < 1/3$ be a real number. Then, any $\mathbf{GL}$-proof or $\mathbf{S4}$-proof of $\bigwedge_i (\Box p_i \vee \Box q_i) \to \Box \neg \mathrm{Clique}^{\epsilon}_n(\neg \bar{p}, \bar{r}) \vee \Box \neg \mathrm{Color}^{\epsilon}_n(\bar{q}, \bar{s})$ has size at least $2^{\Omega(n^{1/3 - \epsilon})}$. Therefore, the proof systems $\mathbf{GL}$ and $\mathbf{S4}$ and hence, for any modal logic $L$ below $\mathsf{S4}$ or $\mathsf{GL}$, the system $L$-Frege is not p-bounded.
\end{corollary}
\begin{proof}
For the first part, we apply Theorem \ref{LowerBoundForIntCliquevsColor} and Theorem \ref{S4HasFI}.
For the second part, use Theorem~\ref{LKEquivFrege}.
\end{proof}

For modal logics of infinite branching, we again have the following extension:

\begin{theorem}[Jeřábek \cite{jevrabek2009substitution}]
For any modal logic $L$ of infinite branching, $L$-Frege is not p-bounded.
\end{theorem}

\bibliographystyle{plain}
\bibliography{Proof}

\end{document}